\documentclass[]{article}
\usepackage{mathptmx}

\usepackage[a4paper,margin=1in]{geometry}
\usepackage{setspace}
\makeatletter
\renewcommand\paragraph{\@startsection{paragraph}{4}{\z@}%
  {1.5ex \@plus 1ex \@minus .2ex}
  {-1em}%
  {\normalfont\normalsize\itshape}}
\makeatother

\usepackage{hyperref}

\usepackage{xfrac}
\usepackage{stmaryrd}
\usepackage{amsmath}
\usepackage{amsthm}
\usepackage{amsfonts}
\usepackage{mathrsfs}
\usepackage[american]{babel}
\usepackage{caption}

\usepackage{multirow}

\usepackage{subcaption}
\usepackage{graphicx}
\usepackage{xcolor}  
\usepackage{float}
\usepackage{capt-of}

\usepackage{mathtools}
\usepackage{amssymb}
\usepackage[capitalise,nameinlink]{cleveref}
\usepackage{enumitem}
\usepackage{bm}
\usepackage{bbm}
\usepackage{algorithm}
\usepackage{algpseudocode}
\usepackage{comment}

\newcommand{\var}{\mathrm{Var}}
\newcommand{\bias}{\mathrm{Bias}}
\newcommand{\esp}{\mathbb{E}}
\DeclareMathOperator{\spn}{span}
\DeclareMathOperator{\ex}{e}

\newtheorem{theorem}{Theorem}[section]
\newtheorem{definition}{Definition}[section]
\newtheorem{proposition}{Proposition}[section]

\newtheorem{lemma}{Lemma}[section] 
 
\newtheorem{remark}{Remark}[section]

\crefname{equation}{Eq.}{Eqs.}
\Crefname{equation}{Eq.}{Eqs.}

\crefname{figure}{Fig.}{Figs.}
\Crefname{figure}{Fig.}{Figs.}

\newcommand{\Fab}[1]{\textcolor{black}{#1}}
\newcommand{\Fabrice}[1]{\textcolor{black}{#1}}

\author{F. Deluzet$^{\dagger}$, C. Guillet$^{\ddagger,*}$, and J. Narski$^{\dagger}$}

\title{A hierarchical sparse-grid particle method for the Vlasov--Poisson system}
\begin{document}

\maketitle

\renewcommand{\thefootnote}{\fnsymbol{footnote}}

\footnotetext[1]{Corresponding author. Email: \texttt{clement.guillet@inria.fr}.}
\footnotetext[2]{Univ Toulouse, INSA Toulouse, CNRS, IMT, Toulouse, France.}
\footnotetext[3]{Inria, Concace joint team between Airbus CR\&T, Cerfacs and Inria, Talence, France.}

\begin{abstract}
We introduce a hierarchical sparse-grid (HSG) particle method for the
numerical solution of the Vlasov--Poisson system. Sparse-grid PIC methods
have so far been formulated within finite-difference frameworks, most
notably through the sparse-grid combination technique (SGCT), which ties
them to tensor-product Cartesian grids and globally defined component
grids. This paper brings sparse-grid particle methods into the Galerkin
setting: the field equation is solved in variational form on a
hierarchical sparse-grid space spanned by B-splines of arbitrary degree,
and the traditional charge deposition step is replaced by a direct
Galerkin projection of the raw Monte Carlo density estimator onto this
space. Beyond preserving the mesh-complexity and noise-reduction benefits
of sparse grids, this reformulation substantially extends the versatility
of the approach, opening the way to spatial adaptivity and to
non-rectangular geometries, which are notoriously difficult to
accommodate within the SGCT framework. We carry out a probabilistic error
analysis decomposing the numerical error into a grid-based bias and a
statistical noise component. Under mixed-derivative regularity assumptions
on the particle distribution, the bias of the charge density in the
$\mathrm{L}^2$-norm is shown to scale as
$\mathcal{O}(h^{p+1}|\log h|^{d-1})$, where $p$ is the B-spline degree and
$d$ the spatial dimension, and the statistical error in the
$\mathrm{L}^1$-norm as $\mathcal{O}(|\log h|^{(d-1)/2}(Nh)^{-1/2})$,
matching the accuracy of high-order SGCT-PIC methods. Corresponding bounds
are derived for the electric field. The theoretical estimates are
validated on classical kinetic plasma benchmarks, including configurations
with limited regularity and strong anisotropies that are known to be
challenging for sparse-PIC approximations.
\end{abstract}


%
% ----------------------------------------------------
% 1. Introduction
% ----------------------------------------------------
\section{Introduction}
Particle methods, such as the well-known particle-in-cell (PIC) algorithm~\cite{hockney88,birdsall18}, are among the most efficient and
widely used numerical methods for the simulation of kinetic plasmas. The PIC algorithm couples a Lagrangian discretization of the Vlasov equation,
based on the integration of particle trajectories, with a mesh-based discretization of Poisson's equation for the computation of the self-consistent electrostatic field. This particle representation, however, comes at a fundamental cost: sampling the distribution function
with a finite number of particles inherently introduces statistical noise into mesh-defined quantities. This statistical error decreases slowly, with a rate proportional to the inverse square root of the mean number of particles per cell. 
In many practical simulations, statistical noise therefore becomes the dominant source of error, and a meaningful noise reduction requires a prohibitively large increase in the number of computational particles. 
Over the years, numerous noise-reduction strategies have been investigated, including variance-reduction techniques such as the $\delta f$ method~\cite{aydemir94,denton95,sydora99}, filtering in either the Fourier~\cite{birdsall18} or the wavelet~\cite{gassama07} domains, and micro--macro decompositions~\cite{crestetto18}. More recently, sparse-grid techniques have emerged as a promising alternative.

Sparse-grid methods aim to alleviate the curse of dimensionality~\cite{bellman61}, namely the exponential growth of the
computational cost with respect to the problem dimension $d$. Initially introduced for function interpolation~\cite{smolyak63}, they were later
extended to the numerical solution of partial differential equations (PDEs) in the seminal papers~\cite{griebel90,bungartz91}, which put
forward two distinct frameworks. In~\cite{bungartz91}, a hierarchical basis representation is introduced within a multi-resolution
tensor-product structure, allowing the truncation of the least significant basis-function contributions. In~\cite{griebel90}, the sparse-grid combination technique (SGCT) constructs an approximation by linearly combining partial solutions computed on a hierarchy of coarse
component grids. Denoting by $h$ the typical mesh step, both approaches reduce the computational complexity from $\mathcal{O}(h^{-d})$, for full
tensor-product grids, to $\mathcal{O}(h^{-1}|\log_2 h|^{d-1})$, while maintaining comparable accuracy under suitable mixed-derivative
regularity assumptions. Sparse-grid methods have since been applied to a wide range of PDEs; see, e.g.,~\cite{griebel98-2,griebel07-1} for the
combination technique,~\cite{bungartz98,griebel07} for the hierarchical basis approach, and~\cite{bungartz04} for a comprehensive survey.

The combination of sparse grids with particle methods was initiated in~\cite{ricketson17}, where the SGCT was incorporated into a PIC scheme.
Compared to the standard PIC (STD-PIC) method, which relies on a single uniformly fine grid, the SGCT-PIC solution is reconstructed by recombining approximations computed on component grids with much coarser cells, which simultaneously reduces the mesh complexity and the statistical noise. The SGCT-PIC method has been assessed on classical kinetic plasma benchmarks~\cite{ricketson17,deluzet22} and successfully applied to low-temperature (including collisional) plasma discharges and to drift instabilities in Hall thrusters~\cite{garrigues21,garrigues21-1,garrigues24,garrigues24-1}.
Several variants~\cite{muralikrishnan21,deluzet22} and a high-order extension~\cite{deluzet25} have been proposed, together with a priori error estimates. Originally developed for explicit time discretizations, the method has also been adapted to semi-implicit formulations~\cite{guillet24,guillet25}, and optimized parallel implementations have been developed for three-dimensional simulations on shared-memory CPU architectures~\cite{deluzet22-1} and GPUs~\cite{deluzet23}.

The SGCT-PIC method and all of its variants share the same structure: the field equations are discretized by finite differences on a collection of
coarse component grids, and the electric field is subsequently reconstructed at the particle positions through the combination
technique. While effective for smooth solutions on regular domains, this construction inherits intrinsic limitations from the SGCT itself. First,
sparse-grid methods may perform poorly for non-smooth solutions or strongly localized gradients, for which the underlying mixed-derivative
regularity assumptions break down; spatial adaptivity then becomes essential. Although adaptive refinement is naturally accommodated by hierarchical basis representations of sparse grids, following, e.g.,~\cite{griebel99,pfluger10,bokanowski13}, it is considerably harder to integrate into the SGCT~\cite{obersteiner21,obersteiner21-1}, owing to its reliance on globally defined component grids. Second, the SGCT-PIC
method is tied to tensor-product Cartesian grids, which restricts its applicability to rectangular geometries; extensions to complex geometries
are non-trivial in this setting, whereas the hierarchical basis offers natural avenues through transfinite interpolation~\cite{Dornseifer:1996aa} or B-spline-based mappings~\cite{guillet25-1}. These limitations motivate the development of a particle method built on the hierarchical basis representation of sparse grids, preserving the noise- and complexity-reduction benefits of the SGCT while offering substantially more flexibility.

The contributions of this paper are twofold. First, we introduce a new explicit-in-time particle method, referred to as the hierarchical sparse-grid (HSG-)PIC method, which brings sparse-grid particle methods, so far confined to finite-difference formulations, into the Galerkin
setting. Conceptually distinct from both STD- and SGCT-PIC methods, the HSG-PIC method abandons traditional cell-based charge deposition.
Instead, the charge density estimator is redefined as a direct Galerkin projection of the raw Monte Carlo estimator onto a hierarchical
sparse-grid spline space, and the Poisson equation is solved in variational form on the same space. This variational structure is
precisely what unlocks the flexibility discussed above: hierarchical increments provide natural local error indicators for adaptivity, and
B-spline spaces extend to mapped, non-rectangular geometries.

Second, we develop a  probabilistic error analysis of the HSG-PIC method. Owing to the statistical nature of particle schemes, the analysis relies on a  decomposition of the approximation error into two components. The bias, or grid-based error, which characterizes the accuracy with which the mean values of grid quantities (particle density, electric potential, electric field) are approximated and is governed by the mesh resolution. The statistical noise, is the second component of the error, it measures the deviation from this mean. Assuming sufficient mixed-derivative regularity of the particle distribution, we
prove that the grid-based component of the charge density error in the $\mathrm{L}^2$ norm scales as $\mathcal{O}(h^{p+1}|\log_2 h|^{d-1})$,
where $p$ denotes the degree of the hierarchical B-spline basis and $d$ the spatial dimension, while the statistical component scales in the
$\mathrm{L}^1$ norm as $\mathcal{O}(|\log_2 h|^{(d-1)/2}(Nh)^{-1/2})$. The resulting accuracy is comparable to that of the high-order SGCT-PIC
method. The analysis is complemented by bounds on the electric field error: the grid-based error in the $\mathrm{L}^2$ norm scales as
$\mathcal{O}(h)$ for linear splines ($p=1$), while the statistical error exhibits the same scaling as that of the charge density. These estimates
are validated through numerical experiments, which further suggest that the bound on the statistical error of the electric field could be
sharpened.

The remainder of the paper is organized as follows. \Cref{sec:1} introduces the necessary background on kinetic formulations and particle methods: the Vlasov--Poisson system, the classical PIC algorithm, and the SGCT-PIC method, together with existing error estimates. In \cref{sec:2}, the HSG-PIC method is introduced and error estimates for both the grid-based and statistical error components are derived. Numerical experiments are presented in \cref{sec:3} to confirm the theoretical results and assess the performance of the method. Conclusions and perspectives are drawn in \cref{sec:4}.

\medskip
Throughout the paper, the symbol $\lesssim$ denotes an inequality holding
up to a positive multiplicative constant independent of the discretization
parameters, and we write $\log := \log_2$. \Fabrice{This work strictly addresses the spatial discretization of the model. Consequently, to streamline the exposition, the explicit time dependence of time-dependent quantities is omitted where possible, or simplified to the bare minimum.}

\section{Preliminaries: continuous model, standard and sparse grid combination technique (SGCT) particle-in-cell (PIC) methods}
\label{sec:1}
\subsection{Continuous model and standard PIC method}\label{sec:1.1}
In this section, we briefly introduce the necessary background on the mathematical formulation of kinetic plasma modelling and their numerical resolution using particle methods.
\subsubsection{Vlasov--Poisson system}
We consider the Vlasov--Poisson system, which describes the evolution of charged particles interacting through a self-consistent electrostatic field. Collisional processes are omitted in this work. A constant external magnetic field is included, but it is assumed to remain unaltered by the plasma dynamics. The system is formulated in phase space, defined as the Cartesian product of the spatial and velocity domains,
\[
\Omega:= \Omega_{\bm{x}} \times \Omega_{\bm{v}}, \quad \text{with}~~ \Omega_{\bm{x}}:=\mathbb{T}^d = [0,L]^d/\mathbb{Z}^d ~~, ~~ \Omega_{\bm{v}}:=\mathbb{R}^d,
\]
where $d\in \mathbb{N}$ denotes the dimension of the problem with respect to both space and velocity variables.
Throughout this work, the spatial domain $\Omega_{\bm{x}}$, of characteristic length $L>0$ and alternatively denoted by $\Omega := \Omega_{\bm{x}}$ in the following, is assumed to be periodic and identified with the $d$-dimensional torus $\mathbb{T}^d$.

The time domain is denoted $\Omega_{t}$ yielding the following notations:
\[
 \Omega_{\bm{x}\bm{v}t} := \Omega_{\bm{x}\bm{v}} \times \Omega_t\,, \quad   \Omega_{\bm{x}t} := \Omega_{\bm{x}} \times \Omega_t\,, \quad \Omega_{t}:=\mathbb{R}_+.
\]

To simplify the exposition, the ions are assumed to be at rest, forming a neutralizing background. All variables are expressed in dimensionless form, with characteristic scales being the Debye length $\lambda_D$ and the plasma period $\omega_p^{-1}$, given by
\[
	\lambda_D = \sqrt{\frac{\varepsilon_0 T_e}{q_e n_0}}, \qquad 
	\omega_p^{-1} =\left(\sqrt{\frac{q_e n_0}{m_e \varepsilon_0}}\right)^{-1}.
\]
In this normalization setting, the electron mass $m_e$, characteristic temperature $T_e$, charge $q_e$,  typical electron density $n_0$ and vacuum permittivity $\varepsilon_0$ are set to unity.
  Consequently, we only consider electrons through the dimensionless, collisionless Vlasov--Poisson system, which reads:
\[\begin{aligned}
\frac{\partial f}{\partial t}
+ \bm{v} \cdot \bm{\nabla}_{\bm{x}} f
- \left( \bm{v} \times \bm{B} + \bm{\nabla}_{\bm{x}} \Phi \right) \cdot \bm{\nabla}_{\bm{v}} f = 0,
\\
-\Delta \Phi = 1 - \rho, \qquad \rho = \int_{\Omega_{\bm{v}}} fd \bm{v}.
\end{aligned}
\]
where $\bm{\nabla}_*$ denotes the gradient operator with respect to the variable $*$, and we use the shorthand notation $\bm{\nabla} := \bm{\nabla}_{\bm{x}}$. The function $f : \Omega_{\bm{x}\bm{v}t} \to \mathbb{R}_+$ denotes the phase-space distribution function of the electron species, and $\rho: \Omega_{\bm{x}t}   \to \mathbb{R}_+$ is the electron charge density. 

%and satisfies the unit mass constraint
%\begin{align}
%\label{eq:3}
%\int_{\Omega} \rho(\bm{x},t)\, d\bm{x} = 1,
%\end{align}
%for all $t \in \Omega_t$.
\subsubsection{Overview of PIC methods}
\label{sec:1.1.2}
We briefly outline the main principles of the particle-in-cell (PIC) method. For a comprehensive presentation, we refer the reader to the classical books~\cite{hockney88,birdsall18} and the references therein.
\paragraph{Particle approximation of the distribution function}  The distribution function is approximated by a collection of $N$ macro-particles, that sample the phase space. Denoting $\delta$ as the Dirac delta distribution function, the particle approximation $f_N$ of the distribution function is defined as:
\[
f_N(\bm{x},\bm{v},t) = \sum_{s=1}^N w_s\delta\big(\bm{x}-\bm{x}_s(t)\big) \delta\big(\bm{v}-\bm{v}_s(t)\big)\,,
\]
where $w_s$ represents the particle weight, which accounts for the phase space volume carried by the $s$-th particle. Under the present dimensionless framework, we assume uniform particle weights $w_s = \int_{\Omega} \rho(x,0) /N$. The position and velocity of the $s$-th particle, denoted by $(\bm{x}_s(t), \bm{v}_s(t))$ for $s = 1,\ldots,N$, are time-dependent quantities such that $\bm{x}_s : \Omega_t \to \Omega$ and $\bm{v}_s : \Omega_t \to \Omega_{\bm{v}}$. 
\paragraph{Evolution of particles} The particles evolve according to Newton's equations of motion, namely
\begin{align}
\label{eq:1}
\frac{d  \bm{x}_s(t)}{d t} = \bm{v}_s(t), \qquad \frac{d  \bm{v}_s(t)}{d t} = \Big(\bm{v}_s(t)\times \bm{B}\big(\bm{x}_s(t)\big) +\bm{E}\big(\bm{x}_s(t),t\big)\Big),
\end{align}
where the fields $\bm{E}$ and $\bm{B}$ are evaluated at the particle position $\bm{x}_s(t)$.
These equations are typically discretized in time using explicit schemes such as the leapfrog method or Runge--Kutta integrators. At each time step, the electric field is computed on a spatial mesh from the charge density, which has to be computed first from the particle distribution. 
\paragraph{Charge density approximation onto the mesh} By definition, the macroscopic density reads $\rho(\bm{x},t) = \int_{\Omega_{\bm{v}}} f(\bm{x},\bm{v},t) \mathrm{d}\bm{v}$. To derive an estimator for the particle density, we integrate the distribution function $f_N$ over the velocity space.  Replacing $f$ with its discrete counterpart $f_N$ and exploiting the shifting property of the Dirac delta distribution with respect to $\bm{v}$, we obtain the raw Monte Carlo density estimator:
\begin{equation}\label{eq:rhoN}
\rho_{N}(\bm{x},t) = \sum_{s=1}^N w_s \delta\big(\bm{x}-\bm{x}_s(t)\big)\,.
\end{equation}
This estimator cannot be directly used on a discrete mesh due to its singular nature. To obtain a smoothed density field on the grid, the classical PIC approach consists in convolving $\rho_N$ with a localization kernel, also referred to as a shape function. These functions, denoted as $W_h$, are compactly supported on a scale proportional to the mesh size $h$, and satisfy the partition of unity property to ensure charge conservation. The regularized grid-based density estimator, $\rho_h(\bm{x},t)$, is then defined as:
\begin{equation}\label{eq:rhoN}
\rho_{h,N}(\bm{x},t) = \sum_{s=1}^N w_s W_h\big(\bm{x}-\bm{x}_s(t)\big)\,.
\end{equation}
Typical choices for $W_h$ \cite{hockney88,birdsall18} are tensor products of univariate functions, including piecewise constant functions (Nearest-Grid-Point), piecewise linear functions (Cloud-In-Cell), or higher-order \Fabrice{piecewise polynomials (quadratic B-splines for the Triangular-Shape-Cloud)}.

\paragraph{Resolution of Poisson equation}
The regularized Monte Carlo density estimator $\rho_h$ serves as the right-hand side of a Poisson equation, which is solved on the mesh $\Omega_h$ using a mesh-based method such as finite difference, fast Fourier transform or finite element methods. The electric potential solution is then differentiated on the grid to obtain the electric field.

\paragraph{Interpolation of the field at particle positions}
The electric field, computed on the grid, is interpolated at the particles' position. To maintain self-consistency, the classical PIC method employs the same shape function for both projecting the particle charges onto the grid and interpolating the electromagnetic fields back to the particle positions.

The sequence of these operations, namely the computation of the electric field on the grid, the interpolation at the particle positions and the evolution of particles, constitutes one iteration, or one time step, of the PIC algorithm. In the following, we refer to standard particle-in-cell (STD-PIC) algorithm this sequence of operations when a single full-grid with identical mesh resolution $h$ in all spatial directions is used.

\subsubsection{Stochastic error analysis framework}
\label{sec:2.3.1}

To rigorously quantify the numerical errors inherent to PIC methods, we adopt a probabilistic description of the particle system following~\cite{bottino15,deluzet22}. Within this framework, we introduce a time-dependent random variable $\bm{X}(t)$ representing the spatial position of a macro-particle at a given time $t$. The probability density function (PDF) associated with $\bm{X}(t)$, denoted by $p(\bm{x},t)$, is defined by normalizing the exact continuous macroscopic density $\rho(\bm{x},t)$ by the total mass of the system $\mathcal{M} := \int_{\Omega} \rho(\bm{\xi},t)\mathrm{d}\bm{\xi}$ yielding $ p(\bm{x},t) = {\rho(\bm{x},t)}/{\mathcal{M}}$.

In the PIC context, the numerical particle positions $\{\bm{x}_s(t)\}_{s=1}^N$ are modeled as $N$ independent and identically distributed (i.i.d.) random realizations drawn from the distribution of $\bm{X}(t)$ at each instant $t$. 
This probabilistic interpretation allows us to evaluate the accuracy of the density estimators by recalling the standard statistical metrics of bias and variance.

\begin{definition}
Let $u_*$ be a statistical estimator (such as the raw estimator $\rho_N$ or the regularized estimator $\rho_{h,N}$) of a deterministic quantity $u$. The \emph{bias} of the estimator, evaluated through the expectation operator $\mathbb{E}[\cdot]$ acting with respect to the underlying PDF $p(\bm{x},t)$, is defined by
\[
\text{Bias}[u_*] := \mathbb{E}[u_*] - u\,,
\]
and its \emph{variance} is defined by
\[
\text{Var}[u_*] := \mathbb{E}\left[(u_*)^2\right] - \big(\mathbb{E}[u_*]\big)^2.
\]
\end{definition}

These two metrics map directly to the two fundamental sources of numerical inaccuracy in PIC codes: the grid-based smoothing error (associated with the bias) and the stochastic noise (associated with the variance)~\cite{deluzet22}. Indeed, while the raw Monte Carlo estimator is fundamentally unbiased, the smoothing procedure used to define $\rho_{h,N}$ introduces a deterministic bias as a trade-off for successfully bounding the statistical noise to the rate $\mathcal{O}(1/\sqrt{Nh^d})$. These properties are summarized in the following lemma.
\begin{lemma}\label{lemma:Density:Estimator:Expected:Value} The raw and regularized Monte Carlo density estimators satisfy the following estimates
  \begin{align*}
  \mathbb{E}[\rho_N(\bm{x},t)]  &= \rho(\bm{x},t) \,, \\
  \mathbb{E}[\rho_{h,N}(\bm{x},t)]&=(W_h * \rho)(\bm{x},t)\,. 
  \end{align*}
\end{lemma}
This systematic error, $\text{Bias}[\rho_{h,N}(\bm{x},t) ]=(W_h * \rho)(\bm{x},t) - \rho(\bm{x},t)$, while $\text{Bias}[\rho_{N}(\bm{x},t) ]=0$ is a direct consequence of the convolution with the shape function $W_h$.

\subsection{SGCT-PIC method}
Recently, the SGCT approach has been incorporated into PIC algorithms for solving the Vlasov--Poisson~\cite{ricketson17} and Vlasov--Ampère~\cite{guillet24} equations, with the explicit goal of mitigating computational costs. The global time-stepping loop follows the same core steps as the standard PIC method. However, the single uniform mesh is replaced by a hierarchy of anisotropic component grids, onto which approximations of the density are deposited and the corresponding electric fields are computed. The electric field approximations obtained across these distinct component grids are then combined to define a unique discrete electric field. This recombined field is subsequently interpolated at the macro-particle positions to integrate the equations of motion. A comprehensive outline of these methods, along with their associated error estimates, is detailed in the following sections.
\subsubsection{Outline}
\label{sec:1.2.1}
We briefly present the classical SGCT-PIC method here and refer to~\cite{deluzet22,deluzet25} for further details and a discussion of its variants.

\paragraph{Component grids}
The cornerstone of the method is the definition of a collection of component grids, each associated with a multi-index level $\bm{\ell} \in \mathbb{N}^d$ that parametrizes the spatial resolution along each coordinate direction. \Fabrice{For a given level $\bm{\ell}$, the multi-index node set $I_{h_{\bm{\ell}}}$ is defined as the Cartesian product of the univariate index sets:
\begin{equation}
  I_{h_{\bm{\ell}}} := \prod_{i=1}^d I_{h_{\ell_i}} \quad \text{with} \quad I_{h_{\ell_i}}:=\Big\{0,\ldots,2^{\ell_i}-1\Big\},
\end{equation}
where the directional mesh size is given by $h_{\ell_i}=2^{-\ell_i}$. The corresponding discrete grid $\Omega_{h_{\bm{\ell}}}$ covering the domain $\Omega$ is then defined by
\begin{equation}
\Omega_{h_{\bm{\ell}}}:= \Big\{ \bm{x}_{\bm{j}}\in \Omega \;\Big|\; \bm{x}_{\bm{j}} = (j_1 h_{\ell_1}, \dots, j_d h_{\ell_d}), \;\; \bm{j} \in I_{h_{\bm{\ell}}} \Big\}.
\end{equation}}

The component grids are partitioned into \Fabrice{$\mathcal{D}$} distinct family levels, denoted as $l$ with $l \in \{0, \dots,  \Fabrice{\mathcal{D}}-1\}$. The number of family levels equals the dimension of the problem $\mathcal{D}=d$. Let $h$ denote the mesh size parameter corresponding to the finest spatial resolution in each direction, and defined from an integer $n \in \mathbb{N}_0$ such that $h = 2^{-n}$. The full set of admissible levels $\mathscr{L}_h$ is decomposed into a disjoint union of subsets $\mathscr{L}_{h,l}$ related to the family level $l$:
\begin{align}
\label{eq:8}
\mathscr{L}_h := \bigcup_{0 \leq l \leq  \Fabrice{\mathcal{D}}-1} \mathscr{L}_{h,l}, \quad \text{with }
\mathscr{L}_{h,l} := \left\{ \bm{\ell} \in \mathbb{N}^d \;\middle|\; 
|\bm{\ell}|_1 = n + (d-1) - l, \ \ell_j \geq 1 \right\},
\end{align}
where the $\ell^1$-norm of the multi-index, $|\bm{\ell}|_1 := \sum_{i=1}^d \ell_i$, defines the total resolution level of a grid. 

A combination coefficient is assigned to any of the component grids in a family level $l$. Consequently, all component grids belonging to the same family share the same weight, defined as a function of the level $l$ by:
\begin{equation}\label{eq:def:comb:coeff}
      c_{\bm{\ell}} = c_{l} := (-1)^l \binom{d-1}{l}, \quad \text{for all } \bm{\ell} \in \mathscr{L}_{h,l}.
\end{equation}

\paragraph{Charge deposition and field equations}
For each component grid, an approximation of the electric field $E_{h,N}$ is computed by performing the two following operations successively:
\begin{enumerate}
  \item The charge density is projected onto any of the component grids $\Omega_{\Fabrice{h_{\bm{\ell}}}}$ ($\bm{\ell} \in \mathscr{L}_{h,l}$) using shape functions $W^p_{h_{\bm{\ell}}}$ constructed as tensor products of one-dimensional shape functions $\mathcal{W}_{h_{\ell_i}}^p$: 
 \[
\rho_{\Fabrice{h_{\bm{\ell}},N}}(\bm{x},t) = \sum_{s=1}^N w_s W^p_{h_{\bm{\ell}}}\big(\bm{x}-\bm{x}_s(t)\big), \quad \text{where}~~ W^p_{h_{\bm{\ell}}}(\bm{x}) = \prod_{i=1}^{d} \mathcal{W}^p_{h_{\ell_i}}(x_i).
  \] 
  The symmetry of the functions $\mathcal{W}_{h_{\ell_i}}^p$ permits the definition of a density estimator with a bias proportional to $h_{\ell_i}^2$. These functions may be specified to increase this precision to $h_{\ell_i}^p$ \cite{deluzet25} and are therefore indexed by $p$ to state this order of approximation.
  \item The density estimator $\rho_{\Fabrice{h_{\bm{\ell}},N}}$ is evaluated at the component grid nodes at time $t^\kappa$, defining the nodal vector \Fabrice{$\hat{\rho}^\kappa_{h_{\bm{\ell}},N}$ whose components are given by
  \begin{equation}
    \left(\hat{\rho}^\kappa_{h_{\bm{\ell}},N}\right)_{\bm{j}} = {\rho}_{h_{\bm{\ell}},N}(\bm{x}_{\bm{j}},t^\kappa) \,, \qquad  \forall \bm{j} \in I_{h_{\bm{\ell}}}.
  \end{equation}
  The Poisson equation is discretized on the component grids with $\hat{\rho}^\kappa_{h_{\bm{\ell}},N}$ as the source term. Solving this system yields the approximation of the electric potential, from which the nodal values of the approximated electric field, denoted $\hat{\bm{E}}^\kappa_{h_{\bm{\ell}},N}$ are derived. Finally, a continuous electric field interpolant ${\bm{E}}^\kappa_{h_{\bm{\ell}},N}$  is reconstructed for any level $\bm{\ell} \in \mathscr{L}_{h,l}$ by combining these nodal values with the shape functions ${W}_{h_{\bm{\ell}}}^p$:
  \begin{equation}
    {\bm{E}}^\kappa_{h_{\bm{\ell}},N}(x) := \sum_{\bm{j} \in I_{h_{\bm{\ell}}}}\left(\hat{\bm{E}}^\kappa_{h_{\bm{\ell}},N}\right)_{\bm{j}} W^p_{h_{\bm{\ell}}}\big(\bm{x}-\bm{x}_{\bm{j}}\big)\,.
  \end{equation}}
\end{enumerate}
 
\paragraph{Combination}
An accurate approximation of the electric field is obtained by linearly recombining the component fields $\bm{E}^{\Fabrice{\kappa}}_{h_{\bm{\ell}},N}$ using the combination coefficients $c_{\bm{\ell}}$ defined in \cref{eq:def:comb:coeff}:
\Fabrice{\begin{equation}\label{eq:recombination}
    \bm{E}^\kappa_{h,N}(\bm{x})
    := \sum_{\bm{\ell}\in \mathscr{L}_h}
    c_{\bm{\ell}}
    \bm{E}^\kappa_{h_{\bm{\ell}},N}(\bm{x}).
\end{equation}}
This recombined continuous field is subsequently evaluated at each particle's position \Fabrice{$\bm{x}_s(t^\kappa)$} to update the macro-particle velocities.

\subsubsection{Error estimates}
Grid-based and statistical error estimates for the SGCT-PIC method were first derived in~\cite{deluzet22} for linear shape functions and later extended to high-order shape functions in~\cite{deluzet25}. The estimates rely on Taylor expansions and therefore require sufficient regularity of the exact solution. More precisely, the analysis is carried out in the mixed-regularity space
\[
\mathrm{X}^q_{\mathrm{mix}}(\Omega):= \left\{f: \Omega \rightarrow \mathbb{R} \;\middle|\; D^{\bm{\alpha}}f \in \mathcal{C}^0(\Omega), ~~ |\bm{\alpha}|_\infty\leq q \right\}, \quad ~\text{with}~  \quad |\bm{\alpha}|_\infty := \max_{1 \le i \le d} \alpha_i,
\]
where $\mathcal{C}^0(\Omega)$ denotes the space of continuous functions on $\Omega$.

Let $\rho_{h,N}$ be the recombined charge density approximation obtained from ~\cref{eq:recombination} using the $p$-th order SGCT-PIC method. The following proposition summarizes the results of~\cite[Proposition 3.6]{deluzet22} and~\cite[Theorem 3.2]{deluzet25}:
\begin{proposition}
\label{prop:1}
Let $f \in \mathrm{L}^1\!\left(\mathrm{X}^{q}_{\mathrm{mix}}(\Omega); \Omega_{\bm{v}}\right)$, with $0 \leq p \leq q-1$. Then the following estimates hold:
\begin{align*}
\| \bias[\rho_{h,N}] \|_{\mathrm{L}^\infty(\Omega)}
&\approx
h^{q} |\log h|^{d-1}
\| D^{\bm{\alpha}} \rho \|_{\mathrm{L}^\infty(\Omega)}, \quad \text{where}~ \bm{\alpha}=(q,\ldots,q),\\
\big\| \var[\rho_{h,N}]^{\frac{1}{2}} \big\|_{\mathrm{L}^\infty(\Omega)}
&\approx
|\log h|^{d-1}\left( \frac{1}{N h} \right)^{\frac{1}{2}}
\| \rho \|_{\mathrm{L}^\infty(\Omega)}.
\end{align*}
\end{proposition}
Similar estimates are obtained for the standard PIC method, namely:
\[
\| \bias[\rho_{h,N}] \|_{\mathrm{L}^\infty(\Omega)} = \mathcal{O}(h^{p+1}), \quad \big\| \var[\rho_{h,N}]^{\frac{1}{2}} \big\|_{\mathrm{L}^\infty(\Omega)}=\mathcal{O}((Nh^{-d})^{\frac{1}{2}}).
\]
For high-dimensional problems (where $d>2$), the SGCT-PIC approximation provides a substantial reduction in computational cost compared to standard PIC methods. This is achieved through a better mitigation of the statistical noise, which depends only very weakly on the dimensionality of the problem for the SGCT-PIC method ($|\log h|^{d-1}( {N h})^{-1/2}$), in contrast to standard PIC methods ($(N\Fabrice{h^d)^{-1/{2}}}$).

\medskip
\noindent
Nonetheless, some limitations of the method can be identified from these error estimates: i) the error analysis relies heavily on mixed-derivative regularity, a requirement that restricts the method's accuracy for solutions exhibiting limited smoothness or sharp, highly anisotropic gradients; ii) sparse-grid methods usually benefit from spatially adaptive mesh refinement when such regularity assumptions are violated, which is difficult to accomodate within the SGCT framework. 

These observations motivate the development of an alternative method that relaxes the regularity assumptions while retaining accuracy and efficiency, naturally calling for the introduction of spatially adaptive mesh refinement strategies. Such an approach is introduced and analyzed in the following section.

\section{Hierarchical sparse-grid (HSG) particle method}
\label{sec:2}
We introduce the hierarchical sparse-grid (HSG) particle method, which relies on a hierarchical variational formulation of the field equations as well as a reconceptualized density estimator tailored to the mathematical structure of hierarchical sparse grids. Specifically, the HSG approach is characterized by two core components that set it apart from traditional STD- and SGCT-PIC algorithms:  i) a Galerkin projection of the raw Monte Carlo estimator of the charge density, as opposed to regularization performed through convolution with shape functions; ii) the solution of a variational formulation of the Poisson equation using a Galerkin method. 

The key ingredient relies on the construction of a suitable finite-dimensional approximation space. Various choices for this approximation space, denoted as $V_h(\Omega)$, are presented in \cref{sec:2.2}. The density estimator along with the variational formulation of the field equations are then introduced in \cref{sec:2.3}. Finally, theoretical estimates of the bias and variance associated with the HSG-PIC method are stated in \cref{sec:2.4}.

\paragraph{Notations} We denote by $\mathrm{H}^{q}(\Omega)$ the isotropic $\mathrm{L}^2$ Sobolev space of regularity $q \in \mathbb{N}_0$. We introduce the Sobolev space of dominating mixed smoothness defined by
\[
\mathrm{H}_{\mathrm{mix}}^q(\Omega) := \left\{f: \Omega \rightarrow \mathbb{R} \;\middle|\; D^{\bm{\alpha}}f \in \mathrm{L}^2(\Omega), ~~ |\bm{\alpha}|_\infty\leq q \right\}.
\]
In the above definition, the derivative operator defined for multi-indices $\bm{\alpha} = (\alpha_1,\ldots,\alpha_d) \in \mathbb{N}^d$ is given by
\[
D^{\bm{\alpha}} f := \frac{\partial^{|\bm{\alpha}|_1} f}{\partial x_1^{\alpha_1} \cdots \partial x_d^{\alpha_d}} \qquad  \text{ where}~~ |\bm{\alpha}|_1 = \sum_{i=1}^d \alpha_i \, ~~ \text{and}~~ |\bm{\alpha}|_\infty = \max_{1\leq i \leq d} \alpha_i, ~~ \text{with}~~  D^{\bm{0}} f := f \,.
\]  
The corresponding semi-norm and norm are defined by  
\[
|f|_{\mathrm{H}_{\mathrm{mix}}^q(\Omega)} := \Bigg( \sum_{|\bm{\alpha}|_\infty = q} \| D^{\bm{\alpha}} f \|_{\mathrm{L}^2(\Omega)}^2 \Bigg)^{1/2}, 
\quad
\|f\|_{\mathrm{H}_{\mathrm{mix}}^q(\Omega)} := \Bigg( \sum_{m=0}^q |f|_{\mathrm{H}_{\mathrm{mix}}^m(\Omega)}^2 \Bigg)^{1/2}.
\]
%
%Given a spatial mesh $\Omega_h$ of the domain $\Omega$ and an associated finite-dimensional approximation subspace $V_h(\Omega)$, which will be introduced in the subsequent sections, one iteration in time of the HSG-PIC algorithm consists of the following steps:
%\begin{enumerate}[label=\roman*)]
%\item The charge density is approximated by a \cor{Galerkin} projection of the Monte-Carlo estimator $\rho_N$ onto the subspace $V_h(\Omega)$. 
%\item A variational formulation of the Poisson equation, where the source term is given by the charge density approximation computed at step i), is discretized by a Galerkin method considering the subspace $V_{h}(\Omega)$.
%\item The electric field, which is the gradient of the electric potential computed at step ii), is evaluated at the particle positions. 
%\item The particles are advanced in time according to \cref{eq:1}.
%\end{enumerate}

\subsection{Hierarchical spline approximation spaces}
\label{sec:2.2}
The central ingredient of the proposed method is the construction of a suitable finite dimensional subspace $V_h(\Omega)$, which shall approximate the Hilbert space $\mathrm{H}^q_{\mathrm{mix}}(\Omega)$, with $q > {1/2}$\footnote{This condition is necessary for the Galerkin projection of the Monte Carlo estimator, defined as a sum of delta Dirac distributions $\delta\in\mathrm{H}^{-j}_{\mathrm{mix}}(\Omega)$ for $j>1/2$, to be well-defined.}. 

\subsubsection{Space construction via truncated tensor products}
\label{sec:2.2.1}
The approximation space is constructed from the following methodology: i) a univariate space of high degree $p\in \mathbb{N}$ and regularity $\mathcal{C}^{p-1}$ B-spline functions is defined; ii) it is extended by tensor products to $d$-dimensional spaces, with a distinct mesh resolution in each dimension; iii) a hierarchical basis formulation is adopted, allowing the spaces to be expressed as a direct sum of hierarchical subspaces of increasingly fine discretizations; iv) a sparse-grid-based rule is applied to select the relevant hierarchical subspaces, thereby defining an approximation space of significantly reduced dimension compared to a full grid, while preserving its essential approximation properties. Each of these components is detailed in the following.

\paragraph{Univariate B-splines approximation space}
We recall the notation $h$ of the dyadic mesh size, i.e., there exists an integer $n \in \mathbb{N}_0$ such that $h = 2^{-n}$. We define the space of univariate B-spline functions of degree $p \in \mathbb{N}_0$, with maximal regularity in $\mathrm{H}^p([0,1])$ and support of size $(p+1)h$, by
\begin{align*}
S_h([0,1])
&:= \left\{ v_h \in \mathrm{H}^p([0,1]) \;\middle|\;
(v_h)_{|_{(h j,\, h (j+1))}} \in \mathbb{P}^p, \quad \forall j \in I_h \right\}\\
&= \spn \left\{ \varphi_{h,j}^p, \; j \in I_h \right\},
\end{align*}
where $I_h := \{0, \ldots, h^{-1} - 1\}$.
In this definition, $\mathbb{P}^p$ denotes the space of univariate polynomials of degree less than or equal to $p$.
By standard trace arguments, the functions in $S_h([0,1])$ belong to $\mathcal{C}^{p-1}([0,1])$, in agreement with the classical characterization of B-spline spaces~\cite{schumaker07}. For simplicity, we omit the dependence on the spline degree $p$ in the notation of the space when no ambiguity arises. The second equality explicitly introduces the B-spline basis functions. We emphasize that these functions differ from the univariate shape functions commonly used in STD- and SGCT-PIC methods: in particular, B-spline basis functions belong to $\mathcal{C}^{p-1}([0,1])$, whereas standard shape functions are typically only $\mathcal{C}^0([0,1])$.
The space $S_h([0,1])$ serves as the fundamental building block for the construction of the approximation space. The latter is obtained by forming a truncated tensor-product space, referred to as a sparse-grid space, of the univariate spaces in each spatial direction. Such constructions are known to yield quasi-optimal approximation properties for Sobolev functions with mixed regularity; see, for example,~\cite{pinkus85}.

\paragraph{Anisotropic $d$-dimensional approximation space} We extend the uni-dimensional approximation space to higher dimension and introduce the so-called hierarchical subspaces, also known as hierarchical increments, which underpin the construction of sparse-grid approximation spaces. A more detail description of sparse grids can be found in the monograph~\cite{bungartz04} and the references therein. 

We begin by defining anisotropic $d$-dimensional B-spline spaces through full tensor products of the univariate spaces. Let $\bm{\ell} = (\ell_1, \ldots, \ell_d) \in \mathbb{N}^d$ denote the mesh resolution level in each spatial direction. The corresponding tensor-product space is defined as
\[
S_{h_{\bm{\ell}}}(\Omega)
:= \bigotimes_{i=1}^d S_{h_{\ell_i}}([0,1]).
\]
Here, the polynomial degree $p$ is chosen identical in all directions, while the mesh sizes $h_{\ell_i}$ may vary anisotropically. 
For later use, we introduce a basis of this space, obtained as the tensor products of the univariate B-spline basis functions:
\begin{align}
\label{eq:2}
S_{h_{\bm{\ell}}}(\Omega) = \text{span}\left\{ \varphi_{h_{\bm{\ell}},\bm{j}}^p \; \middle| \;  \varphi_{h_{\bm{\ell}},\bm{j}}^p := \bigotimes_{i=1}^d \varphi_{h_{\ell_i},j_i}^p, \quad j_i\in I_{h_{\ell_i}}, \quad  1\leq i \leq d \right\},
\end{align}
where $I_{h_{\ell_i}}=\{0,\ldots,h_{\ell_i}^{-1}-1\}$.
The space $S_{h_{\bm{\ell}}}(\Omega)$ is defined as the tensor product of one-dimensional spaces that are subsets of $\mathrm{H}^p([0,1])$, so that it is a subset of the mixed Sobolev space
\[
S_{h_{\bm{\ell}}}(\Omega) \subset \mathrm{H}^p_{\mathrm{mix}}(\Omega).
\]
According to~\cite[Theorem 8.1]{takacs16},~\cite[Corollary 3.1]{sande19} or~\cite[Proposition 3.1]{guillet25-1}, this finite-dimensional space exhibits optimal approximation properties (in the sense of the maximum theoretical convergence rate of $\mathcal{O}(h^{p+1})$ with respect to the mesh size $h$) for functions in the $\mathrm{H}^{p+1}(\Omega)$ Sobolev space measured in the $\mathrm{L}^2$-norm.
%\begin{proposition}
%\label{prop:4}
%For all $u\in\mathrm{H}^{p+1}(\Omega)$, for all $\varepsilon>0$, there exists $\bm{\ell}\in \mathbb{N}$, such that 
%\[
%\min_{\varphi_h\in S_{h_{\bm{\ell}}}(\Omega)} \|u-\varphi_h\|_{\mathrm{L}^2(\Omega)} \leq \varepsilon.
%\]
%\end{proposition}
%\begin{proof}[Proof of~\cref{prop:4}]
%The result follows from~\cite{takacs16}, theorem 8 and using similar arguments than the proofs of theorems 5 and 7.
%\end{proof}

%\Fabrice{We introduce the Galerkin projection operator onto this finite-dimensional space, denoted by $\Pi_{h_{\bm{\ell}}}$. For any distribution $u\in \mathrm{H}^{-p}_{\mathrm{mix}}(\Omega)$, its projection  $\Pi_{h_{\bm{\ell}}}\, u$ is uniquely defined by:
%\[
%\big(\Pi_{h_{\bm{\ell}}}\, u , v_h\big)_{\mathrm{L}^2} = \big\langle u, v_h \big\rangle_{\mathrm{H}^{-p}_{\mathrm{mix}}, \, \mathrm{H}^p_{\mathrm{mix}}}, \quad \forall v_h \in S_{h_{\bm{\ell}}}(\Omega),
%\]
%where $\langle \cdot, \cdot \rangle_{\mathrm{H}^{-p}_{\mathrm{mix}}, \, \mathrm{H}^p_{\mathrm{mix}}}$ denotes the duality pairing between $\mathrm{H}^{-p}_{\mathrm{mix}}(\Omega)$ and its dual $\mathrm{H}^p_{\mathrm{mix}}(\Omega)$. Note that since the Dirac delta distribution $\delta$ belongs to $\mathrm{H}^{-1}_{\mathrm{mix}}(\Omega)$, the Galerkin projection of the raw Monte Carlo estimator (introduced in \cref{eq:9}) is well-defined for linear and higher-degree splines $ \varphi_{h_{\bm{\ell}},\bm{j}}^p$ , i.e., whenever $p \geq 1$.}

\paragraph{Hierarchical subspace decomposition}
To construct the sparse-grid subspace, we exploit the hierarchical structure of the tensor-product spaces and express them via the direct-sum decomposition
\begin{align}
\label{eq:4}
S_{h_{\bm{\ell}}}(\Omega)
= \bigoplus_{\bm{k} \leq \bm{\ell}} W_{h_{\bm{k}}}(\Omega), \qquad \text{ where }
 \bm{k} \leq \bm{\ell} \Longleftrightarrow k_j \leq \ell_j, \quad  1\leq j \leq d .\end{align}
%where $\bm{k} \leq \bm{\ell}$ denotes component-wise inequalities, i.e., $k_j \leq \ell_j$ for all $j = 1, \ldots, d$.

A key property underlying the decomposition in \eqref{eq:4} is the nestedness of B-spline spaces: when a mesh is refined via knot insertion, the coarser space remains entirely contained within the finer one. Formally, letting $\bm{e}_j$ denote the unit vector in the $j$-th coordinate direction, we have:
\[
S_{h_{\bm{\ell} - \bm{e}_j}}(\Omega) \subset S_{h_{\bm{\ell}}}(\Omega), \quad \forall j = 1, \ldots, d.
\]
This property ensures that the new basis functions introduced at each refinement level are linearly independent of those existing at coarser levels, enabling the space to be decomposed into a direct sum of hierarchical increments.

The hierarchical subspaces, or increments, are defined as the following complements:
\begin{align}
\label{eq:14}
W_{h_{\bm{k}}}(\Omega)
:= S_{h_{\bm{k}}}(\Omega)
\bigg/ \bigoplus_{j=1}^d S_{h_{\bm{k} - \bm{e}_j}}(\Omega),% \qquad \text{ with } S_{h_{\bm{k}}}(\Omega) := {0} \text{ if } \exists k_j = 0 \text{ for any $j \in \{1, \ldots, d\}$}
\end{align}
with the convention that $S_{h_{\bm{k}}}(\Omega) := {0}$ if $k_j = 0$ for any $j \in \{1, \ldots, d\}$. Intuitively, each hierarchical increment $W_{h_{\bm{k}}}(\Omega)$ captures the additional information, specifically the new basis functions, required to pass from a coarse resolution to the finer resolution $\bm{k}$. The coefficients associated with these functions are often referred to as the hierarchical surplus. They represent the local correction (or surplus) that must be added to the coarse approximation to resolve finer scales without duplicating information already captured at coarser levels.

\paragraph{Truncations of hierarchical space}
The final step in constructing the approximation space consists of defining a truncation rule to select the appropriate hierarchical levels. To this end, we introduce an index set $ \mathscr{L}_h^{(\star)}$ such that the approximation space is defined as:
\[
V_{h}^{(\star)}(\Omega)
:= \bigoplus_{\bm{\ell} \in \mathscr{L}_h^{(\star)}} W_{h_{\bm{\ell}}}(\Omega).
\]
Sparse-grid methods rely on truncating the hierarchical subspaces by discarding the a priori less significant contributions.  Usually, the solution is assumed to have sufficiently regular mixed derivatives, so that a good approximation can be constructed by only taking into account the subspaces satisfying the criterion $|\bm{\ell}|_1\leq |\log h|$. In comparison to full-grid approximation spaces, constructed using the criterion $|\bm{\ell}|_\infty\leq |\log h|$, the sparse-grid space contains significantly less nodes and thus reduce the computational costs. We introduce in the following section different choices of truncation rules and discuss the resulting approximation space and their respective properties.

\subsubsection{Complexity and approximation properties}
\paragraph{Definition}
We introduce the different approximation spaces with a generic notation $V_h^{(\star)}(\Omega)$, where $(\star)$ refers to a specific definition of the index set $ \mathscr{L}_h^{(\star)}$:
\begin{definition}
\label{def:1}
The approximation spaces are defined by
\[
 V_{h}^{(\infty)}(\Omega)
:=\bigoplus_{\bm{\ell} \in \mathscr{L}_h^{(\infty)}} W_{h_{\bm{\ell}}}(\Omega), \qquad V_{h}^{(1)}(\Omega)
:=\bigoplus_{\bm{\ell} \in \mathscr{L}_h^{(1)}} W_{h_{\bm{\ell}}}(\Omega), 
\]
where the sets of hierarchical indices are given by
\begin{align}
&\mathscr{L}_{h}^{(\infty)}
:= \left\{
\bm{\ell} \in \mathbb{N}^d
\;\middle|\;
|\bm{\ell}|_\infty \leq |\log h| ,
\quad \ell_j \geq 0
\right\}, \\ 
&\mathscr{L}_{h}^{(1)}
:= \left\{
\bm{\ell} \in \mathbb{N}^d
\;\middle|\;
|\bm{\ell}|_1 \leq |\log h| ,
\quad \ell_j \geq 0
\right\}.
\end{align}
\end{definition}
These two constructions are referred to as the full-grid approximation space and the ($\mathrm{L}^2$-based) sparse-grid space. Following~\cite[Theorem 4.1]{guillet25-1}, an alternative characterization of the $\mathrm{L}^2$ approximation space, namely $V_{h}^{(1)}(\Omega)$, can be established based on the sparse-grid combination technique.
\begin{lemma}
\label{prop:3}
The approximation space $V_h^{(1)}(\Omega)$ can equivalently be expressed as
\begin{align}
\label{eq:9}
V_h^{(1)}(\Omega)
:= \left\{
v_h = \sum_{\bm{\ell} \in \mathscr{L}_h} c_{\bm{\ell}} v_{h_{\bm{\ell}}},
\quad \text{where } v_{h_{\bm{\ell}}} \in S_{h_{\bm{\ell}}}(\Omega)
\right\},
\end{align}
where the admissible index set is defined by \cref{eq:8}.
\end{lemma}

 \paragraph{Complexity}
 As a direct consequence of the hierarchical basis representation of~\cref{def:1}, the sparse-grid space $V_{h}^{(1)}(\Omega)$ is included in the full-grid space $V_{h}^{(\infty)}(\Omega)$, that is 
\[ V_{h}^{(1)}(\Omega)\subset V_{h}^{(\infty)}(\Omega).
\]
 The following proposition is a result from~\cite[Proposition 5.1, Equation (19)]{guillet25-1} and gives an estimate of the complexity of the approximation spaces.
 \begin{proposition}
 \label{prop:4}
 The dimension of the different approximation spaces, i.e., the number of mesh nodes, is given by
 \begin{align*}
  \left|V_{h}^{(1)}(\Omega)\right|= \mathcal{O}(|\log h|^{d-1}h^{-1}),  \quad \left|V_{h}^{(\infty)}(\Omega)\right|= \mathcal{O}(h^{-d}).
 \end{align*}
 \end{proposition}
 
 \paragraph{Approximation properties} 
\Fabrice{In order to regularize the discrete particle distribution, the raw Monte Carlo density estimator $\rho_N$ must be appropriately mapped onto the finite-dimensional space $V_h^{(\star)}(\Omega)$. This smoothing process is achieved by means of a generalized $\mathrm{L}^2$-Galerkin projection operator. We first formalize the mathematical existence and uniqueness of this operator for general distributions in \cref{prop:galerkin_projection}. Then, focusing on the case of sufficiently smooth target functions, we recall in \cref{lem:7} the structural approximation errors and optimal convergence rates associated with this projection.}

\Fabrice{\begin{proposition}[$\mathrm{L}^2$-Galerkin projection onto $V_{h}^{(\star)}(\Omega)$] \label{prop:galerkin_projection}
Let $p \geq 1$ and $u \in \mathrm{H}^{-p}_{\mathrm{mix}}(\Omega)$. There exists a unique element $\Pi_{h}^{(\star)}u \in V_{h}^{(\star)}(\Omega)$ satisfying
\begin{equation}
\label{eq:def:Galerkin:projection}
\big(\Pi_{h}^{(\star)} u,\, v_h\big)_{\mathrm{L}^2(\Omega)}
= \langle u, v_h \rangle_{\mathrm{H}^{-p}_{\mathrm{mix}},\,
\mathrm{H}^p_{\mathrm{mix}}},
\qquad \forall v_h \in V_{h}^{(\star)}(\Omega),
\end{equation}
where $\langle \cdot, \cdot \rangle_{\mathrm{H}^{-p}_{\mathrm{mix}},\,
\mathrm{H}^p_{\mathrm{mix}}}$ denotes the duality pairing between
$\mathrm{H}^{-p}_{\mathrm{mix}}(\Omega)$ and its dual $\mathrm{H}^p_{\mathrm{mix}}
(\Omega)$. The operator $\Pi_{h}^{(\star)} : \mathrm{H}^{-p}_{\mathrm{mix}}
(\Omega) \to V_{h}^{(\star)}(\Omega)$ is the $\mathrm{L}^2$-Galerkin
projection onto $V_{h}^{(\star)}(\Omega)$.
\end{proposition}}

\begin{proof}[Proof of~\cref{prop:galerkin_projection}]
See \cref{apd:0} of the appendix.
\end{proof}

 The following lemma recalls the approximation error estimates. 
  \begin{lemma}[Approximation properties]
 \label{lem:7}The following estimates hold:
 \begin{enumerate}[label=\roman*)]
 \item \label{lem:7.1} For all $u\in \mathrm{H}^{q}(\Omega)$, with $0\leq q\leq p+1$, then
 \begin{align}
  \| (\operatorname{I} - \Pi_{h}^{(\infty)})u\|_{\mathrm{L}^2(\Omega)} \lesssim h^{q} \|u\|_{\mathrm{H}^{q}(\Omega)}.
 \end{align}
  \item \label{lem:7.2} For all $u\in \mathrm{H}_{\mathrm{mix}}^{q}(\Omega)$, with $0\leq q\leq p+1$, then
 \begin{align}
   \| (\operatorname{I} - \Pi_{h}^{(1)})u\|_{\mathrm{L}^2(\Omega)} \lesssim h^{q} |\log h|^{d-1}\|u\|_{\mathrm{H}_{\mathrm{mix}}^{q}(\Omega)}.
    \end{align}
     \item \label{lem:7.3}For all $u\in \mathrm{H}_{\mathrm{mix}}^{2}(\Omega)$ and $p=1$, then
 \begin{align}
 \min_{v_h\in V_h^{(\star)}(\Omega)} \| u - v_h\|_{\mathrm{H}^1(\Omega)}\lesssim h \|u\|_{\mathrm{H}_{\mathrm{mix}}^{2}(\Omega)}.
 \end{align}
 \end{enumerate}
 \end{lemma}
 \begin{proof}[Proof of~\cref{lem:7}] 
 \cref{lem:7.1} results from~\cite[Proposition 3.1]{guillet25-1},  \cref{lem:7.2} results from~\cite[Lemma 4.4]{guillet25-1}, and \cref{lem:7.3} results from~\cite[Lemma 3.5, Theorem 3.8]{bungartz04}.
 \end{proof}

% 
%In the following, we use the notation $V_{h}^{(\star)}(\Omega)$ to refer to either the full-grid space, the $\mathrm{L}^2$-norm-based sparse-grid space or the energy-based sparse-grid space.
%
%The $\mathrm{L}^2$-orthogonal projection of a distribution $u$ onto $V_{h}^{(\star)}(\Omega)$ is denoted by $\Pi_{h}^{(\star)} u$ and satisfies
%\[
%(\Pi_{h}^{(\star)} u - u, u_h)_{\mathrm{L}^2(\Omega)} = 0,
%\quad \forall u_h \in V_{h}^{(\star)}(\Omega).
%\]
%The $\mathrm{L}^2$-orthogonal projection onto the sparse-grid space is well-defined for all distributions in $\mathrm{H}^{-q}_{\mathrm{mix}}(\Omega)$ with $d/2 < q \le p+1$. This ensures that particle-based Dirac delta estimators, which are only distributions, can be consistently projected onto the approximation space.

\subsection{Variational discretization}\label{sec:2.3}
In this section, we detail the different steps of the HSG-PIC algorithm.  For notational convenience, the explicit time dependence is omitted in the remainder of this section.

%\paragraph{Charge density approximation onto the mesh}
%We define a particle-based approximation of the charge density on the mesh. Specifically, the Monte-Carlo statistical estimator $\rho_N$, given by the sum of Dirac deltas centered at the particle positions, is projected onto the subspace $V_{h}^{(\star)}(\Omega)$:
%\[
%\Pi_{h}^{(\star)}\rho_N = \sum_{s=1}^N w_s\Pi_{h}^{(\star)}\delta(\cdot - \bm{x}_s).
%\]
We introduce the space of $\mathrm{L}^2$ functions with vanishing mean, defined by
\[
\mathrm{L}^2_0(\Omega)
:= \left\{ u \in \mathrm{L}^2(\Omega) \;\middle|\;
\frac{1}{\mu(\Omega)} \int_{\Omega} u \, d\bm{x} = 0
\right\}, \qquad \mu(\Omega) = \int_{\Omega} \, d\bm{x}.
\]

\paragraph{Resolution of Poisson equation} 
A variational formulation of Poisson equation in the space $V := \mathrm{H}^1(\Omega) \cap \mathrm{L}^2_0(\Omega)$ is considered, written as
\begin{align}
&\text{find } \Phi\in V \text{ such that }  
\quad a(\Phi,v) = l(v), \qquad \forall v \in V.
\label{thm:3:eq:2}
\end{align}
In the above, the bilinear and linear forms are defined, by
\[
a(u,v) = (\bm{\nabla} u , \bm{\nabla} v)_{\mathrm{L}^2(\Omega)},
\quad 
l(v) := \int_{\Omega} \rho \, v \, d\bm{x}, \quad \forall u,v \in V.
\]
\Fabrice{The zero-mean constraint on $V$ removes the constant functions from the
kernel of $a$, which restores the coercivity of $a$ on $V$ by the
Poincar\'e-Wirtinger inequality, and therefore the well-posedness of
\cref{thm:3:eq:2} by the Lax--Milgram theorem.}

This problem is discretized by a Galerkin method on the subspace $V_{h,0}^{(\star)}(\Omega) := V_h^{(\star)}(\Omega) \cap \mathrm{L}^2_0(\Omega)$. Specifically, we solve the discrete problem:
\begin{align}
\left\{  \begin{aligned}
&\text{find } \Phi_{h,N}^{(\star)} \in V_{h,0}^{(\star)}(\Omega) \quad \text{ such that }\\
&a(\Phi_{h,N}^{(\star)}, v_h) = l_{h,N}(v_h), \quad \forall v_h  \in V_{h,0}^{(\star)}(\Omega),
  \end{aligned} \right.
\label{eq:5}
\end{align}
where the discrete linear form is defined by the Galerkin projection of the raw Monte Carlo estimator from \cref{eq:9}:
\[
l_{h,N}(v_h) := \big(\Pi_{h}^{(\star)} \rho_N, v_h\big)_{\mathrm{L}^2(\Omega)}.
\]

%Since $V_{h,0}^{(\star)}(\Omega) \subset V$ is a conforming finite-dimensional subspace, the coercivity of the bilinear form $a(\cdot,\cdot)$ is inherently preserved on $V_{h,0}^{(\star)}(\Omega)$. Consequently, the discrete variational problem \eqref{eq:5} admits a unique solution by the Lax--Milgram theorem.

\Fabrice{We now specify the algebraic formulation of the discrete problem. To derive the corresponding linear system, we expand the approximate potential $\Phi_{h,N}^{(\star)}$ in terms of the hierarchical B-spline basis as
\begin{equation}
  \Phi_{h,N}^{(\star)}(\bm{x}) = \sum_{\bm{\ell}\in\mathscr{H}_h^{(\star)}} \sum_{\bm{j}\in J_{h_{\bm{\ell}}}} \beta_{\bm{\ell},\bm{j}}\, \varphi^p_{h_{\bm{\ell}},\bm{j}}(\bm{x}),
\end{equation}
where the unknown expansion coefficients $\beta_{\bm{\ell},\bm{j}}$ represent the hierarchical surpluses. Testing \cref{eq:5} against all basis functions $\varphi^p_{h_{\bm{\ell}},\bm{j}}$ for $(\bm{\ell},\bm{j})\in\mathscr{H}_h^{(\star)}\times J_{h_{\bm{\ell}}}$ yields the square linear system
\begin{equation}
\label{eq:stiffness_system}
\mathbb{K}\,\bm{\beta} = \bm{f},
\end{equation}
where the entries of the stiffness matrix $\mathbb{K}$ are given by
\[
\mathbb{K}_{(\bm{\ell},\bm{j}),(\bm{\ell}',\bm{j}')}
:= a\big(\varphi^p_{h_{\bm{\ell}'},\bm{j}'},\, \varphi^p_{h_{\bm{\ell}},\bm{j}}\big)
= \int_\Omega \bm{\nabla}\varphi^p_{h_{\bm{\ell}'},\bm{j}'} \cdot \bm{\nabla}\varphi^p_{h_{\bm{\ell}},\bm{j}}\,\mathrm{d}\bm{x}.
\]
The right-hand side vector $\bm{f}$ is assembled directly from the macro-particle positions and weights. By evaluating the discrete linear form $l_{h,N}(v_h) = (\Pi_h^{(\star)}\rho_N, v_h)_{\mathrm{L}^2(\Omega)}$ at a test function $\varphi^p_{h_{\bm{\ell}},\bm{j}}$ and invoking the definition of the generalized $\mathrm{L}^2$-Galerkin projection \cref{eq:def:Galerkin:projection}, we obtain
\[
f_{(\bm{\ell},\bm{j})}
:= l_{h,N}\big(\varphi^p_{h_{\bm{\ell}},\bm{j}}\big)
= \big\langle \rho_N,\, \varphi^p_{h_{\bm{\ell}},\bm{j}} \big\rangle_{\mathrm{H}^{-p}_{\mathrm{mix}},\,\mathrm{H}^p_{\mathrm{mix}}}
= \sum_{s=1}^N w_s\, \big\langle \delta(\cdot-\bm{x}_s), \varphi^p_{h_{\bm{\ell}},\bm{j}} \big\rangle_{\mathrm{H}^{-p}_{\mathrm{mix}},\,\mathrm{H}^p_{\mathrm{mix}}}.
\]
Exploiting the sifting property of the Dirac distribution, this duality pairing simplifies directly to
\begin{equation}
\label{eq:rhs_assembly}
f_{(\bm{\ell},\bm{j})}
= \sum_{s=1}^N w_s\, \varphi^p_{h_{\bm{\ell}},\bm{j}}(\bm{x}_s).
\end{equation}
Each entry of $\bm{f}$ is therefore a weighted sum of the hierarchical
basis function evaluated at the particle positions, with only the
particles lying within the support of $\varphi^p_{h_{\bm{\ell}},\bm{j}}$
contributing a nonzero term.}

\begin{remark}\label{rem:linear:system}
  \Fabrice{Owing to the multiresolution nature of the hierarchical basis, the algebraic structure of the stiffness matrix $\mathbb{K}$ differs notably from standard nodal finite element matrices. At a given resolution level $\bm{\ell}$, the B-spline basis functions $\varphi^p_{h_{\bm{\ell}},\bm{j}}$ and $\varphi^p_{h_{\bm{\ell}},\bm{j}'}$ possess disjoint compact supports for $|\bm{j}-\bm{j}'|$ sufficiently large, ensuring that the intra-level diagonal blocks of the system remain highly sparse. Conversely, the off-diagonal blocks coupling distinct resolution levels $\bm{\ell} \neq \bm{\ell}'$ exhibit a denser connectivity profile, as a single coarse-level basis function naturally overlaps with a multitude of fine-level ones.} 

 \Fabrice{Nevertheless, because a fine-level function only couples with a limited number of hierarchical ancestors, the total number of non-zero entries per row remains bounded, and the matrix remains globally sparse. For the two-dimensional configurations investigated in this work, standard sparse direct solvers exhibit excellent efficiency and robustness, handling the localized hierarchical fill-in without significant computational overhead. For large-scale three-dimensional applications, however, where direct methods face severe memory limitations, the efficient resolution of \eqref{eq:stiffness_system} should require iterative schemes tailored to hierarchical spaces and dedicated multilevel preconditioners~\cite{bungartz04,bungartz_note_1999}.}
\end{remark}

\paragraph{Interpolation of the field at particles' position}
The electric field,  defined as the gradient of the electric potential, is evaluated at the particle positions as 
\[
\bm{E}_{h,N}^{(\star)}(\bm{x}_s) = \left (- \bm{\nabla} \Phi_{h,N}^{(\star)}\right) (\bm{x}_s)=\Fabrice{- \sum_{\bm{\ell}\in\mathscr{H}_h^{(\star)}}
\sum_{\bm{j}\in J_{h_{\bm{\ell}}}} \beta_{\bm{\ell},\bm{j}}\,
\bm{\nabla} \varphi^p_{h_{\bm{\ell}},\bm{j}}(\bm{x}_s)}.
\]
Note that when using linear basis functions, the electric field is not continuous; consequently, the relation above is only defined in the weak sense at the mesh nodes (i.e., when $\bm{x}_s=\bm{j}h_{\bm{\ell}}$). In
practice this is a measure-zero event that does not affect the particle
dynamics.

\paragraph{Evolution of particles}
The particles are advanced in time according to \cref{eq:1}. The particle pushing step is identical to that of the SGCT-PIC and STD-PIC methods. 

\subsection{A priori error estimates}\label{sec:2.4}
In this section, we derive estimates for the grid-based (bias) and the statistical (variance) components of the error.
By linearity of the projection, the bias of the density estimator can be expressed as
\[
\bias[\Pi_{h}^{(\star)} \rho_N] = (\operatorname{I} - \Pi_{h}^{(\star)}) \left( \esp[\rho_N] \right).
\]

The main result of this paper is the following theorem, which provides error estimates on the accuracy of both the density and electric field estimators for the $\mathrm{L}^2$-based, the $\mathrm{H}^1$-based and the full-grid approximation spaces, separating grid-based and statistical contributions.

\begin{theorem}[Error estimates]
\label{thm:0}
Let $f\in \mathrm{L}^1(\mathrm{H}^{q}_{\mathrm{mix}}(\Omega);\Omega_{\bm{v}})$ with $p\in\mathbb{N}$, and $0\leq q\leq p+1$, then the following statements hold:
\begin{enumerate}[label=\roman*)]
\item 
\label{thm:0:1}
The grid-based error of the charge density verifies
\begin{align*}
&\| \bias[\Pi_{h}^{(1)} \rho_{N}]\|_{\mathrm{L}^2(\Omega)} \lesssim h^{q}|\log h|^{d-1}\|\rho\|_{\mathrm{H}^{q}_{\mathrm{mix}}(\Omega)}, \\
& \| \bias[\Pi_{h}^{(\infty)} \rho_{N}]\|_{\mathrm{L}^2(\Omega)} \lesssim h^{q}\|\rho\|_{\mathrm{H}^{q}(\Omega)}.
\end{align*}
\item \label{thm:0:2} The statistical error of the charge density verifies
\begin{align*}
&\big\| \var[\Pi_{h}^{(1)} \rho_{N}]^{\frac{1}{2}}\big\|_{\mathrm{L}^1(\Omega)} \lesssim \left(\frac{|\log h|^{d-1}}{Nh}\right)^{\frac{1}{2}}, \\
& \big\| \var[\Pi_{h}^{(\infty)} \rho_{N}]^{\frac{1}{2}}\big\|_{\mathrm{L}^1(\Omega)} \lesssim \left(\frac{1}{Nh^d}\right)^{\frac{1}{2}}.
\end{align*} 
\item \label{thm:1:1}
If $p=1$ and $(\star)=(1)~\text{or }(\infty)$, the grid-based error of the electric field verifies
\[
\| \bias[\Phi_h^{(\star)}]\|_{\mathrm{H}^1(\Omega)} \lesssim h\|\rho\|_{\mathrm{H}^2_{\mathrm{mix}}(\Omega)}.
\]
\item \label{thm:1:2} The statistical error of the electric field verifies
\begin{align*}
&\sum_{i=1}^d \left\|\var[(\bm{\nabla}\Phi_h^{(1)})_i]^{\frac{1}{2}}\right\|_{\mathrm{L}^1(\Omega)}\lesssim \left(\frac{|\log h|^{d-1}}{Nh}\right)^{\frac{1}{2}}, \\
& \sum_{i=1}^d \left\|\var[(\bm{\nabla}\Phi_h^{(\infty)})_i]^{\frac{1}{2}}\right\|_{\mathrm{L}^1(\Omega)}\lesssim \left(\frac{1}{Nh^d}\right)^{\frac{1}{2}}.
\end{align*}
\end{enumerate}
\end{theorem}
\begin{proof}[Proof of \cref{thm:0}]
See \cref{apd:1} of the appendix.
\end{proof}

We now discuss the results and compare them with the error estimates of STD- and SGTC-PIC methods, which can be found, e.g., in ~\cite{bottino15,tranquilli22} for STD-PIC methods, in~\cite{deluzet22,deluzet25} for SGCT-PIC methods, and are recalled in~\cref{prop:1}.

\paragraph{Regularity assumptions}
A first important difference lies in the regularity assumptions required by the different frameworks. Consider the STD-, SGCT-, and HSG-PIC methods using shape or basis functions of degree~$p$. For the HSG-PIC method, optimal convergence rates are obtained when the particle distribution belongs to the mixed Sobolev space $\mathrm{H}^{p+1}_{\mathrm{mix}}(\Omega)$, which requires all mixed derivatives up to order $p+1$ to be square integrable. In contrast, the STD- and SGCT-PIC analyses assume that the particle distribution belongs to the mixed regularity space $\mathrm{X}^{q}_{\mathrm{mix}}(\Omega)$, with $q = p + q_u$, $q_u\geq1$, thereby requiring all mixed derivatives up to order $q$ to be continuous.
A second distinction concerns the norm in which the error is analyzed. For the STD- and SGCT-PIC methods, the error estimates are established in the $\mathrm{L}^\infty$ norm, whereas for the HSG-PIC method they are derived in the $\mathrm{L}^2$ norm.

\paragraph{Grid-based error}
Assuming sufficient regularity, we now compare the grid-based error convergence rates for both methods. For the charge density, the grid-based error exhibits the same asymptotic behavior in both frameworks, scaling as $\mathcal{O}(h^{p+1}|\log h|^{d-1})$ for sparse-grid approximations (HSG-$(1)$, SGCT) and as $\mathcal{O}(h^{p+1})$ for full-grid approximations (HSG-$(\infty)$, SGCT). We prove, for HSG-PIC methods with $p=1$, that the grid-based error in the electric field scales as $\mathcal{O}(h)$ for the two approximation spaces. We do not provide results for HSG with higher spline degrees, as it requires additional technical proofs, but we expect the error to scale as $\mathcal{O}(h^p)$. This conjecture will be confirmed numerically in~\cref{sec:3}.
 
\paragraph{Statistical error}
We now discuss the statistical component error estimates.
The error estimates obtained for the HSG method scale as $\mathcal{O}(|\log h|^{(d-1)/2}(Nh)^{-1/2})$ both for the charge density and electric field.
However, as proved in~\cite{tranquilli22}, the growth of the statistical error has a significantly milder dependence on the mesh size for the electric field than for the charge density.  Indeed, for two dimensional STD-PIC algorithms, it scales as $\mathcal{O}(\sqrt{\log h / N})$. We conjecture that the HSG estimates derived here is not sharp. This is further discussed in the numerical experiments of this paper in~\cref{sec:3}.

%#######################################
\section{Numerical results}
\label{sec:3}
\subsection{General settings}
We restrict our numerical study to the two-dimensional case, the smallest significant dimension for sparse-grid methods.

We recall that $h$ denotes the mesh size and $\Delta t$ the time step. Given an approximation $u_h^\kappa$ of a function $u(t)$ at time $\kappa$, the relative $\mathrm{L}^q$-norm of the error $u_h^\kappa - u(t)$ is given by
\[
\frac{\|u_h^\kappa - u(t)\|_{\mathrm{L}^q}(\Omega)}{\|u(t)\|_{\mathrm{L}^q(\Omega)}}, 
\]
where the integral are decomposed onto the finest mesh, e.g. 
\[
\|u(t)\|_{\mathrm{L}^q(\Omega)} =  \left( \sum_{i_1,i_2} \int_{i_1 h}^{(i_1+\frac{1}{2}) h} \int_{i_2 h}^{(i_2+\frac{1}{2}) h} | u(t,x,y)|^q \, dx\, dy +  \int_{(i_1+\frac{1}{2}) h}^{(i_1+1) h} \int_{(i_2+\frac{1}{2}) h}^{(i_2+1) h} | u(t,x,y)|^q \, dx\, dy \right)^{1/q}
\]
and  approximated using Gauss-Legendre quadrature. The number of quadrature points $n_q$ is chosen according to the spline degree $p$ as $n_q = p + 1$. With a slight abuse of notation, we write $u(t)$ for the spatial function $u(t,\cdot)\colon(x,y)\mapsto u(t,x,y)$, so that $\|u(t)\|_{\mathrm{L}^2(\Omega)}$ involves an integration over the space variables at fixed time $t$; the dependence on $(x,y)$ is written explicitly whenever needed to avoid ambiguity.

We compare different methods and introduce shortcut notations: the standard (STD-)PIC method, the sparse-grid combination technique (SGCT-)PIC method, the hierarchical sparse-grid PIC method with the $\mathrm{L}^2$-based approximation space (HSG-$(1)$) and the full-grid approximation space (HSG-$(\infty)$). For each of the method, we may use different spline degrees $p_q=q$ where $q=1,3,5$.

\paragraph{Methodology}
For the numerical evaluation of our method, we adopt the following methodology: we first investigate the accuracy of the method (\cref{sec:3.1}); then we evaluate its performance (\cref{sec:3.2}).

We isolate the grid-based (\cref{sec:3.1.1}) and statistical (\cref{sec:3.1.2}) components of the errors and evaluate their convergence separately for a smooth solution; then, we investigate the behavior of the global errors (\cref{sec:3.1.3}) both for a smooth and a non-smooth solution.

For our numerical investigation, we select two representative test cases: a solution constructed using the method of manufactured solutions~\cite{tranquilli22}; and the diocotron instability~\cite{driscoll90}. 

\paragraph{Manufactured solutions}
The method of manufactured solutions, as introduced in~\cite{tranquilli22}, provides a rigorous and deterministic strategy for verification of multidimensional, multispecies, electrostatic PIC implementations. The method requires small modifications to the PIC algorithm as introduced in this paper. Both the electrons and ions have to be discretized and followed in the phase space, and the weights of the particles are evolved throughout the simulation. Specifically, the weights of the particles $w_s$ evolve as
\[
\frac{dw_s}{dt}=\frac{S_f(\bm{x}_s(t),\bm{v}_s(t),t)}{\rho^0(\bm{x}_s(0),\bm{v}_s(0))},
\]
where $\rho^0$ the initial charge density, $S_f$ a forcing term given by
\[
S_{f_s} = \frac{\partial f^s}{\partial t}+\bm{v}\cdot \bm{\nabla}_{\bm{x}}f^s + \bm{E}\cdot  \bm{\nabla}_{\bm{v}}f^s ,
\]
and where $s$ denotes the species of the particles, i.e. electrons or ions. Within this framework, an analytical solution can be constructed. We choose a solution constructed from a distribution of electrons and ions given by
\begin{align*}
f^e(\bm{x},\bm{v})&:= \frac{f^0_{\bm{v}}(\bm{v})}{L^2}\left(1-\sin(\pi t)\sin\left(\frac{2\pi x}{L}\right)\sin\left(\frac{2\pi y}{L}\right)\right), \\
f^i(\bm{x},\bm{v})&:= \frac{f^0_{\bm{v}}(\bm{v})}{L^2},  \quad \text{and}~f^0_{\bm{v}}(\bm{v}):=\frac{4}{\pi}v_1^2v_2^2\exp\left(-v_1^2-v_2^2\right).
\end{align*}
The distributions are defined in the spatial periodic domain $\Omega:=(0,L)^2$, with $L=60$, and the velocity domain $\Omega_{\bm{v}}:=\mathbb{R}^2$.
The electric field is given by
\[
\bm{E}(\bm{x})=-\frac{1}{4L\pi}\sin(\pi t)\left(\cos\left(\frac{2\pi x}{L}\right)\sin\left(\frac{2\pi y}{L}\right),\sin\left(\frac{2\pi x}{L}\right)\cos\left(\frac{2\pi y}{L}\right)\right).
\]
The charge of electrons and ions are related as $q_i=-q_e=1$.
\paragraph{Diocotron instability}
We consider a radially symmetric Gaussian ring as the initial spatial distribution of electrons, defined by
\begin{align}
\label{eq:6}
f^0_{\bm{x}}(\bm{x})=\alpha \exp\left(-\frac{(\|\bm{x}-\frac{L}{2}\|_2-\frac{L}{4})^2}{2(\beta L)^2} \right), \quad \text{with}~\beta = 0.03~ \text{and}~\alpha ~\text{s.t.} \int_{\Omega}f^0_{\bm{x}} dxdy =1,
\end{align}
along with a Maxwellian velocity distribution given by
\begin{align}
\label{eq:7}
f^0_{\bm{v}}(\bm{v}) = \left(\frac{1}{\sqrt{\pi} v_T}\right)^3 \exp\Big(-\frac{\|\bm{v}\|_2^2}{v_T^2}\Big), \quad v_T = \sqrt{2 T_e q_e / m_e}.
\end{align}
Here, $\|\cdot\|_2$ denotes the Euclidean norm. The small value of $\beta$ ensures a narrow Gaussian ring with strong spatial variations that are not aligned with the axes, a scenario known to be very challenging for sparse-grid methods.
The domain size is set to $L=60$, and an external uniform magnetic field is applied along the $z$-axis, $\bm{B}_0 = (0,0,B_z)$, with $B_z=15$. The magnetic field is sufficiently strong such that the electron dynamics is dominated by advection in the self-consistent $\bm{E}\times \bm{B}_0$ field~\cite{driscoll90}.  

The magnetic field induces an instability that deforms the initially radially symmetric electron density, forming vortices. Such dynamics are particularly demanding for sparse-grid approximations~\cite{ricketson17,muralikrishnan21,deluzet22,deluzet22-1}, as strong gradients in mixed directions (xy-plane) develop, and the solution exhibits fine-scale, non-aligned structures. The strong magnetic field also imposes a constraint on the time step: it must be smaller than the electron gyroperiod, with cyclotron frequency $\Omega_c = B_z$. We set $\Delta t = 0.02$ to satisfy $\Omega_c \Delta t \leq 1$.

\paragraph{Grid heating}
The so-called finite grid instability~\cite{langdon70} is a well-known numerical instability of PIC simulations that arises from the mismatch between the discrete Eulerian grid used for field computations and the continuous Lagrangian particle representation~\cite{huang23}. This often manifests as an unphysical plasma heating due to an accumulation of particle kinetic energy over time and is influenced by the numerical discretization parameters. To mitigate this effect, standard PIC algorithm guidelines require the grid spacing $h$ to be smaller than or equal to the Debye length $\lambda_D$, even if the physical scales of interest are larger. In our configuration, the mesh size has to verify $h\leq 1/L$.

\subsection{Convergence and accuracy}
\label{sec:3.1}
In this section, we investigate the convergence of the method and compare the numerical results with the theoretical bounds derived in \cref{thm:0}.  We focus on the convergence according to the spatial discretization; therefore we fix the time discretization to a value such that the global error is dominated by the grid-based or the statistical error. Our numerical investigations have shown that $\Delta t =0.02$ is sufficient so that we fix the time step to this value in the following.

\subsubsection{Grid-based error convergence}
\label{sec:3.1.1}
We first investigate the convergence of the grid-based component of the error, which depends on two discretization parameters: the mesh size $h$ and spline degree $p$. Since the global error of PIC methods is largely dominated by the statistical error (especially for fine meshes and high-degree splines), observing the theoretical convergence orders requires a tremendous amount of numerical particles. 
\Fabrice{To isolate and analyze the structural approximation error, we eliminate the statistical noise by replacing the particle-based estimator $\rho_N$ with the exact continuous charge density $\rho$. Consequently, within the scope of~\cref{sec:3.1.1}, the discrete variational formulation \eqref{eq:5} is modified as follows:
\[
\text{Find } \Phi_{h,N}^{(\star)} \in V_{h,0}^{(\star)}(\Omega) \quad \text{such that} \quad 
a(\Phi_{h,N}^{(\star)}, v_h) = (\Pi_{h}^{(\star)} \rho , v_h)_{\mathrm{L}^2(\Omega)}, \quad \forall v_h \in V_{h,0}^{(\star)}(\Omega),
\]
where the right-hand side linear functional is now evaluated using the exact density $\rho$ instead of the raw Monte Carlo estimator $\rho_N$.}
\paragraph{Manufactured solutions}
We first consider the solution constructed with the method of manufactured solution and we analyze the errors at time $t=0.5$.

The $\mathrm{L}^2$ norm error of the charge density approximation is plotted in the left panel of \cref{fig:1} as a function of $h$ for different spline degrees $p$ and approximations spaces.

\begin{center}
\begin{center}
  \resizebox{0.9\textwidth}{!}{
	% GNUPLOT: LaTeX picture with Postscript
\begingroup
  \makeatletter
  \providecommand\color[2][]{%
    \GenericError{(gnuplot) \space\space\space\@spaces}{%
      Package color not loaded in conjunction with
      terminal option `colourtext'%
    }{See the gnuplot documentation for explanation.%
    }{Either use 'blacktext' in gnuplot or load the package
      color.sty in LaTeX.}%
    \renewcommand\color[2][]{}%
  }%
  \providecommand\includegraphics[2][]{%
    \GenericError{(gnuplot) \space\space\space\@spaces}{%
      Package graphicx or graphics not loaded%
    }{See the gnuplot documentation for explanation.%
    }{The gnuplot epslatex terminal needs graphicx.sty or graphics.sty.}%
    \renewcommand\includegraphics[2][]{}%
  }%
  \providecommand\rotatebox[2]{#2}%
  \@ifundefined{ifGPcolor}{%
    \newif\ifGPcolor
    \GPcolortrue
  }{}%
  \@ifundefined{ifGPblacktext}{%
    \newif\ifGPblacktext
    \GPblacktexttrue
  }{}%
  % define a \g@addto@macro without @ in the name:
  \let\gplgaddtomacro\g@addto@macro
  % define empty templates for all commands taking text:
  \gdef\gplbacktext{}%
  \gdef\gplfronttext{}%
  \makeatother
  \ifGPblacktext
    % no textcolor at all
    \def\colorrgb#1{}%
    \def\colorgray#1{}%
  \else
    % gray or color?
    \ifGPcolor
      \def\colorrgb#1{\color[rgb]{#1}}%
      \def\colorgray#1{\color[gray]{#1}}%
      \expandafter\def\csname LTw\endcsname{\color{white}}%
      \expandafter\def\csname LTb\endcsname{\color{black}}%
      \expandafter\def\csname LTa\endcsname{\color{black}}%
      \expandafter\def\csname LT0\endcsname{\color[rgb]{1,0,0}}%
      \expandafter\def\csname LT1\endcsname{\color[rgb]{0,1,0}}%
      \expandafter\def\csname LT2\endcsname{\color[rgb]{0,0,1}}%
      \expandafter\def\csname LT3\endcsname{\color[rgb]{1,0,1}}%
      \expandafter\def\csname LT4\endcsname{\color[rgb]{0,1,1}}%
      \expandafter\def\csname LT5\endcsname{\color[rgb]{1,1,0}}%
      \expandafter\def\csname LT6\endcsname{\color[rgb]{0,0,0}}%
      \expandafter\def\csname LT7\endcsname{\color[rgb]{1,0.3,0}}%
      \expandafter\def\csname LT8\endcsname{\color[rgb]{0.5,0.5,0.5}}%
    \else
      % gray
      \def\colorrgb#1{\color{black}}%
      \def\colorgray#1{\color[gray]{#1}}%
      \expandafter\def\csname LTw\endcsname{\color{white}}%
      \expandafter\def\csname LTb\endcsname{\color{black}}%
      \expandafter\def\csname LTa\endcsname{\color{black}}%
      \expandafter\def\csname LT0\endcsname{\color{black}}%
      \expandafter\def\csname LT1\endcsname{\color{black}}%
      \expandafter\def\csname LT2\endcsname{\color{black}}%
      \expandafter\def\csname LT3\endcsname{\color{black}}%
      \expandafter\def\csname LT4\endcsname{\color{black}}%
      \expandafter\def\csname LT5\endcsname{\color{black}}%
      \expandafter\def\csname LT6\endcsname{\color{black}}%
      \expandafter\def\csname LT7\endcsname{\color{black}}%
      \expandafter\def\csname LT8\endcsname{\color{black}}%
    \fi
  \fi
    \setlength{\unitlength}{0.0500bp}%
    \ifx\gptboxheight\undefined%
      \newlength{\gptboxheight}%
      \newlength{\gptboxwidth}%
      \newsavebox{\gptboxtext}%
    \fi%
    \setlength{\fboxrule}{0.5pt}%
    \setlength{\fboxsep}{1pt}%
    \definecolor{tbcol}{rgb}{1,1,1}%
\begin{picture}(4020.00,3160.00)%
    \gplgaddtomacro\gplbacktext{%
      \csname LTb\endcsname%%
      \put(615,639){\makebox(0,0)[r]{\strut{}\small$10^{-9}$}}%
      \csname LTb\endcsname%%
      \put(615,902){\makebox(0,0)[r]{\strut{}\small$10^{-8}$}}%
      \csname LTb\endcsname%%
      \put(615,1165){\makebox(0,0)[r]{\strut{}\small$10^{-7}$}}%
      \csname LTb\endcsname%%
      \put(615,1428){\makebox(0,0)[r]{\strut{}\small$10^{-6}$}}%
      \csname LTb\endcsname%%
      \put(615,1691){\makebox(0,0)[r]{\strut{}\small$10^{-5}$}}%
      \csname LTb\endcsname%%
      \put(615,1954){\makebox(0,0)[r]{\strut{}\small$10^{-4}$}}%
      \csname LTb\endcsname%%
      \put(615,2217){\makebox(0,0)[r]{\strut{}\small$10^{-3}$}}%
      \csname LTb\endcsname%%
      \put(615,2480){\makebox(0,0)[r]{\strut{}\small$10^{-2}$}}%
      \csname LTb\endcsname%%
      \put(615,2743){\makebox(0,0)[r]{\strut{}\small$10^{-1}$}}%
      \csname LTb\endcsname%%
      \put(615,3006){\makebox(0,0)[r]{\strut{}\small$10^{0}$}}%
      \csname LTb\endcsname%%
      \put(1278,426){\makebox(0,0){\strut{}$10^{-2}$}}%
      \csname LTb\endcsname%%
      \put(2119,426){\makebox(0,0){\strut{}$10^{-1}$}}%
      \csname LTb\endcsname%%
      \put(2961,426){\makebox(0,0){\strut{}$\small 10^{0}$}}%
    }%
    \gplgaddtomacro\gplfronttext{%
      \csname LTb\endcsname%%
      \put(3612,2845){\makebox(0,0)[r]{\strut{}\footnotesize$h^{2}$}}%
      \csname LTb\endcsname%%
      \put(3612,2698){\makebox(0,0)[r]{\strut{}\footnotesize$h^{4}$}}%
      \csname LTb\endcsname%%
      \put(3612,2552){\makebox(0,0)[r]{\strut{}\footnotesize$h^{6}$}}%
      \csname LTb\endcsname%%
      \put(3612,2405){\makebox(0,0)[r]{\strut{}\footnotesize$(1),p_1$}}%
      \csname LTb\endcsname%%
      \put(3612,2258){\makebox(0,0)[r]{\strut{}\footnotesize$(\infty),p_1$}}%
      \csname LTb\endcsname%%
      \put(3612,2112){\makebox(0,0)[r]{\strut{}\footnotesize$(1),p_3$}}%
      \csname LTb\endcsname%%
      \put(3612,1965){\makebox(0,0)[r]{\strut{}\footnotesize$(\infty),p_3$}}%
      \csname LTb\endcsname%%
      \put(3612,1819){\makebox(0,0)[r]{\strut{}\footnotesize$(1),p_5$}}%
      \csname LTb\endcsname%%
      \put(3612,1672){\makebox(0,0)[r]{\strut{}\footnotesize$(\infty),p_5$}}%
      \csname LTb\endcsname%%
      \put(48,1783){\rotatebox{-270.00}{\makebox(0,0){\strut{}$\|\Pi_{h}^{(\star)} \rho(t) - \rho(t))\|_{\mathrm{L}^2(\Omega)}$}}}%
      \csname LTb\endcsname%%
      \put(1825,93){\makebox(0,0){\strut{}Mesh size $h$}}%
    }%
    \gplbacktext
    \put(0,0){\includegraphics[width={201.00bp},height={158.00bp}]{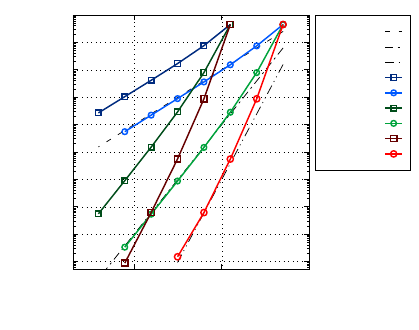}}%
    \gplfronttext
  \end{picture}%
\endgroup
 \hspace*{1em}
	% GNUPLOT: LaTeX picture with Postscript
\begingroup
  \makeatletter
  \providecommand\color[2][]{%
    \GenericError{(gnuplot) \space\space\space\@spaces}{%
      Package color not loaded in conjunction with
      terminal option `colourtext'%
    }{See the gnuplot documentation for explanation.%
    }{Either use 'blacktext' in gnuplot or load the package
      color.sty in LaTeX.}%
    \renewcommand\color[2][]{}%
  }%
  \providecommand\includegraphics[2][]{%
    \GenericError{(gnuplot) \space\space\space\@spaces}{%
      Package graphicx or graphics not loaded%
    }{See the gnuplot documentation for explanation.%
    }{The gnuplot epslatex terminal needs graphicx.sty or graphics.sty.}%
    \renewcommand\includegraphics[2][]{}%
  }%
  \providecommand\rotatebox[2]{#2}%
  \@ifundefined{ifGPcolor}{%
    \newif\ifGPcolor
    \GPcolortrue
  }{}%
  \@ifundefined{ifGPblacktext}{%
    \newif\ifGPblacktext
    \GPblacktexttrue
  }{}%
  % define a \g@addto@macro without @ in the name:
  \let\gplgaddtomacro\g@addto@macro
  % define empty templates for all commands taking text:
  \gdef\gplbacktext{}%
  \gdef\gplfronttext{}%
  \makeatother
  \ifGPblacktext
    % no textcolor at all
    \def\colorrgb#1{}%
    \def\colorgray#1{}%
  \else
    % gray or color?
    \ifGPcolor
      \def\colorrgb#1{\color[rgb]{#1}}%
      \def\colorgray#1{\color[gray]{#1}}%
      \expandafter\def\csname LTw\endcsname{\color{white}}%
      \expandafter\def\csname LTb\endcsname{\color{black}}%
      \expandafter\def\csname LTa\endcsname{\color{black}}%
      \expandafter\def\csname LT0\endcsname{\color[rgb]{1,0,0}}%
      \expandafter\def\csname LT1\endcsname{\color[rgb]{0,1,0}}%
      \expandafter\def\csname LT2\endcsname{\color[rgb]{0,0,1}}%
      \expandafter\def\csname LT3\endcsname{\color[rgb]{1,0,1}}%
      \expandafter\def\csname LT4\endcsname{\color[rgb]{0,1,1}}%
      \expandafter\def\csname LT5\endcsname{\color[rgb]{1,1,0}}%
      \expandafter\def\csname LT6\endcsname{\color[rgb]{0,0,0}}%
      \expandafter\def\csname LT7\endcsname{\color[rgb]{1,0.3,0}}%
      \expandafter\def\csname LT8\endcsname{\color[rgb]{0.5,0.5,0.5}}%
    \else
      % gray
      \def\colorrgb#1{\color{black}}%
      \def\colorgray#1{\color[gray]{#1}}%
      \expandafter\def\csname LTw\endcsname{\color{white}}%
      \expandafter\def\csname LTb\endcsname{\color{black}}%
      \expandafter\def\csname LTa\endcsname{\color{black}}%
      \expandafter\def\csname LT0\endcsname{\color{black}}%
      \expandafter\def\csname LT1\endcsname{\color{black}}%
      \expandafter\def\csname LT2\endcsname{\color{black}}%
      \expandafter\def\csname LT3\endcsname{\color{black}}%
      \expandafter\def\csname LT4\endcsname{\color{black}}%
      \expandafter\def\csname LT5\endcsname{\color{black}}%
      \expandafter\def\csname LT6\endcsname{\color{black}}%
      \expandafter\def\csname LT7\endcsname{\color{black}}%
      \expandafter\def\csname LT8\endcsname{\color{black}}%
    \fi
  \fi
    \setlength{\unitlength}{0.0500bp}%
    \ifx\gptboxheight\undefined%
      \newlength{\gptboxheight}%
      \newlength{\gptboxwidth}%
      \newsavebox{\gptboxtext}%
    \fi%
    \setlength{\fboxrule}{0.5pt}%
    \setlength{\fboxsep}{1pt}%
    \definecolor{tbcol}{rgb}{1,1,1}%
\begin{picture}(4020.00,3160.00)%
    \gplgaddtomacro\gplbacktext{%
      \csname LTb\endcsname%%
      \put(615,559){\makebox(0,0)[r]{\strut{}\small$10^{-9}$}}%
      \csname LTb\endcsname%%
      \put(615,831){\makebox(0,0)[r]{\strut{}\small$10^{-8}$}}%
      \csname LTb\endcsname%%
      \put(615,1103){\makebox(0,0)[r]{\strut{}\small$10^{-7}$}}%
      \csname LTb\endcsname%%
      \put(615,1375){\makebox(0,0)[r]{\strut{}\small$10^{-6}$}}%
      \csname LTb\endcsname%%
      \put(615,1647){\makebox(0,0)[r]{\strut{}\small$10^{-5}$}}%
      \csname LTb\endcsname%%
      \put(615,1919){\makebox(0,0)[r]{\strut{}\small$10^{-4}$}}%
      \csname LTb\endcsname%%
      \put(615,2191){\makebox(0,0)[r]{\strut{}\small$10^{-3}$}}%
      \csname LTb\endcsname%%
      \put(615,2462){\makebox(0,0)[r]{\strut{}\small$10^{-2}$}}%
      \csname LTb\endcsname%%
      \put(615,2734){\makebox(0,0)[r]{\strut{}\small$10^{-1}$}}%
      \csname LTb\endcsname%%
      \put(615,3006){\makebox(0,0)[r]{\strut{}\small$10^{0}$}}%
      \csname LTb\endcsname%%
      \put(1278,426){\makebox(0,0){\strut{}$10^{-2}$}}%
      \csname LTb\endcsname%%
      \put(2119,426){\makebox(0,0){\strut{}$10^{-1}$}}%
      \csname LTb\endcsname%%
      \put(2961,426){\makebox(0,0){\strut{}$\small 10^{0}$}}%
    }%
    \gplgaddtomacro\gplfronttext{%
      \csname LTb\endcsname%%
      \put(3612,2845){\makebox(0,0)[r]{\strut{}\footnotesize $h$}}%
      \csname LTb\endcsname%%
      \put(3612,2698){\makebox(0,0)[r]{\strut{}\footnotesize $h^3$}}%
      \csname LTb\endcsname%%
      \put(3612,2552){\makebox(0,0)[r]{\strut{}\footnotesize$h^{5}$}}%
      \csname LTb\endcsname%%
      \put(3612,2405){\makebox(0,0)[r]{\strut{}\footnotesize$(1),p_1$}}%
      \csname LTb\endcsname%%
      \put(3612,2258){\makebox(0,0)[r]{\strut{}\footnotesize$(\infty),p_1$}}%
      \csname LTb\endcsname%%
      \put(3612,2112){\makebox(0,0)[r]{\strut{}\footnotesize$(1),p_3$}}%
      \csname LTb\endcsname%%
      \put(3612,1965){\makebox(0,0)[r]{\strut{}\footnotesize$(\infty),p_3$}}%
      \csname LTb\endcsname%%
      \put(3612,1819){\makebox(0,0)[r]{\strut{}\footnotesize$(1),p_5$}}%
      \csname LTb\endcsname%%
      \put(3612,1672){\makebox(0,0)[r]{\strut{}\footnotesize$(\infty),p_5$}}%
      \csname LTb\endcsname%%
      \put(48,1783){\rotatebox{-270.00}{\makebox(0,0){\strut{}$|\Phi_{h}^{(\star)}(t)- \Phi(t))|_{\mathrm{H}^1(\Omega)}$}}}%
      \csname LTb\endcsname%%
      \put(1825,93){\makebox(0,0){\strut{}Mesh size $h$}}%
    }%
    \gplbacktext
    \put(0,0){\includegraphics[width={201.00bp},height={158.00bp}]{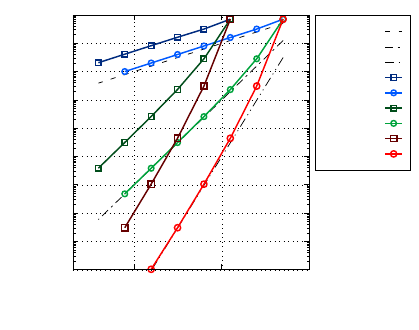}}%
    \gplfronttext
  \end{picture}%
\endgroup
}
 \end{center}
 \captionof{figure}{$\mathrm{L}^2$ norm and $\mathrm{H}^1$ semi-norm approximation error of the charge density (left) and electric field (right) for the manufactured solutions as functions of the mesh size $h$, for different spline degrees $p$, and at time $t=0.5$.}
\label{fig:1}
\end{center}

%First, we observe that the approximations with the truncated sparse-grid spaces do not converge until a sufficiently fine mesh size is achieved (around $h= 2^4$ or $2^5$). The reason for this behavior can be explained as follows. The variations of this specific electron distribution requires a discretization of at least $h=2^2$ in both directions to be well captured and the subspace $W_{h_{(2,2)}}$ is only included the sparse-grid spaces verifying $h\ge 2^n$, with $n=4$ for the $\mathrm{L}^2$-based space and $n=5$ for the $\mathrm{H}^1$-based space. 

For the two approximation spaces, the charge density projection error scales as $\mathcal{O}(h^{p+1})$. These results are consistent with the theoretical bounds of \cref{thm:0}, \cref{thm:0:1} for the full-grid approximation space and better than the estimates predicted for the sparse-grid space as the convergence order does not include the additional logarithmic term. This super-convergence behavior results from the high regularity of the solution in the mixed derivatives. However, we observe a significant difference in the constants between the approximation spaces, which results in a consequent loss of accuracy for the sparse-grid spaces compared to the full-grid space. This deterioration of accuracy can be mitigated by increasing the degree of the B-splines functions.

The $\mathrm{H}^1$ semi-norm error of the electric potential is plotted in the right panel of \cref{fig:1} as a function of the mesh size $h$ for different spline degrees $p$. 
For both approximation spaces, the error converges as $\mathcal{O}(h^{p})$, which is consistent with the error bound of \cref{thm:0}, \cref{thm:1:1} obtained for linear basis functions. As with the charge density, increasing the spline degree yields a higher convergence order for all approximation spaces.

\subsubsection{Statistical error growth}
\label{sec:3.1.2}
We now investigate the statistical component of the error which may depend on three discretization parameters: the mesh size $h$, the spline degree $p$ and the (total) number of particles $N$. 

Preliminary numerical experiments, omitted here for brevity, confirmed that the error scales as $\mathcal{O}(1/\sqrt{N})$ for all methods given a fixed mesh size. These empirical results are consistent with the theoretical bounds derived in \cref{thm:0}, \cref{thm:0:2}, \cref{thm:1:2}, as well as \cref{prop:1}.
This behavior is expected, as all the considered methods rely at some stage on a Monte Carlo sampling of the particle distribution, an approach whose statistical error is well known to scale with the inverse square root of the number of samples.

A more determining property to investigate is the growth of the statistical noise when refining the mesh for a fixed number of particles. Indeed, in PIC algorithms, the statistical noise usually scales as the mean number of particles per cell~\cite{bottino15}, that is $\mathcal{O}(\sqrt{1/Nh^d})$ for STD-PIC methods, under the assumption that the particles are uniformly distributed over the spatial domain. This sampling error actually corresponds to the variance of the Monte-Carlo estimator of the charge density.  
\Fabrice{Owing to the regularizing properties of the Poisson operator, solving the discrete Laplacian acts as an intrinsic spatial low-pass filter on particle noise. In Fourier space, the inverse Laplacian scales as $\mathcal{O}(k^{-2})$, which significantly dampens high-frequency fluctuations present in the particle-derived charge density $\rho$ when computing the electrostatic potential $\phi$ and the electric field $\bm{E}$. Consequently, the numerical noise in $\phi$ and $\bm{E}$ is smoother and substantially reduced compared to that of $\rho$. This regularization mechanism is fully aligned with the theoretical sampling error scalings established by Tranquilli \textit{et al.}~\cite{tranquilli22}, who showed that field quantities exhibit lower stochastic variance than density moments, namely $\mathcal{O}(\sqrt{\log h / N})$ for two dimensional problems, due to this spectral decay.}

For all the experiments presented in this section, we have performed $15$ simulations with different random seeds to draw the initial distribution of particles. The results presented display both the mean and the standard deviation of the $15$ realizations.

\paragraph{Manufactured solutions}
We restrict the investigation to the manufactured solution test case. Similar results (although with different constants) have been obtained considering the diocotron instability test case. We fix the number of particles to $N=1{,}000$, so that, even for coarse mesh sizes, the error is dominated by the statistical component. We analyze the errors at time $t=0.5$.

The $\mathrm{L}^2$ norm and $\mathrm{H}^1$ semi-norm of the charge density and electric field approximation errors are plotted as functions of the mesh size in~\cref{fig:2}, considering linear B-splines ($p_1$) on the left panels and cubic B-splines ($p_3$) on the right panels.

\begin{center}
\begin{center}
  \resizebox{0.9\textwidth}{!}{
  % GNUPLOT: LaTeX picture with Postscript
\begingroup
  \makeatletter
  \providecommand\color[2][]{%
    \GenericError{(gnuplot) \space\space\space\@spaces}{%
      Package color not loaded in conjunction with
      terminal option `colourtext'%
    }{See the gnuplot documentation for explanation.%
    }{Either use 'blacktext' in gnuplot or load the package
      color.sty in LaTeX.}%
    \renewcommand\color[2][]{}%
  }%
  \providecommand\includegraphics[2][]{%
    \GenericError{(gnuplot) \space\space\space\@spaces}{%
      Package graphicx or graphics not loaded%
    }{See the gnuplot documentation for explanation.%
    }{The gnuplot epslatex terminal needs graphicx.sty or graphics.sty.}%
    \renewcommand\includegraphics[2][]{}%
  }%
  \providecommand\rotatebox[2]{#2}%
  \@ifundefined{ifGPcolor}{%
    \newif\ifGPcolor
    \GPcolortrue
  }{}%
  \@ifundefined{ifGPblacktext}{%
    \newif\ifGPblacktext
    \GPblacktexttrue
  }{}%
  % define a \g@addto@macro without @ in the name:
  \let\gplgaddtomacro\g@addto@macro
  % define empty templates for all commands taking text:
  \gdef\gplbacktext{}%
  \gdef\gplfronttext{}%
  \makeatother
  \ifGPblacktext
    % no textcolor at all
    \def\colorrgb#1{}%
    \def\colorgray#1{}%
  \else
    % gray or color?
    \ifGPcolor
      \def\colorrgb#1{\color[rgb]{#1}}%
      \def\colorgray#1{\color[gray]{#1}}%
      \expandafter\def\csname LTw\endcsname{\color{white}}%
      \expandafter\def\csname LTb\endcsname{\color{black}}%
      \expandafter\def\csname LTa\endcsname{\color{black}}%
      \expandafter\def\csname LT0\endcsname{\color[rgb]{1,0,0}}%
      \expandafter\def\csname LT1\endcsname{\color[rgb]{0,1,0}}%
      \expandafter\def\csname LT2\endcsname{\color[rgb]{0,0,1}}%
      \expandafter\def\csname LT3\endcsname{\color[rgb]{1,0,1}}%
      \expandafter\def\csname LT4\endcsname{\color[rgb]{0,1,1}}%
      \expandafter\def\csname LT5\endcsname{\color[rgb]{1,1,0}}%
      \expandafter\def\csname LT6\endcsname{\color[rgb]{0,0,0}}%
      \expandafter\def\csname LT7\endcsname{\color[rgb]{1,0.3,0}}%
      \expandafter\def\csname LT8\endcsname{\color[rgb]{0.5,0.5,0.5}}%
    \else
      % gray
      \def\colorrgb#1{\color{black}}%
      \def\colorgray#1{\color[gray]{#1}}%
      \expandafter\def\csname LTw\endcsname{\color{white}}%
      \expandafter\def\csname LTb\endcsname{\color{black}}%
      \expandafter\def\csname LTa\endcsname{\color{black}}%
      \expandafter\def\csname LT0\endcsname{\color{black}}%
      \expandafter\def\csname LT1\endcsname{\color{black}}%
      \expandafter\def\csname LT2\endcsname{\color{black}}%
      \expandafter\def\csname LT3\endcsname{\color{black}}%
      \expandafter\def\csname LT4\endcsname{\color{black}}%
      \expandafter\def\csname LT5\endcsname{\color{black}}%
      \expandafter\def\csname LT6\endcsname{\color{black}}%
      \expandafter\def\csname LT7\endcsname{\color{black}}%
      \expandafter\def\csname LT8\endcsname{\color{black}}%
    \fi
  \fi
    \setlength{\unitlength}{0.0500bp}%
    \ifx\gptboxheight\undefined%
      \newlength{\gptboxheight}%
      \newlength{\gptboxwidth}%
      \newsavebox{\gptboxtext}%
    \fi%
    \setlength{\fboxrule}{0.5pt}%
    \setlength{\fboxsep}{1pt}%
    \definecolor{tbcol}{rgb}{1,1,1}%
\begin{picture}(4020.00,3160.00)%
    \gplgaddtomacro\gplbacktext{%
      \csname LTb\endcsname%%
      \put(615,559){\makebox(0,0)[r]{\strut{}\small$10^{-1}$}}%
      \csname LTb\endcsname%%
      \put(615,1783){\makebox(0,0)[r]{\strut{}\small$10^{0}$}}%
      \csname LTb\endcsname%%
      \put(615,3006){\makebox(0,0)[r]{\strut{}\small$10^{1}$}}%
      \csname LTb\endcsname%%
      \put(1563,426){\makebox(0,0){\strut{}$10^{-2}$}}%
      \csname LTb\endcsname%%
      \put(2812,426){\makebox(0,0){\strut{}$10^{-1}$}}%
    }%
    \gplgaddtomacro\gplfronttext{%
      \csname LTb\endcsname%%
      \put(3612,2845){\makebox(0,0)[r]{\strut{} }}%
      \csname LTb\endcsname%%
      \put(3612,2698){\makebox(0,0)[r]{\strut{}\footnotesize$\sqrt{\frac{1}{Nh^2}}$}}%
      \csname LTb\endcsname%%
      \put(3612,2552){\makebox(0,0)[r]{\strut{} }}%
      \csname LTb\endcsname%%
      \put(3612,2405){\makebox(0,0)[r]{\strut{}\footnotesize$\sqrt{\frac{\log h}{Nh}}$}}%
      \csname LTb\endcsname%%
      \put(3612,2258){\makebox(0,0)[r]{\strut{} }}%
      \csname LTb\endcsname%%
      \put(3612,2112){\makebox(0,0)[r]{\strut{}\footnotesize$(1),p_1$}}%
      \csname LTb\endcsname%%
      \put(3612,1965){\makebox(0,0)[r]{\strut{}\footnotesize$(\infty),p_1$}}%
      \csname LTb\endcsname%%
      \put(3612,1819){\makebox(0,0)[r]{\strut{}\footnotesize SGCT,$p_1$}}%
      \csname LTb\endcsname%%
      \put(3612,1672){\makebox(0,0)[r]{\strut{}\footnotesize STD,$p_1$}}%
      \csname LTb\endcsname%%
      \put(48,1783){\rotatebox{-270.00}{\makebox(0,0){\strut{}$\|\rho_h^{(\star),\tau} - \rho(t))\|_{\mathrm{L}^2(\Omega)}$}}}%
      \csname LTb\endcsname%%
      \put(1751,93){\makebox(0,0){\strut{}Mesh size $h$}}%
    }%
    \gplbacktext
    \put(0,0){\includegraphics[width={201.00bp},height={158.00bp}]{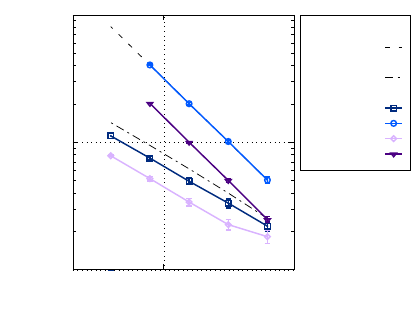}}%
    \gplfronttext
  \end{picture}%
\endgroup
\hspace*{1em}% GNUPLOT: LaTeX picture with Postscript
\begingroup
  \makeatletter
  \providecommand\color[2][]{%
    \GenericError{(gnuplot) \space\space\space\@spaces}{%
      Package color not loaded in conjunction with
      terminal option `colourtext'%
    }{See the gnuplot documentation for explanation.%
    }{Either use 'blacktext' in gnuplot or load the package
      color.sty in LaTeX.}%
    \renewcommand\color[2][]{}%
  }%
  \providecommand\includegraphics[2][]{%
    \GenericError{(gnuplot) \space\space\space\@spaces}{%
      Package graphicx or graphics not loaded%
    }{See the gnuplot documentation for explanation.%
    }{The gnuplot epslatex terminal needs graphicx.sty or graphics.sty.}%
    \renewcommand\includegraphics[2][]{}%
  }%
  \providecommand\rotatebox[2]{#2}%
  \@ifundefined{ifGPcolor}{%
    \newif\ifGPcolor
    \GPcolortrue
  }{}%
  \@ifundefined{ifGPblacktext}{%
    \newif\ifGPblacktext
    \GPblacktexttrue
  }{}%
  % define a \g@addto@macro without @ in the name:
  \let\gplgaddtomacro\g@addto@macro
  % define empty templates for all commands taking text:
  \gdef\gplbacktext{}%
  \gdef\gplfronttext{}%
  \makeatother
  \ifGPblacktext
    % no textcolor at all
    \def\colorrgb#1{}%
    \def\colorgray#1{}%
  \else
    % gray or color?
    \ifGPcolor
      \def\colorrgb#1{\color[rgb]{#1}}%
      \def\colorgray#1{\color[gray]{#1}}%
      \expandafter\def\csname LTw\endcsname{\color{white}}%
      \expandafter\def\csname LTb\endcsname{\color{black}}%
      \expandafter\def\csname LTa\endcsname{\color{black}}%
      \expandafter\def\csname LT0\endcsname{\color[rgb]{1,0,0}}%
      \expandafter\def\csname LT1\endcsname{\color[rgb]{0,1,0}}%
      \expandafter\def\csname LT2\endcsname{\color[rgb]{0,0,1}}%
      \expandafter\def\csname LT3\endcsname{\color[rgb]{1,0,1}}%
      \expandafter\def\csname LT4\endcsname{\color[rgb]{0,1,1}}%
      \expandafter\def\csname LT5\endcsname{\color[rgb]{1,1,0}}%
      \expandafter\def\csname LT6\endcsname{\color[rgb]{0,0,0}}%
      \expandafter\def\csname LT7\endcsname{\color[rgb]{1,0.3,0}}%
      \expandafter\def\csname LT8\endcsname{\color[rgb]{0.5,0.5,0.5}}%
    \else
      % gray
      \def\colorrgb#1{\color{black}}%
      \def\colorgray#1{\color[gray]{#1}}%
      \expandafter\def\csname LTw\endcsname{\color{white}}%
      \expandafter\def\csname LTb\endcsname{\color{black}}%
      \expandafter\def\csname LTa\endcsname{\color{black}}%
      \expandafter\def\csname LT0\endcsname{\color{black}}%
      \expandafter\def\csname LT1\endcsname{\color{black}}%
      \expandafter\def\csname LT2\endcsname{\color{black}}%
      \expandafter\def\csname LT3\endcsname{\color{black}}%
      \expandafter\def\csname LT4\endcsname{\color{black}}%
      \expandafter\def\csname LT5\endcsname{\color{black}}%
      \expandafter\def\csname LT6\endcsname{\color{black}}%
      \expandafter\def\csname LT7\endcsname{\color{black}}%
      \expandafter\def\csname LT8\endcsname{\color{black}}%
    \fi
  \fi
    \setlength{\unitlength}{0.0500bp}%
    \ifx\gptboxheight\undefined%
      \newlength{\gptboxheight}%
      \newlength{\gptboxwidth}%
      \newsavebox{\gptboxtext}%
    \fi%
    \setlength{\fboxrule}{0.5pt}%
    \setlength{\fboxsep}{1pt}%
    \definecolor{tbcol}{rgb}{1,1,1}%
\begin{picture}(4020.00,3160.00)%
    \gplgaddtomacro\gplbacktext{%
      \csname LTb\endcsname%%
      \put(615,559){\makebox(0,0)[r]{\strut{}\small$10^{-1}$}}%
      \csname LTb\endcsname%%
      \put(615,1783){\makebox(0,0)[r]{\strut{}\small$10^{0}$}}%
      \csname LTb\endcsname%%
      \put(615,3006){\makebox(0,0)[r]{\strut{}\small$10^{1}$}}%
      \csname LTb\endcsname%%
      \put(1563,426){\makebox(0,0){\strut{}$10^{-2}$}}%
      \csname LTb\endcsname%%
      \put(2812,426){\makebox(0,0){\strut{}$10^{-1}$}}%
    }%
    \gplgaddtomacro\gplfronttext{%
      \csname LTb\endcsname%%
      \put(3612,2845){\makebox(0,0)[r]{\strut{} }}%
      \csname LTb\endcsname%%
      \put(3612,2698){\makebox(0,0)[r]{\strut{}\footnotesize$\sqrt{\frac{1}{Nh^2}}$}}%
      \csname LTb\endcsname%%
      \put(3612,2552){\makebox(0,0)[r]{\strut{} }}%
      \csname LTb\endcsname%%
      \put(3612,2405){\makebox(0,0)[r]{\strut{}\footnotesize$\sqrt{\frac{\log h}{Nh}}$}}%
      \csname LTb\endcsname%%
      \put(3612,2258){\makebox(0,0)[r]{\strut{} }}%
      \csname LTb\endcsname%%
      \put(3612,2112){\makebox(0,0)[r]{\strut{}\footnotesize$(1),p_3$}}%
      \csname LTb\endcsname%%
      \put(3612,1965){\makebox(0,0)[r]{\strut{}\footnotesize$(\infty),p_3$}}%
      \csname LTb\endcsname%%
      \put(3612,1819){\makebox(0,0)[r]{\strut{}\footnotesize SGCT,$p_3$}}%
      \csname LTb\endcsname%%
      \put(3612,1672){\makebox(0,0)[r]{\strut{}\footnotesize STD,$p_3$}}%
      \csname LTb\endcsname%%
      \put(48,1783){\rotatebox{-270.00}{\makebox(0,0){\strut{}$\|\rho_h^{(\star),\tau} - \rho(t))\|_{\mathrm{L}^2(\Omega)}$}}}%
      \csname LTb\endcsname%%
      \put(1751,93){\makebox(0,0){\strut{}Mesh size $h$}}%
    }%
    \gplbacktext
    \put(0,0){\includegraphics[width={201.00bp},height={158.00bp}]{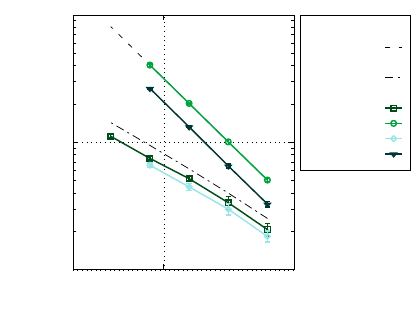}}%
    \gplfronttext
  \end{picture}%
\endgroup
}
 \end{center}
 \begin{center}
  \resizebox{0.9\textwidth}{!}{
  % GNUPLOT: LaTeX picture with Postscript
\begingroup
  \makeatletter
  \providecommand\color[2][]{%
    \GenericError{(gnuplot) \space\space\space\@spaces}{%
      Package color not loaded in conjunction with
      terminal option `colourtext'%
    }{See the gnuplot documentation for explanation.%
    }{Either use 'blacktext' in gnuplot or load the package
      color.sty in LaTeX.}%
    \renewcommand\color[2][]{}%
  }%
  \providecommand\includegraphics[2][]{%
    \GenericError{(gnuplot) \space\space\space\@spaces}{%
      Package graphicx or graphics not loaded%
    }{See the gnuplot documentation for explanation.%
    }{The gnuplot epslatex terminal needs graphicx.sty or graphics.sty.}%
    \renewcommand\includegraphics[2][]{}%
  }%
  \providecommand\rotatebox[2]{#2}%
  \@ifundefined{ifGPcolor}{%
    \newif\ifGPcolor
    \GPcolortrue
  }{}%
  \@ifundefined{ifGPblacktext}{%
    \newif\ifGPblacktext
    \GPblacktexttrue
  }{}%
  % define a \g@addto@macro without @ in the name:
  \let\gplgaddtomacro\g@addto@macro
  % define empty templates for all commands taking text:
  \gdef\gplbacktext{}%
  \gdef\gplfronttext{}%
  \makeatother
  \ifGPblacktext
    % no textcolor at all
    \def\colorrgb#1{}%
    \def\colorgray#1{}%
  \else
    % gray or color?
    \ifGPcolor
      \def\colorrgb#1{\color[rgb]{#1}}%
      \def\colorgray#1{\color[gray]{#1}}%
      \expandafter\def\csname LTw\endcsname{\color{white}}%
      \expandafter\def\csname LTb\endcsname{\color{black}}%
      \expandafter\def\csname LTa\endcsname{\color{black}}%
      \expandafter\def\csname LT0\endcsname{\color[rgb]{1,0,0}}%
      \expandafter\def\csname LT1\endcsname{\color[rgb]{0,1,0}}%
      \expandafter\def\csname LT2\endcsname{\color[rgb]{0,0,1}}%
      \expandafter\def\csname LT3\endcsname{\color[rgb]{1,0,1}}%
      \expandafter\def\csname LT4\endcsname{\color[rgb]{0,1,1}}%
      \expandafter\def\csname LT5\endcsname{\color[rgb]{1,1,0}}%
      \expandafter\def\csname LT6\endcsname{\color[rgb]{0,0,0}}%
      \expandafter\def\csname LT7\endcsname{\color[rgb]{1,0.3,0}}%
      \expandafter\def\csname LT8\endcsname{\color[rgb]{0.5,0.5,0.5}}%
    \else
      % gray
      \def\colorrgb#1{\color{black}}%
      \def\colorgray#1{\color[gray]{#1}}%
      \expandafter\def\csname LTw\endcsname{\color{white}}%
      \expandafter\def\csname LTb\endcsname{\color{black}}%
      \expandafter\def\csname LTa\endcsname{\color{black}}%
      \expandafter\def\csname LT0\endcsname{\color{black}}%
      \expandafter\def\csname LT1\endcsname{\color{black}}%
      \expandafter\def\csname LT2\endcsname{\color{black}}%
      \expandafter\def\csname LT3\endcsname{\color{black}}%
      \expandafter\def\csname LT4\endcsname{\color{black}}%
      \expandafter\def\csname LT5\endcsname{\color{black}}%
      \expandafter\def\csname LT6\endcsname{\color{black}}%
      \expandafter\def\csname LT7\endcsname{\color{black}}%
      \expandafter\def\csname LT8\endcsname{\color{black}}%
    \fi
  \fi
    \setlength{\unitlength}{0.0500bp}%
    \ifx\gptboxheight\undefined%
      \newlength{\gptboxheight}%
      \newlength{\gptboxwidth}%
      \newsavebox{\gptboxtext}%
    \fi%
    \setlength{\fboxrule}{0.5pt}%
    \setlength{\fboxsep}{1pt}%
    \definecolor{tbcol}{rgb}{1,1,1}%
\begin{picture}(4020.00,3160.00)%
    \gplgaddtomacro\gplbacktext{%
      \csname LTb\endcsname%%
      \put(615,559){\makebox(0,0)[r]{\strut{}\small$10^{-1}$}}%
      \csname LTb\endcsname%%
      \put(615,1783){\makebox(0,0)[r]{\strut{}\small$10^{0}$}}%
      \csname LTb\endcsname%%
      \put(615,3006){\makebox(0,0)[r]{\strut{}\small$10^{1}$}}%
      \csname LTb\endcsname%%
      \put(1088,426){\makebox(0,0){\strut{}$10^{-2}$}}%
      \csname LTb\endcsname%%
      \put(2413,426){\makebox(0,0){\strut{}$10^{-1}$}}%
    }%
    \gplgaddtomacro\gplfronttext{%
      \csname LTb\endcsname%%
      \put(3612,2845){\makebox(0,0)[r]{\strut{} }}%
      \csname LTb\endcsname%%
      \put(3612,2698){\makebox(0,0)[r]{\strut{}\footnotesize$C\sqrt{\frac{\log h}{N}}$}}%
      \csname LTb\endcsname%%
      \put(3612,2552){\makebox(0,0)[r]{\strut{} }}%
      \csname LTb\endcsname%%
      \put(3612,2405){\makebox(0,0)[r]{\strut{}\footnotesize$(1),p_1$}}%
      \csname LTb\endcsname%%
      \put(3612,2258){\makebox(0,0)[r]{\strut{}\footnotesize$(\infty),p_1$}}%
      \csname LTb\endcsname%%
      \put(3612,2112){\makebox(0,0)[r]{\strut{}\footnotesize SGCT,$p_1$}}%
      \csname LTb\endcsname%%
      \put(3612,1965){\makebox(0,0)[r]{\strut{}\footnotesize STD,$p_1$}}%
      \csname LTb\endcsname%%
      \put(48,1783){\rotatebox{-270.00}{\makebox(0,0){\strut{}$|\Phi_{h}^{(\star,\tau)} - \Phi(t))|_{\mathrm{H}^1(\Omega)}$}}}%
      \csname LTb\endcsname%%
      \put(1751,93){\makebox(0,0){\strut{}Mesh size $h$}}%
    }%
    \gplbacktext
    \put(0,0){\includegraphics[width={201.00bp},height={158.00bp}]{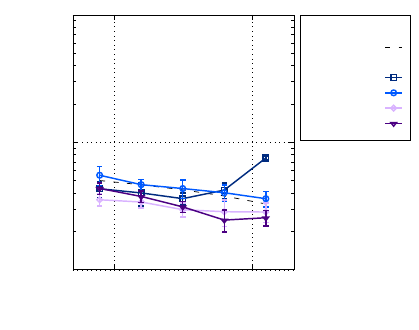}}%
    \gplfronttext
  \end{picture}%
\endgroup
\hspace*{1em}% GNUPLOT: LaTeX picture with Postscript
\begingroup
  \makeatletter
  \providecommand\color[2][]{%
    \GenericError{(gnuplot) \space\space\space\@spaces}{%
      Package color not loaded in conjunction with
      terminal option `colourtext'%
    }{See the gnuplot documentation for explanation.%
    }{Either use 'blacktext' in gnuplot or load the package
      color.sty in LaTeX.}%
    \renewcommand\color[2][]{}%
  }%
  \providecommand\includegraphics[2][]{%
    \GenericError{(gnuplot) \space\space\space\@spaces}{%
      Package graphicx or graphics not loaded%
    }{See the gnuplot documentation for explanation.%
    }{The gnuplot epslatex terminal needs graphicx.sty or graphics.sty.}%
    \renewcommand\includegraphics[2][]{}%
  }%
  \providecommand\rotatebox[2]{#2}%
  \@ifundefined{ifGPcolor}{%
    \newif\ifGPcolor
    \GPcolortrue
  }{}%
  \@ifundefined{ifGPblacktext}{%
    \newif\ifGPblacktext
    \GPblacktexttrue
  }{}%
  % define a \g@addto@macro without @ in the name:
  \let\gplgaddtomacro\g@addto@macro
  % define empty templates for all commands taking text:
  \gdef\gplbacktext{}%
  \gdef\gplfronttext{}%
  \makeatother
  \ifGPblacktext
    % no textcolor at all
    \def\colorrgb#1{}%
    \def\colorgray#1{}%
  \else
    % gray or color?
    \ifGPcolor
      \def\colorrgb#1{\color[rgb]{#1}}%
      \def\colorgray#1{\color[gray]{#1}}%
      \expandafter\def\csname LTw\endcsname{\color{white}}%
      \expandafter\def\csname LTb\endcsname{\color{black}}%
      \expandafter\def\csname LTa\endcsname{\color{black}}%
      \expandafter\def\csname LT0\endcsname{\color[rgb]{1,0,0}}%
      \expandafter\def\csname LT1\endcsname{\color[rgb]{0,1,0}}%
      \expandafter\def\csname LT2\endcsname{\color[rgb]{0,0,1}}%
      \expandafter\def\csname LT3\endcsname{\color[rgb]{1,0,1}}%
      \expandafter\def\csname LT4\endcsname{\color[rgb]{0,1,1}}%
      \expandafter\def\csname LT5\endcsname{\color[rgb]{1,1,0}}%
      \expandafter\def\csname LT6\endcsname{\color[rgb]{0,0,0}}%
      \expandafter\def\csname LT7\endcsname{\color[rgb]{1,0.3,0}}%
      \expandafter\def\csname LT8\endcsname{\color[rgb]{0.5,0.5,0.5}}%
    \else
      % gray
      \def\colorrgb#1{\color{black}}%
      \def\colorgray#1{\color[gray]{#1}}%
      \expandafter\def\csname LTw\endcsname{\color{white}}%
      \expandafter\def\csname LTb\endcsname{\color{black}}%
      \expandafter\def\csname LTa\endcsname{\color{black}}%
      \expandafter\def\csname LT0\endcsname{\color{black}}%
      \expandafter\def\csname LT1\endcsname{\color{black}}%
      \expandafter\def\csname LT2\endcsname{\color{black}}%
      \expandafter\def\csname LT3\endcsname{\color{black}}%
      \expandafter\def\csname LT4\endcsname{\color{black}}%
      \expandafter\def\csname LT5\endcsname{\color{black}}%
      \expandafter\def\csname LT6\endcsname{\color{black}}%
      \expandafter\def\csname LT7\endcsname{\color{black}}%
      \expandafter\def\csname LT8\endcsname{\color{black}}%
    \fi
  \fi
    \setlength{\unitlength}{0.0500bp}%
    \ifx\gptboxheight\undefined%
      \newlength{\gptboxheight}%
      \newlength{\gptboxwidth}%
      \newsavebox{\gptboxtext}%
    \fi%
    \setlength{\fboxrule}{0.5pt}%
    \setlength{\fboxsep}{1pt}%
    \definecolor{tbcol}{rgb}{1,1,1}%
\begin{picture}(4020.00,3160.00)%
    \gplgaddtomacro\gplbacktext{%
      \csname LTb\endcsname%%
      \put(615,559){\makebox(0,0)[r]{\strut{}\small$10^{-1}$}}%
      \csname LTb\endcsname%%
      \put(615,1783){\makebox(0,0)[r]{\strut{}\small$10^{0}$}}%
      \csname LTb\endcsname%%
      \put(615,3006){\makebox(0,0)[r]{\strut{}\small$10^{1}$}}%
      \csname LTb\endcsname%%
      \put(1088,426){\makebox(0,0){\strut{}$10^{-2}$}}%
      \csname LTb\endcsname%%
      \put(2413,426){\makebox(0,0){\strut{}$10^{-1}$}}%
    }%
    \gplgaddtomacro\gplfronttext{%
      \csname LTb\endcsname%%
      \put(3612,2845){\makebox(0,0)[r]{\strut{} }}%
      \csname LTb\endcsname%%
      \put(3612,2698){\makebox(0,0)[r]{\strut{}\footnotesize$C\sqrt{\frac{\log h}{N}}$}}%
      \csname LTb\endcsname%%
      \put(3612,2552){\makebox(0,0)[r]{\strut{} }}%
      \csname LTb\endcsname%%
      \put(3612,2405){\makebox(0,0)[r]{\strut{}\footnotesize$(1),p_3$}}%
      \csname LTb\endcsname%%
      \put(3612,2258){\makebox(0,0)[r]{\strut{}\footnotesize$(\infty),p_3$}}%
      \csname LTb\endcsname%%
      \put(3612,2112){\makebox(0,0)[r]{\strut{}\footnotesize SGCT,$p_3$}}%
      \csname LTb\endcsname%%
      \put(3612,1965){\makebox(0,0)[r]{\strut{}\footnotesize STD,$p_3$}}%
      \csname LTb\endcsname%%
      \put(48,1783){\rotatebox{-270.00}{\makebox(0,0){\strut{}$|\Phi_{h}^{(\star,\tau)} - \Phi(t))|_{\mathrm{H}^1(\Omega)}$}}}%
      \csname LTb\endcsname%%
      \put(1751,93){\makebox(0,0){\strut{}Mesh size $h$}}%
    }%
    \gplbacktext
    \put(0,0){\includegraphics[width={201.00bp},height={158.00bp}]{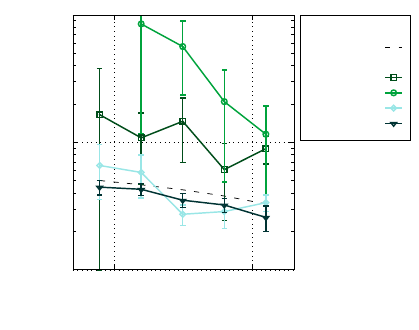}}%
    \gplfronttext
  \end{picture}%
\endgroup
}
 \end{center}
 \captionof{figure}{$\mathrm{L}^2$ norm and $\mathrm{H}^1$ semi-norm of the charge density (top) and electric potential (bottom) approximation errors as functions of the mesh size $h$, for a fixed number of particles $N=1{,}000$ at time $t=0.5$, with linear (left) and cubic B-splines (right). The mean and standard deviation are computed from $15$ realizations.}
\label{fig:2}
\end{center}

\Fabrice{We first comment on the charge density error results. The observed convergence rates of the statistical errors fully align with the theoretical bounds established in \cref{thm:0,thm:0:2,thm:1:2,prop:1}, scaling as $\mathcal{O}\big(1/\sqrt{N h^2}\big)$ for full-grid approximations and $\mathcal{O}\big(\sqrt{|\log h| / (N h)}\big)$ for sparse-grid approximations. When compared with the spatial complexity of the respective approximation spaces (\cref{prop:4}), these convergence behaviors strictly comply with the fundamental rule of PIC algorithms, where statistical noise scales inversely with the square root of the average number of particles per cell.}

\Fabrice{Concerning the electric field, several distinctive features emerge. First, assessing empirical convergence orders on the electric field proves challenging, particularly for cubic splines ($p_3$), which exhibit significant variance across independent realizations. This variability complicates formal order verification, a practical drawback also reported by other authors \cite{tranquilli22} regarding 2D electric field metrics. 
Nevertheless, for linear splines ($p_1$), all approximations exhibit an error scaling of $\mathcal{O}\big(\sqrt{|\log h| / N}\big)$. This reduced noise level, compared to the charge density, reflects the regularizing property of the discrete Poisson operator, whose inverse acts as an intrinsic spatial low-pass filter that significantly dampens high-frequency sampling fluctuations \cite{tranquilli22}. This behavior extends to cubic splines ($p_3$), with the notable exception of the hierarchical full-grid approximation, which undergoes a substantial growth in statistical noise as the grid is refined.}

It should be noted that the numerical tests presented here are restricted to two space dimensions, which is not the most favorable setting for assessing the benefits of sparse-grid methods. Indeed, the statistical noise of a PIC scheme scales as the inverse square root of the mean number of particles per cell, which involves the mesh size $h^{-d/2}$ for a full-grid approximation and to a power growing only as $|\log h|^{(d-1)/2}$ for a sparse-grid one. The ratio between these two scalings, which measures the noise-reduction gain of the sparse-grid approach at fixed particle count, grows substantially with the dimension $d$: in two dimensions the gain is already visible but moderate, whereas in three dimensions it becomes decisive, since the curse of dimensionality is dramatically more severe for full-grid methods than for their sparse-grid counterparts. Three-dimensional simulations with the SGCT-PIC method have confirmed this trend and demonstrated substantial reductions in the number of particles required to reach a prescribed accuracy~\cite{deluzet22-1,deluzet23,garrigues24, garrigues24-1,guillet24,guillet25}, an advantage that the present HSG-PIC method is expected to share and that will be assessed in a forthcoming three-dimensional implementation.

\subsubsection{Global error}
\label{sec:3.1.3}
We now investigate the global error, composed of both the grid-based and statistical components. We restrict our investigation to the examination of the charge density error, as it has been proven that it guarantees the accuracy of PIC algorithms~\cite{tranquilli22}; it is also easier to perform verifications and we have already examined the electric field error in the previous sections.

\paragraph{Manufactured solutions}
We fix the number of particles to $N=1{,}000{,}000$ and analyze the error at time $t=1.5$.
The $\mathrm{L}^2$ norm of the charge density approximation error is plotted as a function of the mesh size in~\cref{fig:3}, considering linear B-splines ($p_1$) on the left panel and cubic B-splines ($p_3$) on the right panel.

\begin{center}
\begin{center}
  \resizebox{0.9\textwidth}{!}{
  % GNUPLOT: LaTeX picture with Postscript
\begingroup
  \makeatletter
  \providecommand\color[2][]{%
    \GenericError{(gnuplot) \space\space\space\@spaces}{%
      Package color not loaded in conjunction with
      terminal option `colourtext'%
    }{See the gnuplot documentation for explanation.%
    }{Either use 'blacktext' in gnuplot or load the package
      color.sty in LaTeX.}%
    \renewcommand\color[2][]{}%
  }%
  \providecommand\includegraphics[2][]{%
    \GenericError{(gnuplot) \space\space\space\@spaces}{%
      Package graphicx or graphics not loaded%
    }{See the gnuplot documentation for explanation.%
    }{The gnuplot epslatex terminal needs graphicx.sty or graphics.sty.}%
    \renewcommand\includegraphics[2][]{}%
  }%
  \providecommand\rotatebox[2]{#2}%
  \@ifundefined{ifGPcolor}{%
    \newif\ifGPcolor
    \GPcolortrue
  }{}%
  \@ifundefined{ifGPblacktext}{%
    \newif\ifGPblacktext
    \GPblacktexttrue
  }{}%
  % define a \g@addto@macro without @ in the name:
  \let\gplgaddtomacro\g@addto@macro
  % define empty templates for all commands taking text:
  \gdef\gplbacktext{}%
  \gdef\gplfronttext{}%
  \makeatother
  \ifGPblacktext
    % no textcolor at all
    \def\colorrgb#1{}%
    \def\colorgray#1{}%
  \else
    % gray or color?
    \ifGPcolor
      \def\colorrgb#1{\color[rgb]{#1}}%
      \def\colorgray#1{\color[gray]{#1}}%
      \expandafter\def\csname LTw\endcsname{\color{white}}%
      \expandafter\def\csname LTb\endcsname{\color{black}}%
      \expandafter\def\csname LTa\endcsname{\color{black}}%
      \expandafter\def\csname LT0\endcsname{\color[rgb]{1,0,0}}%
      \expandafter\def\csname LT1\endcsname{\color[rgb]{0,1,0}}%
      \expandafter\def\csname LT2\endcsname{\color[rgb]{0,0,1}}%
      \expandafter\def\csname LT3\endcsname{\color[rgb]{1,0,1}}%
      \expandafter\def\csname LT4\endcsname{\color[rgb]{0,1,1}}%
      \expandafter\def\csname LT5\endcsname{\color[rgb]{1,1,0}}%
      \expandafter\def\csname LT6\endcsname{\color[rgb]{0,0,0}}%
      \expandafter\def\csname LT7\endcsname{\color[rgb]{1,0.3,0}}%
      \expandafter\def\csname LT8\endcsname{\color[rgb]{0.5,0.5,0.5}}%
    \else
      % gray
      \def\colorrgb#1{\color{black}}%
      \def\colorgray#1{\color[gray]{#1}}%
      \expandafter\def\csname LTw\endcsname{\color{white}}%
      \expandafter\def\csname LTb\endcsname{\color{black}}%
      \expandafter\def\csname LTa\endcsname{\color{black}}%
      \expandafter\def\csname LT0\endcsname{\color{black}}%
      \expandafter\def\csname LT1\endcsname{\color{black}}%
      \expandafter\def\csname LT2\endcsname{\color{black}}%
      \expandafter\def\csname LT3\endcsname{\color{black}}%
      \expandafter\def\csname LT4\endcsname{\color{black}}%
      \expandafter\def\csname LT5\endcsname{\color{black}}%
      \expandafter\def\csname LT6\endcsname{\color{black}}%
      \expandafter\def\csname LT7\endcsname{\color{black}}%
      \expandafter\def\csname LT8\endcsname{\color{black}}%
    \fi
  \fi
    \setlength{\unitlength}{0.0500bp}%
    \ifx\gptboxheight\undefined%
      \newlength{\gptboxheight}%
      \newlength{\gptboxwidth}%
      \newsavebox{\gptboxtext}%
    \fi%
    \setlength{\fboxrule}{0.5pt}%
    \setlength{\fboxsep}{1pt}%
    \definecolor{tbcol}{rgb}{1,1,1}%
\begin{picture}(4020.00,3160.00)%
    \gplgaddtomacro\gplbacktext{%
      \csname LTb\endcsname%%
      \put(615,879){\makebox(0,0)[r]{\strut{}\small$10^{-2}$}}%
      \csname LTb\endcsname%%
      \put(615,1943){\makebox(0,0)[r]{\strut{}\small$10^{-1}$}}%
      \csname LTb\endcsname%%
      \put(615,3006){\makebox(0,0)[r]{\strut{}\small$10^{0}$}}%
      \csname LTb\endcsname%%
      \put(1239,426){\makebox(0,0){\strut{}$10^{-2}$}}%
      \csname LTb\endcsname%%
      \put(2026,426){\makebox(0,0){\strut{}$10^{-1}$}}%
      \csname LTb\endcsname%%
      \put(2812,426){\makebox(0,0){\strut{}$\small 10^{0}$}}%
    }%
    \gplgaddtomacro\gplfronttext{%
      \csname LTb\endcsname%%
      \put(3612,2845){\makebox(0,0)[r]{\strut{}\footnotesize$C\sqrt{\frac{1}{Nh^2}}$}}%
      \csname LTb\endcsname%%
      \put(3612,2698){\makebox(0,0)[r]{\strut{} }}%
      \csname LTb\endcsname%%
      \put(3612,2552){\makebox(0,0)[r]{\strut{}\footnotesize$C\sqrt{\frac{\log h}{Nh}}$}}%
      \csname LTb\endcsname%%
      \put(3612,2405){\makebox(0,0)[r]{\strut{} }}%
      \csname LTb\endcsname%%
      \put(3612,2258){\makebox(0,0)[r]{\strut{}\footnotesize$h^{2}$}}%
      \csname LTb\endcsname%%
      \put(3612,2112){\makebox(0,0)[r]{\strut{}\footnotesize$h^{2}\log h$}}%
      \csname LTb\endcsname%%
      \put(3612,1965){\makebox(0,0)[r]{\strut{}\footnotesize$(1),p_1$}}%
      \csname LTb\endcsname%%
      \put(3612,1819){\makebox(0,0)[r]{\strut{}\footnotesize$(\infty),p_1$}}%
      \csname LTb\endcsname%%
      \put(3612,1672){\makebox(0,0)[r]{\strut{}\footnotesize SGCT,$p_1$}}%
      \csname LTb\endcsname%%
      \put(3612,1525){\makebox(0,0)[r]{\strut{}\footnotesize STD,$p_1$}}%
      \csname LTb\endcsname%%
      \put(48,1783){\rotatebox{-270.00}{\makebox(0,0){\strut{}$\|\rho_h^{(\star),\tau} - \rho(t))\|_{\mathrm{L}^2(\Omega)}$}}}%
      \csname LTb\endcsname%%
      \put(1751,93){\makebox(0,0){\strut{}Mesh size $h$}}%
    }%
    \gplbacktext
    \put(0,0){\includegraphics[width={201.00bp},height={158.00bp}]{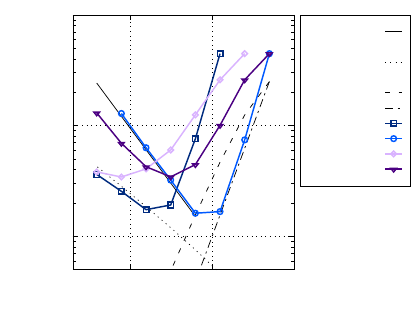}}%
    \gplfronttext
  \end{picture}%
\endgroup
\hspace*{1em}% GNUPLOT: LaTeX picture with Postscript
\begingroup
  \makeatletter
  \providecommand\color[2][]{%
    \GenericError{(gnuplot) \space\space\space\@spaces}{%
      Package color not loaded in conjunction with
      terminal option `colourtext'%
    }{See the gnuplot documentation for explanation.%
    }{Either use 'blacktext' in gnuplot or load the package
      color.sty in LaTeX.}%
    \renewcommand\color[2][]{}%
  }%
  \providecommand\includegraphics[2][]{%
    \GenericError{(gnuplot) \space\space\space\@spaces}{%
      Package graphicx or graphics not loaded%
    }{See the gnuplot documentation for explanation.%
    }{The gnuplot epslatex terminal needs graphicx.sty or graphics.sty.}%
    \renewcommand\includegraphics[2][]{}%
  }%
  \providecommand\rotatebox[2]{#2}%
  \@ifundefined{ifGPcolor}{%
    \newif\ifGPcolor
    \GPcolortrue
  }{}%
  \@ifundefined{ifGPblacktext}{%
    \newif\ifGPblacktext
    \GPblacktexttrue
  }{}%
  % define a \g@addto@macro without @ in the name:
  \let\gplgaddtomacro\g@addto@macro
  % define empty templates for all commands taking text:
  \gdef\gplbacktext{}%
  \gdef\gplfronttext{}%
  \makeatother
  \ifGPblacktext
    % no textcolor at all
    \def\colorrgb#1{}%
    \def\colorgray#1{}%
  \else
    % gray or color?
    \ifGPcolor
      \def\colorrgb#1{\color[rgb]{#1}}%
      \def\colorgray#1{\color[gray]{#1}}%
      \expandafter\def\csname LTw\endcsname{\color{white}}%
      \expandafter\def\csname LTb\endcsname{\color{black}}%
      \expandafter\def\csname LTa\endcsname{\color{black}}%
      \expandafter\def\csname LT0\endcsname{\color[rgb]{1,0,0}}%
      \expandafter\def\csname LT1\endcsname{\color[rgb]{0,1,0}}%
      \expandafter\def\csname LT2\endcsname{\color[rgb]{0,0,1}}%
      \expandafter\def\csname LT3\endcsname{\color[rgb]{1,0,1}}%
      \expandafter\def\csname LT4\endcsname{\color[rgb]{0,1,1}}%
      \expandafter\def\csname LT5\endcsname{\color[rgb]{1,1,0}}%
      \expandafter\def\csname LT6\endcsname{\color[rgb]{0,0,0}}%
      \expandafter\def\csname LT7\endcsname{\color[rgb]{1,0.3,0}}%
      \expandafter\def\csname LT8\endcsname{\color[rgb]{0.5,0.5,0.5}}%
    \else
      % gray
      \def\colorrgb#1{\color{black}}%
      \def\colorgray#1{\color[gray]{#1}}%
      \expandafter\def\csname LTw\endcsname{\color{white}}%
      \expandafter\def\csname LTb\endcsname{\color{black}}%
      \expandafter\def\csname LTa\endcsname{\color{black}}%
      \expandafter\def\csname LT0\endcsname{\color{black}}%
      \expandafter\def\csname LT1\endcsname{\color{black}}%
      \expandafter\def\csname LT2\endcsname{\color{black}}%
      \expandafter\def\csname LT3\endcsname{\color{black}}%
      \expandafter\def\csname LT4\endcsname{\color{black}}%
      \expandafter\def\csname LT5\endcsname{\color{black}}%
      \expandafter\def\csname LT6\endcsname{\color{black}}%
      \expandafter\def\csname LT7\endcsname{\color{black}}%
      \expandafter\def\csname LT8\endcsname{\color{black}}%
    \fi
  \fi
    \setlength{\unitlength}{0.0500bp}%
    \ifx\gptboxheight\undefined%
      \newlength{\gptboxheight}%
      \newlength{\gptboxwidth}%
      \newsavebox{\gptboxtext}%
    \fi%
    \setlength{\fboxrule}{0.5pt}%
    \setlength{\fboxsep}{1pt}%
    \definecolor{tbcol}{rgb}{1,1,1}%
\begin{picture}(4020.00,3160.00)%
    \gplgaddtomacro\gplbacktext{%
      \csname LTb\endcsname%%
      \put(615,879){\makebox(0,0)[r]{\strut{}\small$10^{-2}$}}%
      \csname LTb\endcsname%%
      \put(615,1943){\makebox(0,0)[r]{\strut{}\small$10^{-1}$}}%
      \csname LTb\endcsname%%
      \put(615,3006){\makebox(0,0)[r]{\strut{}\small$10^{0}$}}%
      \csname LTb\endcsname%%
      \put(1239,426){\makebox(0,0){\strut{}$10^{-2}$}}%
      \csname LTb\endcsname%%
      \put(2026,426){\makebox(0,0){\strut{}$10^{-1}$}}%
      \csname LTb\endcsname%%
      \put(2812,426){\makebox(0,0){\strut{}$\small 10^{0}$}}%
    }%
    \gplgaddtomacro\gplfronttext{%
      \csname LTb\endcsname%%
      \put(3612,2845){\makebox(0,0)[r]{\strut{}\footnotesize$C\sqrt{\frac{1}{Nh^2}}$}}%
      \csname LTb\endcsname%%
      \put(3612,2698){\makebox(0,0)[r]{\strut{} }}%
      \csname LTb\endcsname%%
      \put(3612,2552){\makebox(0,0)[r]{\strut{}\footnotesize$C\sqrt{\frac{\log h}{Nh}}$}}%
      \csname LTb\endcsname%%
      \put(3612,2405){\makebox(0,0)[r]{\strut{} }}%
      \csname LTb\endcsname%%
      \put(3612,2258){\makebox(0,0)[r]{\strut{}\footnotesize$h^{4}$}}%
      \csname LTb\endcsname%%
      \put(3612,2112){\makebox(0,0)[r]{\strut{}\footnotesize$h^{4}\log h$}}%
      \csname LTb\endcsname%%
      \put(3612,1965){\makebox(0,0)[r]{\strut{}\footnotesize$(1),p_3$}}%
      \csname LTb\endcsname%%
      \put(3612,1819){\makebox(0,0)[r]{\strut{}\footnotesize$(\infty),p_3$}}%
      \csname LTb\endcsname%%
      \put(3612,1672){\makebox(0,0)[r]{\strut{}\footnotesize SGCT,$p_3$}}%
      \csname LTb\endcsname%%
      \put(3612,1525){\makebox(0,0)[r]{\strut{}\footnotesize STD,$p_3$}}%
      \csname LTb\endcsname%%
      \put(48,1783){\rotatebox{-270.00}{\makebox(0,0){\strut{}$\|\rho_h^{(\star),\tau} - \rho(t))\|_{\mathrm{L}^2(\Omega)}$}}}%
      \csname LTb\endcsname%%
      \put(1751,93){\makebox(0,0){\strut{}Mesh size $h$}}%
    }%
    \gplbacktext
    \put(0,0){\includegraphics[width={201.00bp},height={158.00bp}]{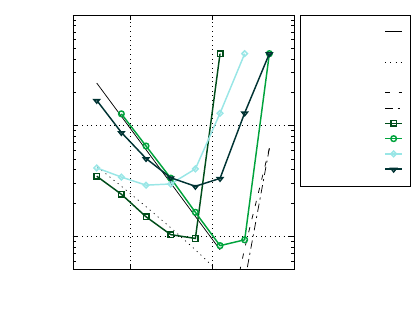}}%
    \gplfronttext
  \end{picture}%
\endgroup
}
 \end{center}
 \captionof{figure}{$\mathrm{L}^2$ norm of the charge density approximation error for the manufactured solutions as function of the mesh size $h$, for a fixed number of particles $N=1{,}000{,}000$ at time $t=1.5$, with linear (left) and cubic B-splines (right).}
\label{fig:3}
\end{center}

The convergence of the error is decomposed into two phases: one dominated by the grid-based error; and one dominated by the statistical error. 

In the first phase, the global error decreases with the mesh size and we can observe the theoretical convergence orders as experienced in~\cref{sec:3.1.1}. That is, the HSG errors scale as $\mathcal{O}(h^{p+1})$ for the two approximation spaces ($(1)$ and $(\infty)$). Concerning the STD and SGCT errors, at least for linear splines ($p_1$), the theoretical bounds of~\cref{prop:1} are recovered, that is $\mathcal{O}(h^{2})$ for STD and $\mathcal{O}(h^{2}\log h)$ for SGCT. We cannot conclude for the cubic splines ($p_3$) as the global error is quickly dominated by the statistical error. %For all spline degrees, the statistical error is significantly larger for the STD and SGCT methods than the HSG ones.

The global error then deviates from the theoretical orders, reaches a plateau and enters in the second phase, where it increases following the statistical growth examined in~\cref{sec:3.1.2}. %For significant mesh sizes where the grid heating constraint is verified, i.e. $h\leq 1/60$, the hierarchical sparse-grid approximation provide the smallest error. 

The electron density is plotted at time $t=0.5$ in~\cref{fig:6} for the different methods with mesh size $h=2^{-7}$, $N=500{,}000$ particles and linear ($p_1$) B-splines. 

\begin{center}
\begin{minipage}{\textwidth}
  \centering
  \begin{minipage}[]{0.31\textwidth}
\centering {\footnotesize STD-PIC} \\  
 \includegraphics[width=1\textwidth]{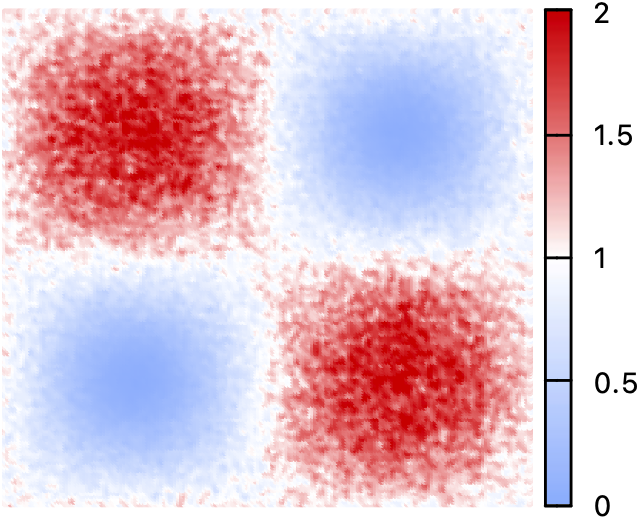}   
\end{minipage} 
  \begin{minipage}[]{0.31\textwidth}
\centering {\footnotesize HSG-$(\infty)$-PIC} \\  
 \includegraphics[width=1\textwidth]{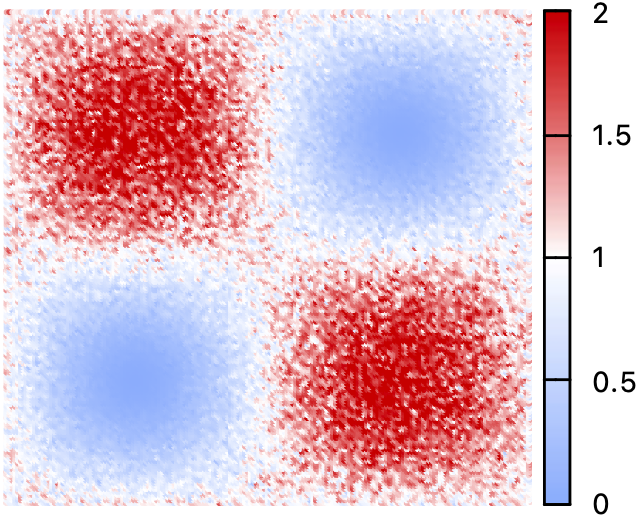}   
\end{minipage} \\ \vspace{1em}
\begin{minipage}[]{0.31\textwidth}
\centering {\footnotesize SGCT-PIC} \\ 
 \includegraphics[width=1\textwidth]{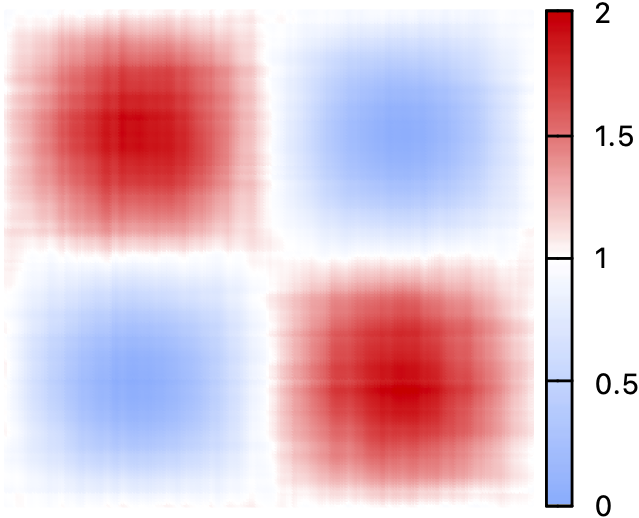}   
\end{minipage} 
\begin{minipage}[]{0.31\textwidth}
\centering {\footnotesize HSG-$(1)$-PIC} \\
  \includegraphics[width=1\textwidth]{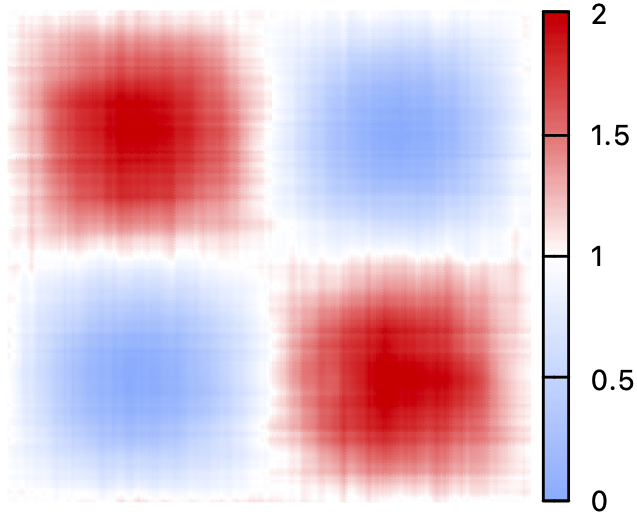}   
\end{minipage}
\captionof{figure}{Manufactured solutions: electron density at $t=0.5$, with mesh size $h=2^{-7}$, $N=500{,}000$ particles and linear ($p_1$) B-splines.} \label{fig:6}
\end{minipage}
\end{center}

\paragraph{Diocotron instability}
We first investigate the charge density approximation error at initial time $t=0$, for which the exact solution is known. The $\mathrm{L}^2$ norm approximation error is plotted as a function of the mesh size in~\cref{fig:4}, considering linear B-splines ($p_1$) on the left panel and cubic B-splines ($p_3$) on the right panel.

\begin{center}
\begin{center}
  \resizebox{0.9\textwidth}{!}{
  % GNUPLOT: LaTeX picture with Postscript
\begingroup
  \makeatletter
  \providecommand\color[2][]{%
    \GenericError{(gnuplot) \space\space\space\@spaces}{%
      Package color not loaded in conjunction with
      terminal option `colourtext'%
    }{See the gnuplot documentation for explanation.%
    }{Either use 'blacktext' in gnuplot or load the package
      color.sty in LaTeX.}%
    \renewcommand\color[2][]{}%
  }%
  \providecommand\includegraphics[2][]{%
    \GenericError{(gnuplot) \space\space\space\@spaces}{%
      Package graphicx or graphics not loaded%
    }{See the gnuplot documentation for explanation.%
    }{The gnuplot epslatex terminal needs graphicx.sty or graphics.sty.}%
    \renewcommand\includegraphics[2][]{}%
  }%
  \providecommand\rotatebox[2]{#2}%
  \@ifundefined{ifGPcolor}{%
    \newif\ifGPcolor
    \GPcolortrue
  }{}%
  \@ifundefined{ifGPblacktext}{%
    \newif\ifGPblacktext
    \GPblacktexttrue
  }{}%
  % define a \g@addto@macro without @ in the name:
  \let\gplgaddtomacro\g@addto@macro
  % define empty templates for all commands taking text:
  \gdef\gplbacktext{}%
  \gdef\gplfronttext{}%
  \makeatother
  \ifGPblacktext
    % no textcolor at all
    \def\colorrgb#1{}%
    \def\colorgray#1{}%
  \else
    % gray or color?
    \ifGPcolor
      \def\colorrgb#1{\color[rgb]{#1}}%
      \def\colorgray#1{\color[gray]{#1}}%
      \expandafter\def\csname LTw\endcsname{\color{white}}%
      \expandafter\def\csname LTb\endcsname{\color{black}}%
      \expandafter\def\csname LTa\endcsname{\color{black}}%
      \expandafter\def\csname LT0\endcsname{\color[rgb]{1,0,0}}%
      \expandafter\def\csname LT1\endcsname{\color[rgb]{0,1,0}}%
      \expandafter\def\csname LT2\endcsname{\color[rgb]{0,0,1}}%
      \expandafter\def\csname LT3\endcsname{\color[rgb]{1,0,1}}%
      \expandafter\def\csname LT4\endcsname{\color[rgb]{0,1,1}}%
      \expandafter\def\csname LT5\endcsname{\color[rgb]{1,1,0}}%
      \expandafter\def\csname LT6\endcsname{\color[rgb]{0,0,0}}%
      \expandafter\def\csname LT7\endcsname{\color[rgb]{1,0.3,0}}%
      \expandafter\def\csname LT8\endcsname{\color[rgb]{0.5,0.5,0.5}}%
    \else
      % gray
      \def\colorrgb#1{\color{black}}%
      \def\colorgray#1{\color[gray]{#1}}%
      \expandafter\def\csname LTw\endcsname{\color{white}}%
      \expandafter\def\csname LTb\endcsname{\color{black}}%
      \expandafter\def\csname LTa\endcsname{\color{black}}%
      \expandafter\def\csname LT0\endcsname{\color{black}}%
      \expandafter\def\csname LT1\endcsname{\color{black}}%
      \expandafter\def\csname LT2\endcsname{\color{black}}%
      \expandafter\def\csname LT3\endcsname{\color{black}}%
      \expandafter\def\csname LT4\endcsname{\color{black}}%
      \expandafter\def\csname LT5\endcsname{\color{black}}%
      \expandafter\def\csname LT6\endcsname{\color{black}}%
      \expandafter\def\csname LT7\endcsname{\color{black}}%
      \expandafter\def\csname LT8\endcsname{\color{black}}%
    \fi
  \fi
    \setlength{\unitlength}{0.0500bp}%
    \ifx\gptboxheight\undefined%
      \newlength{\gptboxheight}%
      \newlength{\gptboxwidth}%
      \newsavebox{\gptboxtext}%
    \fi%
    \setlength{\fboxrule}{0.5pt}%
    \setlength{\fboxsep}{1pt}%
    \definecolor{tbcol}{rgb}{1,1,1}%
\begin{picture}(4020.00,3160.00)%
    \gplgaddtomacro\gplbacktext{%
      \csname LTb\endcsname%%
      \put(615,559){\makebox(0,0)[r]{\strut{}\small$10^{-2}$}}%
      \csname LTb\endcsname%%
      \put(615,1783){\makebox(0,0)[r]{\strut{}\small$10^{-1}$}}%
      \csname LTb\endcsname%%
      \put(615,3006){\makebox(0,0)[r]{\strut{}\small$10^{0}$}}%
      \csname LTb\endcsname%%
      \put(689,426){\makebox(0,0){\strut{}$10^{-3}$}}%
      \csname LTb\endcsname%%
      \put(1546,426){\makebox(0,0){\strut{}$10^{-2}$}}%
      \csname LTb\endcsname%%
      \put(2403,426){\makebox(0,0){\strut{}$10^{-1}$}}%
    }%
    \gplgaddtomacro\gplfronttext{%
      \csname LTb\endcsname%%
      \put(3612,2845){\makebox(0,0)[r]{\strut{}\footnotesize$Ch^{2}$}}%
      \csname LTb\endcsname%%
      \put(3612,2698){\makebox(0,0)[r]{\strut{}\footnotesize$Ch^{2}\log h$}}%
      \csname LTb\endcsname%%
      \put(3612,2552){\makebox(0,0)[r]{\strut{}\footnotesize$(1),p_1$}}%
      \csname LTb\endcsname%%
      \put(3612,2405){\makebox(0,0)[r]{\strut{}\footnotesize$(\infty),p_1$}}%
      \csname LTb\endcsname%%
      \put(3612,2258){\makebox(0,0)[r]{\strut{}\footnotesize SGCT,$p_1$}}%
      \csname LTb\endcsname%%
      \put(3612,2112){\makebox(0,0)[r]{\strut{}\footnotesize STD,$p_1$}}%
      \csname LTb\endcsname%%
      \put(48,1783){\rotatebox{-270.00}{\makebox(0,0){\strut{}$\|\rho_h^{(\star),\tau} - \rho(t))\|_{\mathrm{L}^2(\Omega)}$}}}%
      \csname LTb\endcsname%%
      \put(1751,93){\makebox(0,0){\strut{}Mesh size $h$}}%
    }%
    \gplbacktext
    \put(0,0){\includegraphics[width={201.00bp},height={158.00bp}]{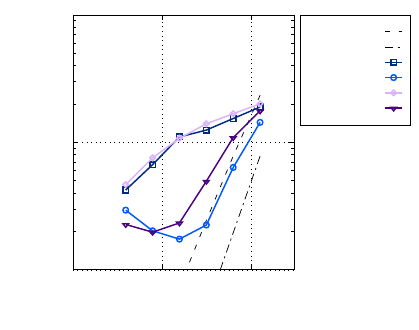}}%
    \gplfronttext
  \end{picture}%
\endgroup
\hspace*{1em}% GNUPLOT: LaTeX picture with Postscript
\begingroup
  \makeatletter
  \providecommand\color[2][]{%
    \GenericError{(gnuplot) \space\space\space\@spaces}{%
      Package color not loaded in conjunction with
      terminal option `colourtext'%
    }{See the gnuplot documentation for explanation.%
    }{Either use 'blacktext' in gnuplot or load the package
      color.sty in LaTeX.}%
    \renewcommand\color[2][]{}%
  }%
  \providecommand\includegraphics[2][]{%
    \GenericError{(gnuplot) \space\space\space\@spaces}{%
      Package graphicx or graphics not loaded%
    }{See the gnuplot documentation for explanation.%
    }{The gnuplot epslatex terminal needs graphicx.sty or graphics.sty.}%
    \renewcommand\includegraphics[2][]{}%
  }%
  \providecommand\rotatebox[2]{#2}%
  \@ifundefined{ifGPcolor}{%
    \newif\ifGPcolor
    \GPcolortrue
  }{}%
  \@ifundefined{ifGPblacktext}{%
    \newif\ifGPblacktext
    \GPblacktexttrue
  }{}%
  % define a \g@addto@macro without @ in the name:
  \let\gplgaddtomacro\g@addto@macro
  % define empty templates for all commands taking text:
  \gdef\gplbacktext{}%
  \gdef\gplfronttext{}%
  \makeatother
  \ifGPblacktext
    % no textcolor at all
    \def\colorrgb#1{}%
    \def\colorgray#1{}%
  \else
    % gray or color?
    \ifGPcolor
      \def\colorrgb#1{\color[rgb]{#1}}%
      \def\colorgray#1{\color[gray]{#1}}%
      \expandafter\def\csname LTw\endcsname{\color{white}}%
      \expandafter\def\csname LTb\endcsname{\color{black}}%
      \expandafter\def\csname LTa\endcsname{\color{black}}%
      \expandafter\def\csname LT0\endcsname{\color[rgb]{1,0,0}}%
      \expandafter\def\csname LT1\endcsname{\color[rgb]{0,1,0}}%
      \expandafter\def\csname LT2\endcsname{\color[rgb]{0,0,1}}%
      \expandafter\def\csname LT3\endcsname{\color[rgb]{1,0,1}}%
      \expandafter\def\csname LT4\endcsname{\color[rgb]{0,1,1}}%
      \expandafter\def\csname LT5\endcsname{\color[rgb]{1,1,0}}%
      \expandafter\def\csname LT6\endcsname{\color[rgb]{0,0,0}}%
      \expandafter\def\csname LT7\endcsname{\color[rgb]{1,0.3,0}}%
      \expandafter\def\csname LT8\endcsname{\color[rgb]{0.5,0.5,0.5}}%
    \else
      % gray
      \def\colorrgb#1{\color{black}}%
      \def\colorgray#1{\color[gray]{#1}}%
      \expandafter\def\csname LTw\endcsname{\color{white}}%
      \expandafter\def\csname LTb\endcsname{\color{black}}%
      \expandafter\def\csname LTa\endcsname{\color{black}}%
      \expandafter\def\csname LT0\endcsname{\color{black}}%
      \expandafter\def\csname LT1\endcsname{\color{black}}%
      \expandafter\def\csname LT2\endcsname{\color{black}}%
      \expandafter\def\csname LT3\endcsname{\color{black}}%
      \expandafter\def\csname LT4\endcsname{\color{black}}%
      \expandafter\def\csname LT5\endcsname{\color{black}}%
      \expandafter\def\csname LT6\endcsname{\color{black}}%
      \expandafter\def\csname LT7\endcsname{\color{black}}%
      \expandafter\def\csname LT8\endcsname{\color{black}}%
    \fi
  \fi
    \setlength{\unitlength}{0.0500bp}%
    \ifx\gptboxheight\undefined%
      \newlength{\gptboxheight}%
      \newlength{\gptboxwidth}%
      \newsavebox{\gptboxtext}%
    \fi%
    \setlength{\fboxrule}{0.5pt}%
    \setlength{\fboxsep}{1pt}%
    \definecolor{tbcol}{rgb}{1,1,1}%
\begin{picture}(4020.00,3160.00)%
    \gplgaddtomacro\gplbacktext{%
      \csname LTb\endcsname%%
      \put(615,559){\makebox(0,0)[r]{\strut{}\small$10^{-2}$}}%
      \csname LTb\endcsname%%
      \put(615,1783){\makebox(0,0)[r]{\strut{}\small$10^{-1}$}}%
      \csname LTb\endcsname%%
      \put(615,3006){\makebox(0,0)[r]{\strut{}\small$10^{0}$}}%
      \csname LTb\endcsname%%
      \put(689,426){\makebox(0,0){\strut{}$10^{-3}$}}%
      \csname LTb\endcsname%%
      \put(1546,426){\makebox(0,0){\strut{}$10^{-2}$}}%
      \csname LTb\endcsname%%
      \put(2403,426){\makebox(0,0){\strut{}$10^{-1}$}}%
    }%
    \gplgaddtomacro\gplfronttext{%
      \csname LTb\endcsname%%
      \put(3612,2845){\makebox(0,0)[r]{\strut{}\footnotesize$Ch^{4}$}}%
      \csname LTb\endcsname%%
      \put(3612,2698){\makebox(0,0)[r]{\strut{}\footnotesize$Ch^{4}\log h$}}%
      \csname LTb\endcsname%%
      \put(3612,2552){\makebox(0,0)[r]{\strut{}\footnotesize$(1),p_3$}}%
      \csname LTb\endcsname%%
      \put(3612,2405){\makebox(0,0)[r]{\strut{}\footnotesize$(\infty),p_3$}}%
      \csname LTb\endcsname%%
      \put(3612,2258){\makebox(0,0)[r]{\strut{}\footnotesize SGCT,$p_3$}}%
      \csname LTb\endcsname%%
      \put(3612,2112){\makebox(0,0)[r]{\strut{}\footnotesize STD,$p_3$}}%
      \csname LTb\endcsname%%
      \put(48,1783){\rotatebox{-270.00}{\makebox(0,0){\strut{}$\|\rho_h^{(\star),\tau} - \rho(t))\|_{\mathrm{L}^2(\Omega)}$}}}%
      \csname LTb\endcsname%%
      \put(1751,93){\makebox(0,0){\strut{}Mesh size $h$}}%
    }%
    \gplbacktext
    \put(0,0){\includegraphics[width={201.00bp},height={158.00bp}]{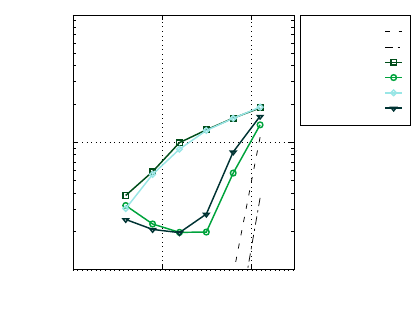}}%
    \gplfronttext
  \end{picture}%
\endgroup
}
 \end{center}
 \captionof{figure}{$\mathrm{L}^2$ norm of the charge density approximation error for the diocotron instability as function of the mesh size $h$, for a fixed number of particles $N=1{,}000{,}000$ at time $t=0$, with linear (left) and cubic B-splines (right).}
\label{fig:4}
\end{center}

No significant difference is observed between the errors obtained with linear and cubic B-splines, confirming that a higher polynomial degree brings no accuracy improvement when the solution \Fabrice{is not regular enough}. By contrast, increasing
the resolution of the approximation space, from $V^{(1)}_h$ to $V^{(\infty)}_h$, which amounts to introducing a finer partition and additional basis functions; does improve the accuracy of the computed solution. This is consistent with the expected behavior of the bias for non-smooth solutions, for which the convergence is governed by the resolution of the mesh rather than by the polynomial degree. In both cases, the optimal theoretical convergence rates are not achieved, neither for the sparse-grid nor for the full-grid approximation. This observation calls for the development of local refinement strategies, in which degrees of freedom are concentrated where the solution presents sharp gradients rather than distributed uniformly over the whole domain. In this respect, the present work, which brings sparse-grid particle methods into the hierarchical Galerkin framework, is a preparatory step toward such adaptive methods: the hierarchical structure of the approximation space naturally provides local error indicators through the decay of the hierarchical surplus, opening the way to spatially adaptive sparse-grid PIC schemes.

We now investigate the evolution of the diocotron instability. In order for the debye length to be resolved, the mesh shall meet the condition $h\leq 1/60$. We fix the mesh size to $h=2^{-8}$ for the full-grid methods and since sparse-grid truncations deteriorate accuracy for non-smooth solutions, we set the mesh size to $h=2^{-10}$ for sparse-grid methods. The number of particles is set so that the average number of particles per cell is $100$ for the STD- and SGCT-PIC methods, which results in $N=2{,}621{,}440$ (STD, HSG-$(\infty)$) and $N=1{,}187{,}840$ (SGCT), and amounts for equivalent statistical noise between the methods. We also set $N=1{,}187{,}840$ for the sparse-grid HSG-PIC method.

The electron density is plotted at time $t=35$ for the different methods in~\cref{fig:5}. 
\begin{center}
\begin{minipage}{\textwidth}
  \centering
  \begin{minipage}[]{0.31\textwidth}
\centering {\footnotesize STD-PIC} \\  
 \includegraphics[width=1\textwidth]{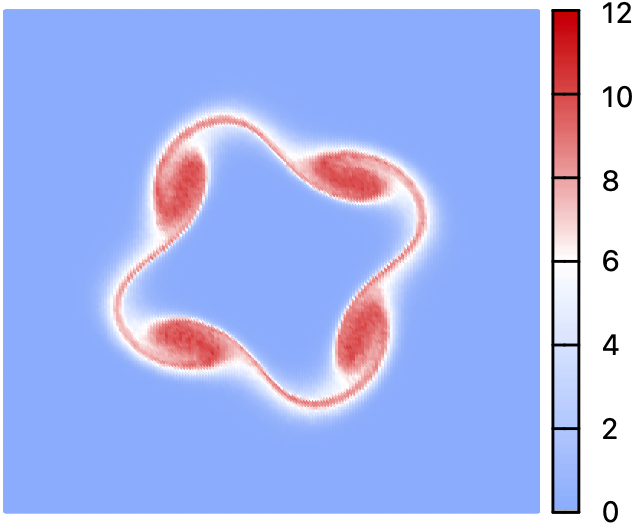}   
\end{minipage} 
  \begin{minipage}[]{0.31\textwidth}
\centering {\footnotesize HSG-$(\infty)$-PIC} \\  
 \includegraphics[width=1\textwidth]{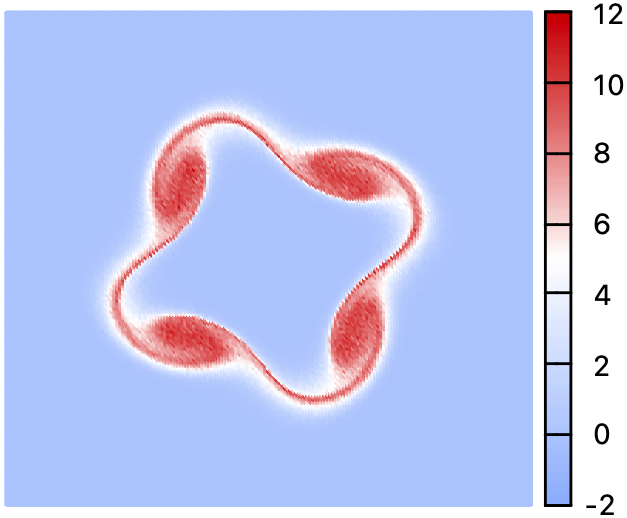}   
\end{minipage} \\ \vspace{0.5em}
\begin{minipage}[]{0.31\textwidth}
\centering {\footnotesize SGCT-PIC} \\ 
 \includegraphics[width=1\textwidth]{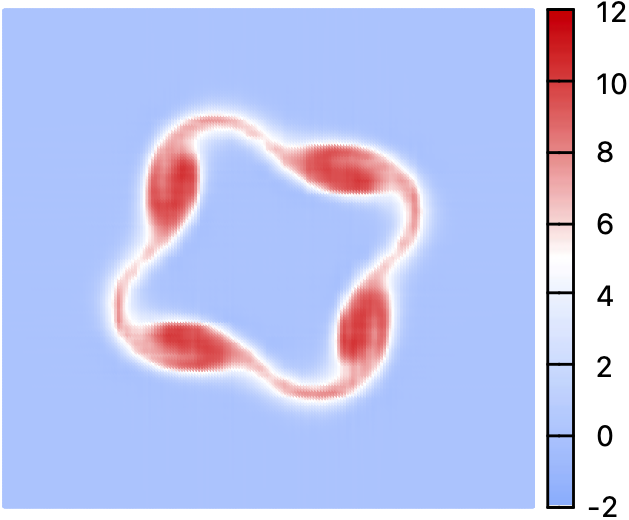}   
\end{minipage} 
\begin{minipage}[]{0.31\textwidth}
\centering {\footnotesize HSG-$(1)$-PIC} \\
  \includegraphics[width=1\textwidth]{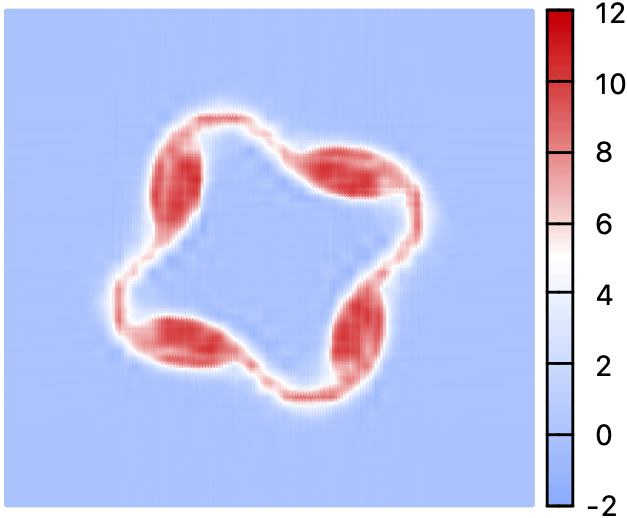}   
\end{minipage}
\captionof{figure}{Electron density at $t=35$, with linear ($p_1$) B-splines.} \label{fig:5}
\end{minipage}
\end{center}

The accuracy of the SGCT and HSG-$(1)$ approximations are comparable, providing a fair representation of the STD and HSG-$(\infty)$ ones, which are very close from each other.
We observe that unlike the full-grid approximations, sparse-grid approximations do not preserve the non-negativity of the solution. 

\subsection{Performance}
\label{sec:3.2}
We now investigate the performance of the method and compare it to the STD- and SGCT-PIC algorithms. The results shall provide insights into the computational benefits of the approach but are not necessarily representative of the real gains to expect in large plasma applications. As demonstrated in~\cite{deluzet23}, the gains of sparse-grid approximations are significantly larger in three-dimensional configurations with a dedicated implementation and  even further for parallel implementations.

\paragraph{Settings}
We restrict our study to sequential executions on a single CPU core. All experiments are performed on a single Apple M4 core with $32\,\mathrm{GB}$ of RAM. The methods are implemented in a code developed in Python and C++\footnote{The code is available at \url{https://gitlab.inria.fr/cguillet/sg-pic}}. The linear systems are solved either using a Cholesky factorization from the \texttt{scipy.sparse} library~\cite{scipy20} for the SGCT- and HSG-PIC algorithms, or using a flexible GMRES solver with an algebraic multigrid preconditioner from the \texttt{PETSc} library~\cite{petsc-user-ref} for the STD-PIC algorithm.

\paragraph{Manufactured solutions}
We consider the manufactured solution test case for which we know the exact solution at all time, which makes it possible to have a precise assessment of the approximation error. For this benchmark, the Debye length is resolved provided the mesh size satisfies $h \leq 1/60$.
 We compare the computational costs of the different methods in terms of algorithmic complexity, measured by $N$, the number of particles and $|N_h|$ the number of mesh nodes, and wall-clock execution time. 

We report in~\cref{tab:1} the computational costs required to reduce the $\mathrm{L}^2$ norm approximation error of the charge density at time $t=0.3$ below the threshold $\varepsilon=10^{-2}$. For each configuration, we report the time execution per time step, averaged over the $15$ realizations, for the three most computationally intensive steps of the algorithm: the charge density approximation (projection), the field interpolation at the particle positions (interpolation), and the particle advance (push).  The cost of the resolution of Poisson equation is negligible here.

\begin{table}[ht!]
\centering
\caption{Computational costs required to achieve a $\mathrm{L}^2$ norm error of
$\varepsilon=10^{-2}$ at time $t^\kappa=0.3$ for the manufactured solution.
Complexity and mean execution time per iteration (averaged over 15 iterations).}
\label{tab:1}
\setlength{\tabcolsep}{4pt}
\begin{tabular}{c|c|cc|cc|c|ccc}
\hline
method &
$\|\rho_h^{(\star),\kappa}-\rho(t^\kappa)\|_{2}$ &
$h$ & $p$ &
$N$ & $|N_h|$ &
\multicolumn{4}{c}{time per iteration [s]}
\\
\cline{7-10}
&
&
&
&
&
&
total & projection & interpolation & push
\\
\hline

STD
&0.0094
&$2^{-6}$&1
&28\,672\,000&4\,096
&4.08&1.67&0.08&1.90
\\

HSG-$(\infty)$
&0.0097
&$2^{-6}$&1
&40\,960\,000&4\,096
&29.4&20.1&6.4&3.09
\\

SGCT
&0.0100
&$2^{-7}$&1
&17\,920\,000&2\,560
&15.1&13.3&0.63&1.17
\\

SGCT
&0.0099
&$2^{-6}$&3
&3\,264\,000&1\,088
&2.89&2.48&0.28&0.20
\\

HSG-$(1)$
&0.0095
&$2^{-6}$&1
&2\,816\,000&256
&0.94&0.49&0.24&0.18
\\

\hline
\end{tabular}

\end{table}
The algorithms achieve the targeted accuracy $\varepsilon=10^{-2}$, but at substantially different computational costs. 
The sparse-grid configurations require significantly fewer particles and mesh nodes than the full-grid STD method, confirming the theoretical
complexity reduction established in \cref{prop:4}, even in the two-dimensional setting considered here, which is the least favorable for sparse-grid methods and does not yet benefit from a dedicated high-performance implementation. The most efficient method is HSG-$(1)$: it requires approximately ten times fewer particles than STD and sixteen times fewer mesh nodes, while running more than four times faster. Compared to the best SGCT configuration (cubic splines, $p = 3$), HSG-$(1)$ uses a comparable number of particles but $\times 4$ fewer mesh nodes, and is approximately three times faster overall, the gain being concentrated in the projection and interpolation steps. Conversely, HSG-$(\infty)$ is the slowest configuration despite sharing the same mesh resolution as STD: its larger approximation space induces a higher cost per particle in the projection and interpolation steps, which dominate its total runtime.

\section{Conclusion and Perspectives}
\label{sec:4}
In this paper, we have introduced the hierarchical sparse-grid particle-in-cell (HSG-PIC) method, an explicit-in-time particle approximation for the numerical simulation of the Vlasov--Poisson system. Conceptually distinct from both standard cell-based PIC and sparse-grid combination technique (SGCT) frameworks, the HSG-PIC method replaces the traditional charge deposition step with a direct Galerkin projection of the raw Monte Carlo density estimator onto a hierarchical spline approximation space, within a variational formulation of the field equations.

The main contribution of this work is a probabilistic error analysis. By decomposing the numerical error into a deterministic, grid-based bias and a stochastic variance component, we have established convergence rates for both contributions: for an approximation space spanned by B-splines of degree $p$ in $d$ dimensions, the $\mathrm{L}^2$-norm of the bias scales as $\mathcal{O}\big(h^{p+1}|\log h|^{d-1}\big)$, while the statistical error scales in the $\mathrm{L}^1$-norm as $\mathcal{O}\big(|\log h|^{(d-1)/2}(Nh)^{-1/2}\big)$. These estimates show
that HSG-PIC attains an asymptotic accuracy comparable to that of high-order SGCT-PIC methods \cite{deluzet22-1,deluzet23}, within a unified variational framework.

Numerical experiments, ranging from manufactured solutions to the highly anisotropic diocotron instability, corroborate the theoretical error bounds. The results indicate that the hierarchical formulation captures fine-scale kinetic structures and exhibits good energy-conservation behavior, while substantially reducing the number of macro-particles required to reach a prescribed accuracy, thereby mitigating the computational cost.

The present tests indicate that HSG-PIC matches the accuracy of state-of-the-art SGCT-PIC methods; a definitive assessment of its computational efficiency, however, requires dedicated developments. Parallel optimizations and a high-performance three-dimensional implementation on modern CPU/GPU architectures remain to be addressed in order to establish fair runtime comparisons with existing codes.

Beyond efficiency considerations, the variational and hierarchical structure of the HSG-PIC approach offers two extensions that are difficult to accommodate within the SGCT framework. First, the hierarchical increments provide
natural local error indicators, opening the way to spatially adaptive
mesh refinement for the resolution of localized, non-smooth structures, a
limitation of grid-combination methods identified in precedent works~\cite{deluzet22}. Second, the
formulation extends naturally to non-rectangular geometries through
B-spline-based geometric mappings, as proposed in~\cite{guillet25-1}. Both extensions are under active
investigation and will be presented in a forthcoming paper.

\section*{Acknowledgements}
%This work has been carried out within the framework of the EUROfusion Consortium, funded by the European Union via the
%Euratom Research and Training Programme (Grant Agreement No 101052200 — EUROfusion). 
%Views and opinions expressed are however those of the author only and do not necessarily re­flect those of the European Union or the European Commission. Neither the European Union nor the European Commission can be held responsible for them.\\
This work has been supported by a grant from the French National Research Agency (ANR) project MATURATION (reference
ANR-22-CE46-0012) \\
Support from the FrFCM (Fédération de recherche pour la Fusion par Co­finement Magnétique) in the frame of the SPARCLE
project (SParse grid Acceleration for the paRticle-in-CelL mEthod) is also acknowledged.
%#######################################################################################
%#######################################################################################
%Appendix 
\appendix

\section{Proof of \cref{thm:0}}
\label{apd:1}
To simplify the presentation, the time dependence of all quantities is omitted throughout this appendix. In the remainder of this section, we assume that the particle positions $\bm{x}_s$ ($s=1, \ldots, N$) are independent and identically distributed (i.i.d.) realizations of a random variable $\bm{X}$ with density $p(\bm{x}) = \rho(\bm{x})/\mathcal{M}$, where $\mathcal{M} = \int_\Omega \rho(\bm{\xi}) \, d\bm{\xi}$ is the total mass. Furthermore, $\mu(\Omega) = \int_\Omega d\bm{x}$ denotes the measure of the domain $\Omega$, which is assumed to be finite. The frameworks used to derive the bound estimates is resumed from the combination technique and the characterization of $V^{(1)}_h$ stated by \cref{prop:3}.
\subsection{Proofs of \cref{lemma:Density:Estimator:Expected:Value} and \cref{thm:0:1}}

\begin{proof}[Proof of \cref{lemma:Density:Estimator:Expected:Value}]
Owing to the transfer theorem in the distributional sense, the continuous density $p(\bm{x})$ is linked to the Dirac delta distribution via:
\begin{equation}
\label{eq:transfer_dirac}
\mathbb{E}\big[\delta(\bm{x}-\bm{X})\big] = \int_{\Omega} \delta(\bm{x}-\bm{\xi}) p(\bm{\xi})\mathrm{d}\bm{\xi} = p(\bm{x}) = \frac{\rho(\bm{x})}{\mathcal{M}}\,.
\end{equation}
Since the macro-particles are independent identically distributed (i.i.d) realizations of $\bm{X}$, their individual expectations satisfy $\mathbb{E}[\delta(\bm{x}-\bm{x}_s)] = \mathbb{E}[\delta(\bm{x}-\bm{X})]$ for all $s=1,\ldots,N$. Utilizing the linearity of the expectation operator along with the uniform weight definition, we obtain:
\[
\mathbb{E}[\rho_N(\bm{x})] = \mathbb{E}\left[ \sum_{s=1}^N w_s \delta\big(\bm{x}-\bm{x}_s\big) \right] = \frac{\mathcal{M}}{N}\sum_{s=1}^N \mathbb{E}\big[\delta(\bm{x}-\bm{x}_s)\big] = \mathcal{M} \, \mathbb{E}\big[\delta(\bm{x}-\bm{X})\big]= \rho(\bm{x})\,,
\]
confirming that $\text{Bias}[\rho_N(\bm{x})] = 0$.
 Within the same probabilistic framework, the expectation of $\rho_{h,N}$ evaluates to:
\begin{equation*}
\mathbb{E}[\rho_{h,N}(\bm{x})] %&= \mathbb{E}\left[ \sum_{s=1}^N w_s W_h\big(\bm{x}-\bm{x}_s(t)\big) \right] = \frac{\mathcal{M}}{N} \sum_{s=1}^N \mathbb{E}\big[ W_h(\bm{x}-\bm{x}_s(t)) \big] \nonumber \\
= \mathcal{M} \, \mathbb{E}\big[ W_h(\bm{x}-\bm{X}) \big] = \int_{\Omega} W_h(\bm{x}-\bm{\xi}) \rho(\bm{\xi}) \mathrm{d}\bm{\xi} \nonumber  = (W_h * \rho)(\bm{x})\,.
\end{equation*}
\end{proof}

\begin{proof}[Proof of \cref{thm:0}, \cref{thm:0:1}]
Owing to \cref{lemma:Density:Estimator:Expected:Value}, we have $\esp[\rho_{N}]= \rho$.
Since the density $\rho\in \mathrm{H}^{q}_{\mathrm{mix}}(\Omega)$, we can apply~\cref{lem:7.2} of~\cref{lem:7} to $\esp[\rho_{N}]$ to obtain the result for $(\star)=(1)$ and~\cref{lem:7.1} of~\cref{lem:7} to obtain the result for $(\star)=(\infty)$.
\end{proof}

\subsection{Proof of \cref{thm:0:2}}
To prove \cref{thm:0}, \cref{thm:0:2}, we have to first introduce three auxiliary lemmas.

\begin{lemma} 
\label{lem:1}
For all $\bm{x}\in \Omega$, it holds
\[
\var\big[\Pi_{h}^{(\star)} \rho_{N}(\bm{x})\big] \leq \frac{\mathcal{M}^2}{N} \esp\big[(\Pi_{h}^{(\star)}\delta (\bm{x} - \bm{X}))^2\big].
\]
\end{lemma}

\begin{lemma} 
\label{lem:2}
For all  $v_h(\cdot,\bm{X})\in V_{h}^{(1)}(\Omega)$ defined by
\[
v_h(\bm{x},\bm{X}) = \sum_{\bm{\ell}\in \mathscr{L}_h }c_{\bm{\ell}} \sum_{\bm{j} \in I_{h_{\bm{\ell}}}} \alpha_{\bm{j}}(\bm{X})\varphi_{h_{\bm{k}},\bm{j}}^p(\bm{x}),
\]
we have the relation
\begin{align}
\label{lem:2:1}
\|\esp[\Pi_{h}^{(\star)}\delta (\cdot - \bm{X}) v_h(\cdot,\bm{X})]\|_{\mathrm{L}^1(\Omega)} \leq \sum_{\bm{\ell}\in \mathscr{L}_h }|c_{\bm{\ell}}| \sum_{\bm{j} \in I_{h_{\bm{\ell}}}} \esp[\alpha_{\bm{j}}(\bm{X})],
\end{align}
and 
\begin{align}
\label{lem:2:2}
\sum_{\bm{j} \in I_{h_{\bm{\ell}}}} \esp[\alpha_{\bm{j}}(\bm{X})]  \lesssim \prod_{i=1}^d h_{\ell_i}^{-1}.
\end{align}
\end{lemma}

\begin{lemma} 
\label{lem:4}
The following estimate holds
\[
\sum_{\bm{\ell}\in \mathscr{L}_h }|c_{\bm{\ell}}| \prod_{i=1}^d h_{\ell_i}^{-1} \lesssim h^{-1} |\log h|^{d-1}.\]
\end{lemma}

\begin{proof}[Proof of \cref{thm:0}, \cref{thm:0:2}]
If $(\star)=(1)$, we obtain the result by applying \cref{lem:1}, \cref{lem:2} with $\varphi_h(\cdot, \bm{X}) = \Pi_{h}^{(\star)}\delta (\cdot -\bm{X})$ and \cref{lem:4} and taking the square root of the variance. If $(\star)=(\infty)$, expressing $\varphi_h(\cdot,\bm{X})\in V_{h}^{(\infty)}(\Omega)$ as the combination of an unique component grid of mesh size $h=(h,\ldots,h)$, using \cref{lem:1} and \cref{lem:2}, we obtain the result.
\end{proof}

\begin{proof}[Proof of \cref{lem:1}]
The particles being i.i.d. and owing to the linearity of the projection operator, the variance can be expressed as
\begin{align*}
\var[\Pi_{h}^{(\star)} \rho_{N}(\bm{x})] = \var\left[\frac{\mathcal{M}}{N}\sum_{s=1}^N \Pi_{h}^{(\star)}\delta (\bm{x} - \bm{X})\right]
= \frac{\mathcal{M}^2}{N} \var\left[ \Pi_{h}^{(\star)}\delta (\bm{x} - \bm{X})\right]. \label{lem:1:eq:2}
\end{align*}
The proof is concluded noting that $\big(\esp[\Pi_{h}^{(\star)} \rho_{N}(\bm{x})]\big)^2 \geq 0$ which yields the desired bound.
\end{proof}

\begin{proof}[Proof of \cref{lem:2}, \cref{lem:2:1}] The proof is decomposed into three steps.
  \begin{enumerate}
    \item We use the SGCT characterization of the approximation space, given by~\cref{prop:3}, to the Galerkin projection of the Delta dirac function  $ \Pi_{h}^{(\star)} \delta(\cdot-\bm{X})\subset V_{h}^{(\star)}(\Omega)$ so that we have, for any $x\in \Omega$
\begin{equation}\label{eq:ansatz:pi:delta}
  \Pi_{h}^{(\star)} \delta(\bm{x}-\bm{X}) = \sum_{\bm{\ell}\in \mathscr{L}_h }c_{\bm{\ell}} \sum_{\bm{i} \in I_{h_{\bm{\ell}}}}\alpha_{\bm{i}}(\bm{X}) \varphi_{h_{\bm{\ell}},\bm{i}}^p(\bm{x}).
\end{equation}
Note that, if $(\star)=(\infty)$, the level index set contains only one unique term.
\item This decomposition is introduced within the computation of one term of the variance of the Galerkin projection, yielding
\begin{equation}
  \esp\big[\big(\Pi_{h}^{(\star)}\delta (\bm{x} - \bm{X})\big)^2\big]= \sum_{\bm{\ell}\in \mathscr{L}_h }c_{\bm{\ell}}\sum_{\bm{j} \in I_{h_{\bm{\ell}}}} \int_{\Omega}\alpha_{\bm{j}}(\bm{\xi}) \left( \sum_{\bm{k}\in \mathscr{L}_h }c_{\bm{k}}\sum_{\bm{i} \in I_{h_{\bm{k}}}}\alpha_{\bm{i}}(\bm{\xi})\varphi_{h_{\bm{k}},\bm{i}}^p(\bm{x})\varphi_{h_{\bm{\ell}},\bm{j}}^p(\bm{x}) \right) p(\bm{\xi}) \, d\bm{\xi} .
\end{equation}
Integrating with respect to $\bm{x}$ we obtain:
\begin{equation}\label{lem:2:eq0}
  \begin{multlined}[0.85\textwidth]
\big\|\esp\big[\big(\Pi_{h}^{(\star)}\delta (\cdot - \bm{X})\big)^2\big]\big\|_{\mathrm{L}^1(\Omega)} =\\
\hfill {\sum_{\bm{\ell}\in \mathscr{L}_h }c_{\bm{\ell}}\sum_{\bm{j} \in I_{h_{\bm{\ell}}}} \int_{\Omega}\alpha_{\bm{j}}(\bm{\xi}) \left( \sum_{\bm{k}\in \mathscr{L}_h }c_{\bm{k}}\sum_{\bm{i} \in I_{h_{\bm{k}}}}\alpha_{\bm{i}}(\bm{\xi})\int_\Omega \varphi_{h_{\bm{k}},\bm{i}}^p(\bm{x})\varphi_{h_{\bm{\ell}},\bm{j}}^p(\bm{x})  \, d\bm{x} \right) p(\bm{\xi}) \, d\bm{\xi}}.
\end{multlined}
\end{equation}
\item By definition of the Galerkin projection of the delta Dirac function, the following identity holds for any $\bm{\ell}\in\mathscr{L}_h, ~ \bm{j}\in I_{h_{\bm{\ell}}}$
\[
\big(\Pi_{h}^{(\star)}\delta(\cdot-\bm{X}),\varphi_{h_{\bm{\ell}},\bm{j}}^p\big)_{\mathrm{L}^2} = \big\langle \delta(\cdot-\bm{X}),\varphi_{h_{\bm{\ell}},\bm{j}}^p \big\rangle_{\mathrm{H}^{-p}_{\mathrm{mix}}, \, \mathrm{H}^p_{\mathrm{mix}}},
\]
from which, thanks to the Ansatz defined by \cref{eq:ansatz:pi:delta}, the following identity is recovered
\begin{align} 
\sum_{\bm{k}\in \mathscr{L}_h }c_{\bm{k}} \sum_{\bm{i} \in I_{h_{\bm{k}}}}\alpha_{\bm{i}}(\bm{X}) \int_{\Omega} \varphi_{h_{\bm{k}},\bm{i}}^p(\bm{x})\varphi_{h_{\bm{\ell}},\bm{j}}^p(\bm{x}) d \bm{x} &= \int_{\Omega} \delta(\bm{x}-\bm{X})\varphi_{h_{\bm{\ell}},\bm{j}}^p(\bm{x}) d \bm{x}=\varphi_{h_{\bm{\ell}},\bm{j}}^p(\bm{X}). \label{lem:2:eq:1}
\end{align}
\end{enumerate}
The proof is concluded by inserting \cref{lem:2:eq:1} into \cref{lem:2:eq0}, to write
\begin{align*}
\big\|\esp\big[\big(\Pi_{h}^{(\star)}\delta (\bm{x} - \bm{X})\big)^2\big]\big\|_{\mathrm{L}^1(\Omega)} 
& = \sum_{\bm{\ell}\in \mathscr{L}_h }c_{\bm{\ell}}\sum_{\bm{j} \in I_{h_{\bm{\ell}}}}  \int_{\Omega} \alpha_{\bm{j}}(\bm{\xi}) \varphi_{h_{\bm{\ell}},\bm{j}}^p(\bm{\xi}) p(\bm{\xi}) \, d\bm{\xi}.
\end{align*}
Owing to the properties, $\|\varphi_{h_{\bm{\ell}},\bm{j}}^p\|_{L^\infty(\Omega)}\leq 1$ and $\alpha_{\bm{j}} \geq 0$, $ \quad \forall \bm{\ell}\in \mathscr{L}_h, \bm{j} \in I_{h_{\bm{\ell}}}$ and  the bound is obtained
\begin{align*}
\big\|\esp\big[\big(\Pi_{h}^{(\star)}\delta (\bm{x} - \bm{X})\big)^2\big]\big\|_{\mathrm{L}^1(\Omega)} 
& \leq\sum_{\bm{\ell}\in \mathscr{L}_h }|c_{\bm{\ell}}|\sum_{\bm{j} \in I_{h_{\bm{\ell}}}} \esp\big[\alpha_{\bm{j}}(\bm{X})\big].
\end{align*}
\end{proof}

\begin{proof}[Proof of \cref{lem:2},  \cref{lem:2:2}] The proof relies on establishing a bound for the integrated projection error (or bias), denoted by $\mathcal{B}_{\bm{\ell}}$ and defined as
  $$\mathcal{B}_{\bm{\ell}} = \int_{\Omega} \big(\operatorname{I}-\Pi_{h_{\bm{\ell}}}\big) \left(\esp[\delta(\bm{x}-\bm{X})]\right)d\bm{x}.$$

First, from Cauchy-Schwarz inequality, we have
\begin{align}
\label{lem:3:eq:1}
\left |\mathcal{B}_{\bm{\ell}}\right| %\leq \left\| \big(\operatorname{I}-\Pi_{h_{\bm{\ell}}}\big) \left(\esp[\delta(\bm{x}-\bm{X})]\right)\right\|_{\mathrm{L}^1(\Omega)} 
\leq \mu(\Omega)^{1/2} \left\| \big(\operatorname{I}-\Pi_{h_{\bm{\ell}}}\big) \left(\esp[\delta(\cdot-\bm{X})]\right)\right\|_{\mathrm{L}^2(\Omega)}.
\end{align}
Owing to \cref{eq:transfer_dirac}, the expectation corresponds to scaled continuous density, i.e., $\esp[\delta(\cdot-\bm{X})]= \mathcal{M}^{-1} \rho$. Consequently, applying proposition 3.1, from \cite{guillet25-1} leads to the following bound:
\begin{align}
\label{lem:3:eq:2}
\left |\mathcal{B}_{\bm{\ell}}\right| \lesssim \sum_{i=1}^d h_{\ell_i}^{p+1}\|\rho\|_{\mathrm{H}^{p+1}(\Omega)}.
\end{align}

Second, using the characterization defined by \cref{eq:ansatz:pi:delta} along with Fubini's theorem, 
 the integrated bias can be recast as
\begin{align}\label{eq:Bl:recast}
\mathcal{B}_{\bm{\ell}} &=%\Fab{\int_\Omega \frac{\rho(\bm{x})}{\mathcal{M}} d\bm{x}}
1  -  \sum_{\bm{j} \in I_{h_{\bm{\ell}}}} \left(\int_{\Omega} \varphi_{h_{\bm{\ell}},\bm{j}}^p(\bm{x}) \,d\bm{x} \right)\int_{\Omega} \alpha_{\bm{j}}(\bm{\xi})  p(\bm{\xi})d\bm{\xi} = 1 -  \sum_{\bm{j} \in I_{h_{\bm{\ell}}}} \left(\int_{\Omega} \varphi_{h_{\bm{\ell}},\bm{j}}^p(\bm{x}) \,d\bm{x}\right)\esp[\alpha_{\bm{j}}(\bm{X})].
\end{align}
The $d$-dimensional B-spline basis functions $\varphi_{h_{\bm{\ell}},\bm{i}}^p$ are constructed via tensor products of univariate cardinal B-splines. Since a standard cardinal B-spline integrates to one over its support $[-\frac{p+1}{2},\frac{p+1}{2}]$, a straightforward change of variables accounts for the grid spacing in each dimension, yielding
\begin{equation}\label{eq:bspline:integral}
  \int_{\Omega} \varphi_{h_{\bm{\ell}},\bm{i}}^p(\bm{x}) \, d\bm{x} = \prod_{k=1}^d h_{\ell_k}, \quad \forall \bm{i}\in I_{h_{\bm{\ell}}}.
\end{equation}

Finally, substituting the normalization property \eqref{eq:bspline:integral} into \eqref{eq:Bl:recast} and utilizing the bound established in \cref{lem:3:eq:2}, we obtain
\[
\sum_{\bm{j} \in I_{h_{\bm{\ell}}}} \esp[\alpha_{\bm{j}}(\bm{X})] \lesssim \frac{1+\sum_{i=1}^dh_{\ell_i}^{p+1}\|\rho\|_{\mathrm{H}^{p+1}(\Omega)}}{\prod_{i=1}^d h_{\ell_i}}\lesssim \prod_{i=1}^d h_{\ell_i}^{-1},
\]
which concludes the proof.
\end{proof}

\begin{proof}[Proof of \cref{lem:4}]
By using the definition of the admissible index set $\mathscr{L}_h$, we have
\begin{align*}
\sum_{\bm{\ell}\in \mathscr{L}_h }|c_{\bm{\ell}}| \prod_{i=1}^d h_{\ell_i}^{-1} &=h^{-1} \sum_{l=1}^{d-1}2^{d-1-l}\left|(-1)^l \right|\left|\binom{d-1}{l}\right|  \sum_{|\bm{\ell}|_1=\log h +d-1-l} 1 \\
&= h^{-1} \sum_{l=1}^{d-1}2^{d-1-l}\left|(-1)^l \right|\left|\binom{d-1}{l}\right|  \binom{\log h +d-2-l}{d-1}.
\end{align*}
We can conclude by observing that
\[
 \binom{\log h +d-2-l}{d-1} \lesssim |\log h|^{d-1}.
\]
\end{proof}

\subsection{Proof of \cref{thm:1:1}}
%\Fab{\noindent\textbf{Change of notation.}
\paragraph{Change of notation}
So far, the quantities $\Phi_{h,N}^{(\star)}$, $l_{h,N}$, and related objects have been treated as deterministic, with the particle positions $\bm{x}_s$ regarded as fixed. To establish the bias estimate for the electric field, however, we must adopt a probabilistic viewpoint: the particle positions $\bm{x}_s$ are i.i.d.\ realizations of the random variable $\bm{X}$, so that any mesh-based quantity depending on the particle configuration $(\bm{x}_1,\ldots,\bm{x}_N)$ becomes itself a random variable, and its expectation is well-defined.

In particular, the discrete electric potential $\Phi_{h,N}^{(\star)}$ and the discrete linear form $l_{h,N}$ inherit this random character through their dependence on the particle positions. To make this dependency explicit and to avoid any confusion with the deterministic notation used in the rest of the document, we write $\Phi_{h,N}^{(\star)}(\cdot,\bm{X})$ and $l_{h,N}(v_h,\bm{X})$ throughout this subsection to emphasize that these quantities are functions of the random variable $\bm{X}$. The expected electric potential $\esp[\Phi_{h,N}^{(\star)}(\cdot,\bm{X})]$ is therefore a deterministic function of the spatial variable $\bm{x}$ alone, and the bias of the discrete electric potential is defined as
\[
\bias\big[\Phi_{h,N}^{(\star)}\big] := \esp\big[\Phi_{h,N}^{(\star)}(\cdot,\bm{X})\big] - \Phi.
\]
To prove \cref{thm:0}, \cref{thm:1:1}, we rely on the following two auxiliary lemmas.

\begin{lemma}
\label{lem:6}
Let $\rho \in \mathrm{H}^q_{\mathrm{mix}}(\Omega)$ for some $q \in \mathbb{N}$. Then the
solution of \cref{thm:3:eq:2} satisfies $\Phi \in \mathrm{H}^{q}_{\mathrm{mix}}(\Omega)
\cap \mathrm{L}^2_0(\Omega)$, with
\[
\|\Phi\|_{\mathrm{H}^{q}_{\mathrm{mix}}(\Omega)} \lesssim
\|\rho\|_{\mathrm{H}^q_{\mathrm{mix}}(\Omega)}.
\]
\end{lemma}

The second lemma provides a Galerkin identity for the expected potential. Recall that
the discrete linear form $l_{h,N}(\cdot,
\bm{X})$ is a random object through its dependence on the particle positions. We introduce
its deterministic counterpart $\bar{l}_h(\cdot, \bm{X})$, defined by replacing the random
density estimator with its expectation, and establish that the expected potential satisfies
the corresponding Galerkin equation.
\begin{lemma}
\label{lem:5}
Let the linear form $\bar{l}_{h}(\cdot,\bm{X})$ be defined by
\[
\bar{l}_{h}(v_h,\bm{X}) := \int_{\Omega}
\esp\big[\Pi_{h}^{(\star)}\rho_{N}(\bm{x},\bm{X})\big]\, v_h(\bm{x})\, d\bm{x},
\quad \forall v_h \in V_{h,0}^{(\star)}(\Omega).
\]
Then the expected potential $\esp[\Phi_{h,N}^{(\star)}(\cdot,\bm{X})]$ satisfies the
following Galerkin equation:
\[
a\!\left(\esp\big[\Phi_{h,N}^{(\star)}(\cdot,\bm{X})\big], v_h\right)
= \bar{l}_{h}(v_h,\bm{X}),
\quad \forall v_h \in V_{h,0}^{(\star)}(\Omega).
\]
\end{lemma}

\begin{proof}[Proof of \cref{lem:5}]
By definition of the bilinear form, we have
\begin{align*}
a\left(\esp\big[\Phi_{h,N}^{(\star)}(\cdot,\bm{X})\big],v_h\right) %&= \int_{\Omega} \bm{\nabla}\left(\esp\big[\Phi_{h,N}^{(\star)}(\bm{x},\bm{X})\big]\right) \cdot \bm{\nabla}v_h(\bm{x}) \, d\bm{x} \\
&= \int_{\Omega}\bm{\nabla}\left(\int_{\Omega}\Phi_{h,N}^{(\star)}(\bm{x},\bm{\xi})p(\bm{\xi}) \,d\bm{\xi}\right) \cdot \bm{\nabla}v_h(\bm{x})\, d\bm{x}.
\end{align*}
  Applying Fubini's theorem, we obtain
\begin{align*}
a\left(\esp\big[\Phi_{h,N}^{(\star)}(\cdot,\bm{X})\big],v_h\right)&= \int_{\Omega}\left(\int_{\Omega}  \bm{\nabla}\Phi_{h,N}^{(\star)}(\bm{x},\bm{\xi})\cdot \bm{\nabla}v_h(\bm{x}) \, d\bm{x}\right)  p(\bm{\xi}) \,d\bm{\xi} \\
&= \int_{\Omega}  l_{h,N}(v_h,\bm{\xi}) p(\bm{\xi})\,d\bm{\xi} = \bar{l}_{h}(v_h,\bm{X}).
\end{align*}
\end{proof}

\begin{proof}[Proof of \cref{lem:6}] The proof relies on the characterization of the mixed Sobolev spaces on the periodic torus $\Omega = [0,L]^d/\mathbb{Z}^d$ in terms of Fourier coefficients.

The charge of particles and the electric potential are assumed to be of zero mean: $\Phi\in\mathrm{L}^2_0(\Omega)$ and 
$(1-\rho) \in\mathrm{L}^2_0(\Omega)$ so that both expand into Fourier series over the
nonzero modes only:
\[
1-\rho(\bm{x}) = -\!\!\sum_{\bm{k}\in\mathbb{Z}^d\setminus\{\bm{0}\}}\!\!
\hat{\rho}(\bm{k})\,{e}^{\bm{k}}(\bm{x}),
\qquad
\Phi(\bm{x}) = \!\!\sum_{\bm{k}\in\mathbb{Z}^d\setminus\{\bm{0}\}}\!\!
\hat{\Phi}(\bm{k})\,{e}^{\bm{k}}(\bm{x}),\qquad {e}^{\bm{k}}(\bm{x})={\exp}\left(\,i\frac{2\pi}{L}\bm{k}\cdot\bm{x}\right).
\]
%where we used $\hat{1}(\bm{k})=0$ for $\bm{k}\neq\bm{0}$. 
Since the Fourier basis functions ${e}^{\bm{k}}$ are eigenfunctions of the Laplacian, Poisson equation $-\Delta \Phi =1-\rho$ yields,
 identifying the coefficients mode by mode:
\begin{equation}
\label{lem:6:eq:fourier_relation}
\hat{\Phi}(\bm{k})
= -\left(\frac{L}{2\pi}\right)^2\frac{\hat{\rho}(\bm{k})}{\|\bm{k}\|_2^2},
\qquad \bm{k}\in\mathbb{Z}^d\setminus\{\bm{0}\}, \qquad \|\bm{k}\|_2^2=\sum_{j=1}^d k_j^2.
\end{equation}

For $f\in \mathrm{H}^q_{\mathrm{mix}}$ and $q\in\mathbb{N}$, we use the norm equivalence
\begin{equation}
\label{lem:6:eq:norm_equiv}
\|f\|_{\mathrm{H}^q_{\mathrm{mix}}(\Omega)}^2
\sim \sum_{\bm{k}\in\mathbb{Z}^d}
\Big(\prod_{j=1}^d(1+k_j^2)^q\Big)\,|\hat{f}(\bm{k})|^2,
\end{equation}
with constants depending only on $q$, $d$, and $L$. Substituting \cref{lem:6:eq:fourier_relation} into \cref{lem:6:eq:norm_equiv} applied to $\Phi$ yields
\begin{equation}
\label{lem:6:eq:substituted}
\|\Phi\|_{\mathrm{H}^q_{\mathrm{mix}}(\Omega)}^2
\lesssim
\sum_{\bm{k}\in\mathbb{Z}^d\setminus\{\bm{0}\}}
\,
\mathcal{A}(\bm{k}) \, |\hat{\rho}(\bm{k})|^2 , \qquad \mathcal{A}(\bm{k}) = 
\frac{\prod_{j=1}^d(1+k_j^2)^q}{\|\bm{k}\|_2^4}.
\end{equation}

For any nonzero frequency vector $\bm{k}\in\mathbb{Z}^d\setminus\{\bm{0}\}$, at least one component $k_m$, $1\leq m\leq d$ is a nonzero integer, hence satisfies $|k_m|\geq 1$. This provides the uniform lower bound
\[
\|\bm{k}\|_2^2 = \sum_{j=1}^d k_j^2 \geq k_m^2 \geq 1,
\qquad\text{hence}\qquad
\frac{1}{\|\bm{k}\|_2^4}\leq 1.
\]
Inserting this bound into \cref{lem:6:eq:substituted} and applying \cref{lem:6:eq:norm_equiv} to $\rho$,
\[
\|\Phi\|_{\mathrm{H}^q_{\mathrm{mix}}(\Omega)}^2
\lesssim
\sum_{\bm{k}\in\mathbb{Z}^d\setminus\{\bm{0}\}}
\Big(\prod_{j=1}^d(1+k_j^2)^q\Big)\,|\hat{\rho}(\bm{k})|^2
\lesssim \|\rho\|_{\mathrm{H}^q_{\mathrm{mix}}(\Omega)}^2;
\]
which completes the proof.
\end{proof}

\begin{remark}[Gain of one mixed order in dimension $d=2$]
\label{rem:lem6:d2gain}
In dimension $d=2$, the stability estimate of \cref{lem:6} can be
improved by one full order in the mixed scale: if
$\rho \in \mathrm{H}^q_{\mathrm{mix}}(\Omega)$, then
$\Phi \in \mathrm{H}^{q+1}_{\mathrm{mix}}(\Omega) \cap
\mathrm{L}^2_0(\Omega)$, with
\[
\|\Phi\|_{\mathrm{H}^{q+1}_{\mathrm{mix}}(\Omega)}
\lesssim \|\rho\|_{\mathrm{H}^q_{\mathrm{mix}}(\Omega)} .
\]
Indeed, in view of \cref{lem:6:eq:substituted}, it suffices to bound the multiplier
\[
\mathcal{A}(\bm{k})
= \frac{(1+k_1^2)(1+k_2^2)}{(k_1^2+k_2^2)^2},
\qquad \bm{k}\in\mathbb{Z}^2\setminus\{\bm{0}\}.
\]
If both components are nonzero, then $1+k_j^2 \leq 2k_j^2$ and the
arithmetic--geometric mean inequality
$k_1^2 k_2^2 \leq \tfrac14 (k_1^2+k_2^2)^2$ give $\mathcal{A}(\bm{k}) \leq 1$;
if a single component is nonzero, say $k_1 = n \neq 0$ and $k_2 = 0$,
then $\mathcal{A}(\bm{k}) = (1+n^2)/n^4 \leq 2$. Hence $\mathcal{A}$ is uniformly bounded
and the estimate follows from Parseval's identity.

This gain is specific to dimension two, where the derivative count is
exactly balanced: one additional mixed order costs two derivatives in
each direction, i.e.\ $2d = 4$ derivatives in total, which is precisely
what the symbol $\|\bm{k}\|_2^4$ of the squared Laplacian provides. For
$d \geq 3$, the requirement of $2d \geq 6$ derivatives exceeds the
available four, and the corresponding multiplier diverges as
$k^{2d-4}$ on the diagonal modes $k_1 = \dots = k_d = k$. The gain of two mixed orders fails
even for $d = 2$, where the multiplier behaves as $k^4/4$ on the
diagonal: the improvement is thus of exactly one order.
\end{remark}

%\begin{lemma}
%\label{lem:6}
%If $\rho \in \mathrm{H}^m_{\mathrm{mix}}$, with $m\in \mathbb{N}$, then the solution of \cref{thm:3:eq:2} verifies $\Phi\in \mathrm{H}^{m+2}_{\mathrm{mix}}\cap \mathrm{L}^2_0(\Omega)$ and 
%\[
%\|\Phi\|_{\mathrm{H}^{m+2}_{\mathrm{mix}}} \lesssim \|\rho\|_{ \mathrm{H}^m_{\mathrm{mix}}}.
%\]
%\end{lemma}
%\begin{proof}[Proof of \cref{lem:6}]
%The density and electric potential are periodic functions so that we can introduce their Fourier series, defined for all $\bm{x}\in \Omega$ by
%\[
%\rho(\bm{x}) = \sum_{\bm{k}\in \mathbb{Z}^d}\hat{\rho}_k \ex^{2\pi i \bm{k}\cdot \bm{x}}, \quad \Phi(\bm{x}) = \sum_{\bm{k}\in \mathbb{Z}^d/\{0\}}\hat{\Phi}_k \ex^{2\pi i \bm{k}\cdot \bm{x}},
%\]
%where, because these functions verify the Poisson equation, the coefficients are related according to 
%\[
%\hat{\Phi}_k = \frac{\hat{\rho}_k}{(2\pi)^2|\bm{k}|_1^2}, \quad \forall \bm{k}\in \mathbb{Z}^d/\{0\}.
%\]
%Using Parseval egality, and because $|k_j|^4\leq |\bm{k}|_1^4$ we have for $|\bm{\alpha}|_{\infty}\leq m+2$, 
%\begin{align*}
%\|D^{\bm{\alpha}}\Phi \|_{\mathrm{L}^2(\Omega)} &=  \sum_{\bm{k}\in \mathbb{Z}^d/\{0\}} |\hat{\Phi}_k|^2 \prod_{j=1}^d(2\pi |k_j|)^{2\alpha_j} \\
%&\lesssim \sum_{\bm{k}\in \mathbb{Z}^d/\{0\}} |\hat{\rho}_k|^2 \prod_{j=1}^d|k_j|^{2(\alpha_j-2)} \\
%& \lesssim \|D^{\bm{\beta}}\rho \|_{\mathrm{L}^2(\Omega)}, \quad \text{with}~~|\bm{\beta}|_\infty \leq m, \\
%&\lesssim \|\rho\|_{\mathrm{H}^m_{\mathrm{mix}}(\Omega)}.
%\end{align*}
%Summing this relation for all $\bm{\alpha}$ such that $|\bm{\alpha}|_{\infty}\leq m+2$, we obtain the result.
%\end{proof}

\begin{proof}[Proof of \cref{thm:0}, \cref{thm:1:1}]

% --- Setup ---
We set $V := \mathrm{H}^1(\Omega) \cap \mathrm{L}^2_0(\Omega)$, equipped with the $\mathrm{H}^1$ norm.

The proof relies on the introduction of the deterministic Galerkin solution
$\Phi_h \in V_{h,0}^{(\star)}(\Omega)$, defined as the solution of the
variational problem with exact source $\rho$:
\[
a(\Phi_h, v_h) = l(v_h), \quad \forall v_h \in V_{h,0}^{(\star)}(\Omega).
\]
This intermediate quantity allows us to decompose the bias of the electric
potential into two contributions of different nature via the triangular
inequality:
\[
\big\|\esp\big[\Phi_{h,N}^{(\star)}(\cdot,\bm{X})\big] - \Phi\big\|_V
\leq
\underbrace{\big\|\esp\big[\Phi_{h,N}^{(\star)}(\cdot,\bm{X})\big] -
\Phi_h\big\|_V}_{\displaystyle \mathcal{E}_{\mathrm{stoch}}}
+
\underbrace{\|\Phi_h - \Phi\|_V}_{\displaystyle \mathcal{E}_{\mathrm{disc}}}.
\]

The term $\mathcal{E}_{\mathrm{disc}}$ measures the error introduced by replacing the
continuous variational space $V$ with the discrete space
$V_{h,0}^{(\star)}(\Omega)$, while keeping the source term $\rho$ exact. Since
both $\Phi$ and $\Phi_h$ satisfy the variational problem with the same linear
form $l$, the Galerkin orthogonality holds exactly, and $\mathcal{E}_{\mathrm{disc}}$ is
controlled by a standard C\'ea argument:
\[
\mathcal{E}_{\mathrm{disc}} = \|\Phi_h - \Phi\|_{\mathrm{H}^1(\Omega)}
\lesssim \inf_{\varphi_h \in V_{h,0}^{(\star)}(\Omega)} \|\Phi - \varphi_h\|_{\mathrm{H}^1(\Omega)}
\lesssim h\|\Phi\|_{\mathrm{H}^2_{\mathrm{mix}}(\Omega)}
\lesssim h\|\rho\|_{\mathrm{H}^2_{\mathrm{mix}}(\Omega)},
\]
where the last two inequalities follow from \cref{lem:6} and the
approximation properties of $V_{h,0}^{(\star)}(\Omega)$.

The term $\mathcal{E}_{\mathrm{stoch}}$ measures the error introduced, within the
discrete space $V_{h,0}^{(\star)}(\Omega)$, by replacing the exact source
$\rho$ with the most probable value of its projected estimator $\esp[\Pi_h^{(\star)}\rho_N(\cdot,
\bm{X})]$. Both $\Phi_h$ and $\esp[\Phi_{h,N}^{(\star)}(\cdot,\bm{X})]$
live in the same discrete space and satisfy variational problems with the
same bilinear form $a$, so that $\mathcal{E}_{\mathrm{stoch}}$ can be bound thanks to coercivity arguments.

Setting $w_h := \esp[\Phi_{h,N}^{(\star)}(\cdot,\bm{X})] - \Phi_h \in V_{h,0}^{(\star)}(\Omega)$,
coercivity of $a$ gives $\|w_h\|_{\mathrm{H}^1(\Omega)}^2 \lesssim a(w_h, w_h)$. By \cref{lem:5},
$a(\esp[\Phi_{h,N}^{(\star)}(\cdot,\bm{X})], w_h) = \bar{l}_h(w_h, \bm{X})$,
and by definition of $\Phi_h$, $a(\Phi_h, w_h) = l(w_h)$. Hence
\[
a(w_h, w_h) = \bar{l}_h(w_h, \bm{X}) - l(w_h),
\]
and dividing by $\|w_h\|_{\mathrm{H}^1(\Omega)}$,
\[
\mathcal{E}_{\mathrm{stoch}} = \|w_h\|_{\mathrm{H}^1(\Omega)}
\lesssim \frac{|\bar{l}_h(w_h,\bm{X}) - l(w_h)|}{\|w_h\|_{\mathrm{H}^1(\Omega)}}
%\leq \sup_{v \in V} \frac{|\bar{l}_h(v,\bm{X}) - l(v)|}{\|v\|_{\mathrm{H}^1(\Omega)}}
%\leq \big\|\bias\big[\Pi_h^{(\star)}\rho_N\big]\big\|_{\mathrm{L}^2(\Omega)}.
\]
where, by Cauch-Schwartz inequality and the embedding $\left\|\cdot \right\|_{\mathrm{L}^2(\Omega)}\leq \left\|\cdot \right\|_{\mathrm{H}^1(\Omega)}$
\[
|\bar{l}_h(v,\bm{X}) - l(v)| =\left| \int_{\Omega}
\left(\esp\big[\Pi_{h}^{(\star)}\rho_{N}(\bm{x},\bm{X})\big] - \rho(\bm{x}) \right)\, v_h(\bm{x})\, d\bm{x} \right|\leq 
 \big\|\bias\big[\Pi_h^{(\star)}\rho_N\big]\big\|_{\mathrm{L}^2(\Omega)} \left\|w_h \right\|_{\mathrm{H}^1(\Omega)}
\]

Combining both bounds, we obtain
\[
\big\|\esp\big[\Phi_{h,N}^{(\star)}(\cdot,\bm{X})\big] - \Phi\big\|_{\mathrm{H}^1(\Omega)}
\lesssim h\|\rho\|_{\mathrm{H}^2_{\mathrm{mix}}(\Omega)}
+ \big\|\bias\big[\Pi_h^{(\star)}\rho_N\big]\big\|_{\mathrm{L}^2(\Omega)},
\]
which completes the proof.
\end{proof}

\subsection{Proof of \cref{thm:1:2}}
\begin{proof}[Proof of \cref{thm:0}, \cref{thm:1:2}]
The proof relies on the following bound, for the variance of the component $1\leq i \leq d$ of the electric field:
\begin{align*}
  \left\|\var\big[(\bm{\nabla}\Phi_{h,N}^{(\star)})_i(\cdot,\bm{X})\big]\right\|_{\mathrm{L}^1(\Omega)} &\leq \left\|\esp\big[\big(\bm{\nabla}\Phi_{h,N}^{(\star)}(\cdot,\bm{X})\big)_i^2\big]\right\|_{\mathrm{L}^1(\Omega)}\,,
\end{align*}
where, owing to Fubini's theorem
\begin{align*}
 \left\|\esp\big[(\bm{\nabla}\Phi_{h,N}^{(\star)}(\cdot,\bm{X}))_i^2\big]\right\|_{\mathrm{L}^1(\Omega)} =\int_{\Omega}\left(\int_{\Omega } \bm{\nabla}\Phi_{h,N}^{(\star)}(\bm{x},\bm{\xi}) \cdot  \bm{\nabla}\Phi_{h,N}^{(\star)}(\bm{x},\bm{\xi})\, d\bm{x}\right) p( d\bm{\xi}) \, d\bm{\xi}  \,.
\end{align*}
Now using the identity $a(\Phi_{h,N}^{(\star)},\Phi_{h,N}^{(\star)})= l_{h,N}\big( \Phi_{h,N}^{(\star)}\big)$, we obtain 
\begin{align*}
\left\|\var\big[(\bm{\nabla}\Phi_{h,N}^{(\star)})_i(\cdot,\bm{X})\big]\right\|_{\mathrm{L}^1(\Omega)}
&\leq \int_{\Omega}\left(\int_{\Omega } \Pi_{h}^{(\star)} \rho_N(\bm{x},\bm{\xi}) \Phi_{h,N}^{(\star)}(\bm{x},\bm{\xi})\, d\bm{x}\right) p( d\bm{\xi}) \, d\bm{\xi} \\
&\leq  \left\|\esp\big[\Pi_{h}^{(\star)}\delta (\bm{x} - \bm{X})\Phi_{h,N}^{(\star)}(\cdot,\bm{X}))\big]\right\|_{\mathrm{L}^1(\Omega)}.
\end{align*}
If $(\star)=(1)$, we conclude by applying \cref{lem:2}, \cref{lem:2:1} to  $\Phi_{h,N}^{(\star)}(\cdot,\bm{X})\in V_{h}^{(1)}(\Omega)$, and \cref{lem:2:2}, and \cref{lem:4}.  If $(\star)=(\infty)$, we use \cref{lem:2} with $\varphi_h(\cdot,\bm{X})\in V_{h}^{(\infty)}(\Omega)$ expressed as the combination of a unique component grid of mesh size $h=(h,\ldots,h)$.
\end{proof}

\section{Proof of \cref{prop:galerkin_projection}}
\label{apd:0}
\Fabrice{Since the basis functions $\varphi^p_{h_{\bm{\ell}},\bm{j}}$ of degree $p\geq 1$ belong to $\mathrm{H}^p_{\mathrm{mix}}(\Omega)$, we have $V_h^{(\star)}(\Omega)\subset \mathrm{H}^p_{\mathrm{mix}}(\Omega)$, so that the duality pairing on the right-hand side of \cref{eq:def:Galerkin:projection} is well-defined for all $v_h \in V_h^{(\star)}(\Omega)$.} 

\Fabrice{For any $u \in \mathrm{H}^{-p}_{\mathrm{mix}}(\Omega)$ let $\mathcal{L}_u : V_{h}^{(\star)}(\Omega) \to \mathbb{R}$  be the linear functional defined by $\mathcal{L}_u(v_h) := \langle u, v_h \rangle_{\mathrm{H}^{-p}_{\mathrm{mix}}, \mathrm{H}^p_{\mathrm{mix}}}$. By definition of the dual norm, $\mathcal{L}_u$ is continuous with respect to $\|\cdot\|_{\mathrm{H}^p_{\mathrm{mix}}(\Omega)}$ on $V_h^{(\star)}(\Omega)$.}

\Fabrice{Since $V_h^{(\star)}(\Omega)$ is finite-dimensional, all norms on it are equivalent, so that $\mathcal{L}_u$ is also continuous with respect to $\|\cdot\|_{\mathrm{L}^2(\Omega)}$ on $V_h^{(\star)}(\Omega)$. The latter, equipped with the $\mathrm{L}^2$ inner product, is a Hilbert space. 
By the Riesz representation theorem, there exists a unique element $\Pi_h^{(\star)}u \in V_h^{(\star)}(\Omega)$
such that $\mathcal{L}_u(v_h) = (\Pi_h^{(\star)}u, v_h)_{\mathrm{L}^2
(\Omega)}$ for all $v_h \in V_h^{(\star)}(\Omega)$, which concludes the proof.}

\bibliography{bib/bib}

\begin{thebibliography}{10}

\bibitem{aydemir94}
A.~Y. Aydemir.
\newblock A unified monte carlo interpretation of particle simulations and
  applications to non--neutral plasmas.
\newblock {\em Physics of Plasmas}, 1(4):822--831, 1994.

\bibitem{petsc-user-ref}
S.~Balay, S.~Abhyankar, M.F. Adams, S.~Benson, J.~Brown, P.~Brune,
  K.~Buschelman, E.~Constantinescu, L.~Dalcin, A.~Dener, V.~Eijkhout,
  J.~Faibussowitsch, W.D. Gropp, V.~Hapla, T.~Isaac, P.~Jolivet, D.~Karpeev,
  D.~Kaushik, M.G. Knepley, F.~Kong, S.~Kruger, D.A. May, L.C. McInnes, R.T.
  Mills, L.~Mitchell, T.~Munson, J.E. Roman, K.~Rupp, P.~Sanan, Sarich J., B.F.
  Smith, S.~Zampini, H.~Zhang, and J.~Zhang.
\newblock {PETSc/TAO} users manual.
\newblock Technical Report ANL-21/39 - Revision 3.21, Argonne National
  Laboratory, 2024.

\bibitem{bellman61}
R.E. Bellman.
\newblock {\em Adaptive Control Processes: A Guided Tour}.
\newblock Princeton University Press, Princeton, 1961.

\bibitem{birdsall18}
C.K. Birdsall and A.B Langdon.
\newblock {\em Plasma Physics via Computer Simulation}.
\newblock CRC Press, 2018.

\bibitem{bokanowski13}
O.~Bokanowski, J.~Garcke, M.~Griebel, and I.~Klompmaker.
\newblock An adaptive sparse grid semi-lagrangian scheme for first order
  {Hamilton}-{Jacobi} {Bellman} equations.
\newblock {\em Journal of Scientific Computing}, 55(3):575--605, 2013.

\bibitem{bottino15}
A.~Bottino and E.~Sonnendrücker.
\newblock Monte carlo particle-in-cell methods for the simulation of the
  {Vlasov}--{Maxwel}l gyrokinetic equations.
\newblock {\em Journal of Plasma Physics}, 81(5):435810501, 2015.

\bibitem{bungartz91}
H.-J. Bungartz.
\newblock An adaptive {Poisson} solver using hierarchical bases and sparse
  grids.
\newblock {\em Forschungsberichte, TU Munich}, TUM I 9130:1--21, 1991.

\bibitem{bungartz98}
H.-J. Bungartz and T.~Dornseifer.
\newblock Sparse grids: Recent developments for elliptic partial differential
  equations.
\newblock In {\em Multigrid Methods V}, pages 45--70. Springer Berlin
  Heidelberg, 1998.

\bibitem{bungartz_note_1999}
H.-J. Bungartz and M.~Griebel.
\newblock A note on the complexity of solving {Poisson}'s equation for spaces
  of bounded mixed derivatives.
\newblock {\em Journal of Complexity}, 15(2):167--199, 1999.

\bibitem{bungartz04}
H.-J. Bungartz and M.~Griebel.
\newblock Sparse grids.
\newblock {\em Acta Numerica}, 13:147--269, 2004.

\bibitem{crestetto18}
A.~Crestetto, N.~Crouseilles, and M.~Lemou.
\newblock {A particle micro-macro decomposition based numerical scheme for
  collisional kinetic equations in the diffusion scaling}.
\newblock {\em {Communications in Mathematical Sciences}}, 16(4):887--911,
  2018.

\bibitem{deluzet22}
F.~Deluzet, G.~Fubiani, L.~Garrigues, C.~Guillet, and J.~Narski.
\newblock Sparse grid reconstructions for particle-in-cell methods.
\newblock {\em ESAIM: M2AN}, 56:1809--1841, 2022.

\bibitem{deluzet22-1}
F.~Deluzet, G.~Fubiani, L.~Garrigues, C.~Guillet, and J.~Narski.
\newblock Efficient parallelization for 3d-3v sparse grid particle-in-cell:
  Shared memory architectures.
\newblock {\em Journal of Computational Physics}, 480:112022, 2023.

\bibitem{deluzet23}
F.~Deluzet, G.~Fubiani, L.~Garrigues, C.~Guillet, and J.~Narski.
\newblock Efficient parallelization for 3d-3v sparse grid particle-in-cell:
  Single {GPU} architectures.
\newblock {\em Computer Physics Communications}, 289:108755, 2023.

\bibitem{deluzet25}
F.~Deluzet, C.~Guillet, J.~Narski, and P.~Pace.
\newblock High-order sparse-{PIC} methods: Analysis and numerical
  investigations.
\newblock {\em SIAM Journal on Numerical Analysis}, 63(3):1281--1314, 2025.

\bibitem{denton95}
R.E. Denton and M.~Kotschenreuther.
\newblock {$\delta$f} algorithm.
\newblock {\em Journal of Computational Physics}, 119(2):283--294, 1995.

\bibitem{Dornseifer:1996aa}
T.~Dornseifer and C.~Pflaum.
\newblock Discretization of elliptic differential equations on curvilinear
  bounded domains with sparse grids.
\newblock {\em Computing}, 56(3):197--213, 1996.

\bibitem{driscoll90}
C.~F. Driscoll and K.~S. Fine.
\newblock Experiments on vortex dynamics in pure electron plasmas.
\newblock {\em Physics of Fluids B: Plasma Physics}, 2(6):1359--1366, 1990.

\bibitem{garrigues24}
L.~Garrigues, M.~Chung-To-Sang, G.~Fubiani, C.~Guillet, F.~Deluzet, and
  J.~Narski.
\newblock Acceleration of particle-in-cell simulations using sparse grid
  algorithms. {I.} {Application} to dual frequency capacitive discharges.
\newblock {\em Physics of Plasmas}, 31(7):073907, 2024.

\bibitem{garrigues24-1}
L.~Garrigues, M.~Chung-To-Sang, G.~Fubiani, C.~Guillet, F.~Deluzet, and
  J.~Narski.
\newblock Acceleration of particle-in-cell simulations using sparse grid
  algorithms. {II.} {Application} to partially magnetized low temperature
  plasmas.
\newblock {\em Physics of Plasmas}, 31(7):073908, 2024.

\bibitem{garrigues21}
L.~Garrigues, B.~Tezenas~du Montcel, G.~Fubiani, F.~Bertomeu, F.~Deluzet, and
  J.~Narski.
\newblock Application of sparse grid combination techniques to low temperature
  plasmas particle-in-cell simulations. {I.} {Capacitively} coupled radio
  frequency discharges.
\newblock {\em Journal of Applied Physics}, 129(15):153303, 2021.

\bibitem{garrigues21-1}
L.~Garrigues, B.~Tezenas~du Montcel, G.~Fubiani, and B.~C.~G. Reman.
\newblock Application of sparse grid combination techniques to low temperature
  plasmas particle-in-cell simulations. {II.} {Electron} drift instability in a
  hall thruster.
\newblock {\em Journal of Applied Physics}, 129(15):153304, 2021.

\bibitem{gassama07}
S.~Gassama, E.~Sonnendrücker, K.~Schneider, M.~Farge, and M.~O. Domingues.
\newblock Wavelet denoising for postprocessing of a 2d particle-in-cell code.
\newblock {\em ESAIM: Proc.}, 16:195--210, 2007.

\bibitem{griebel07}
M.~Griebel and J.~Hamaekers.
\newblock Sparse grids for the {Schr{\"o}dinger} equation.
\newblock {\em ESAIM: M2AN}, 41(2):215--247, 2007.

\bibitem{griebel07-1}
M.~Griebel and D.~Oeltz.
\newblock A sparse grid space-time discretization scheme for parabolic
  problems.
\newblock {\em Computing}, 81(1):1--34, 2007.

\bibitem{griebel98-2}
M.~Griebel, P.~Oswald, and T.~Schiekofer.
\newblock Sparse grids for boundary integral equations.
\newblock {\em Numerische Mathematik}, 83, 1998.

\bibitem{griebel90}
M.~Griebel, M.~Schneider, and C.~Zenger.
\newblock A combination technique for the solution of sparse grid problems.
\newblock {\em Forschungsberichte, TU Munich}, TUM I 9038:1--24, 1990.

\bibitem{griebel99}
M.~Griebel and G.~Zumbusch.
\newblock Adaptive sparse grids for hyperbolic conservation laws.
\newblock In {\em Hyperbolic Problems: Theory, Numerics, Applications}, pages
  411--422, Basel, 1999. Birkh{\"a}user Basel.

\bibitem{guillet24}
C.~Guillet.
\newblock Semi-implicit particle-in-cell methods embedding sparse grid
  reconstructions.
\newblock {\em Multiscale Modeling and Simulation}, 22:891--924, 2024.

\bibitem{guillet25}
C.~Guillet.
\newblock Energy-conserving particle-in-cell scheme based on {Galerkin} methods
  with sparse grids.
\newblock {\em Journal of Computational Physics}, 524:113739, 2025.

\bibitem{guillet25-1}
C.~Guillet.
\newblock Error estimates for sparse tensor products of {B}-spline
  approximation spaces.
\newblock {\em ESAIM: M2AN}, 2026.

\bibitem{hockney88}
R.W. Hockney and J.W. Eastwood.
\newblock {\em Computer Simulation Using Particles}.
\newblock Advanced book program. McGraw-Hill International Book Company, 1981.

\bibitem{huang23}
C.-K. Huang, Y.~Zeng, Y.~Wang, M.D. Meyers, S.~Yi, and B.J. Albright.
\newblock Finite grid instability and spectral fidelity of the electrostatic
  particle-in-cell algorithm.
\newblock {\em Computer Physics Communications}, 207:123--135, 2016.

\bibitem{langdon70}
A.B. Langdon.
\newblock Effects of the spatial grid in simulation plasmas.
\newblock {\em Journal of Computational Physics}, 6(2):247--267, 1970.

\bibitem{muralikrishnan21}
S.~Muralikrishnan, A.J. Cerfon, M.~Frey, L.F. Ricketson, and A.~Adelmann.
\newblock Sparse grid-based adaptive noise reduction strategy for
  particle-in-cell schemes.
\newblock {\em Journal of Computational Physics}, 11:100094, 2021.

\bibitem{obersteiner21-1}
M.~Obersteiner and H.-J. Bungartz.
\newblock A generalized spatially adaptive sparse grid combination technique
  with dimension-wise refinement.
\newblock {\em SIAM Journal on Scientific Computing}, 43(4):A2381--A2403, 2021.

\bibitem{obersteiner21}
M.~Obersteiner and H.-J. Bungartz.
\newblock A spatially adaptive sparse grid combination technique for numerical
  quadrature.
\newblock In {\em Sparse Grids and Applications - Munich 2018}, pages 161--185.
  Springer International Publishing, 2021.

\bibitem{pfluger10}
D.~Pflüger, B.~Peherstorfer, and H.-J. Bungartz.
\newblock Spatially adaptive sparse grids for high-dimensional data-driven
  problems.
\newblock {\em Journal of Complexity}, 26(5):508--522, 2010.

\bibitem{pinkus85}
A.~Pinkus.
\newblock {\em $n$-Widths in Approximation Theory}.
\newblock Springer Berlin, Heidelberg, 1985.

\bibitem{ricketson17}
L.F. Ricketson and A.J. Cerfon.
\newblock Sparse grid techniques for particle-in-cell schemes.
\newblock {\em Plasma Phys. Control. Fusion}, 59(2):024002, 2017.

\bibitem{sande19}
E.~Sande, C.~Manni, and H.~Speleers.
\newblock Sharp error estimates for spline approximation: Explicit constants,
  n-widths, and eigenfunction convergence.
\newblock {\em Mathematical Models and Methods in Applied Sciences},
  29(06):1175--1205, 2019.

\bibitem{schumaker07}
L.~Schumaker.
\newblock {\em Spline Functions: Basic Theory}.
\newblock Cambridge Mathematical Library. Cambridge University Press, 3
  edition, 2007.

\bibitem{smolyak63}
S.A. Smolyak.
\newblock Quadrature and interpolation formulas for tensor products of certain
  classes of functions.
\newblock {\em Soviet Math. Dokl.}, 4:240--243, 1963.

\bibitem{sydora99}
R.D. Sydora.
\newblock Low-noise electromagnetic and relativistic particle-in-cell plasma
  simulation models.
\newblock {\em Journal of Computational and Applied Mathematics},
  109(1):243--259, 1999.

\bibitem{takacs16}
S.~Takacs and T.~Takacs.
\newblock Approximation error estimates and inverse inequalities for
  {B}-splines of maximum smoothness.
\newblock {\em Mathematical Models and Methods in Applied Sciences},
  26:1411--1445, 2016.

\bibitem{tranquilli22}
P.~Tranquilli, L.~Ricketson, and L.~Chac{\'o}n.
\newblock A deterministic verification strategy for electrostatic
  particle-in-cell algorithms in arbitrary spatial dimensions using the method
  of manufactured solutions.
\newblock {\em Journal of Computational Physics}, 448:110751, 2022.

\bibitem{scipy20}
P.~Virtanen, R.~Gommers, T.E. Oliphant, M.~Haberland, T.~Reddy, D.~Cournapeau,
  E.~Burovski, P.~Peterson, W.~Weckesser, J.~Bright, S.~J. {van der Walt},
  M.~Brett, J.~Wilson, K.~J. Millman, N.~Mayorov, A.R.J. Nelson, E.~Jones,
  R.~Kern, E.~Larson, C.~J. Carey, I.~Polat, Y.~Feng, E.W. Moore,
  J.~{VanderPlas}, D.~Laxalde, J.~Perktold, R.~Cimrman, I.~Henriksen, E.~A.
  Quintero, C.R. Harris, A.M. Archibald, A.~H. Ribeiro, F.~Pedregosa, P.~{van
  Mulbregt}, and {SciPy 1.0 Contributors}.
\newblock {{SciPy} 1.0: Fundamental Algorithms for Scientific Computing in
  Python}.
\newblock {\em Nature Methods}, 17:261--272, 2020.

\end{thebibliography}
\bibliographystyle{plain}

\end{document}